\documentclass[reqno,11pt]{amsart}

\usepackage{a4wide}
\usepackage{color}
\usepackage{longtable}
\usepackage{mathrsfs}
\usepackage{mathtools}
\usepackage{amsmath}
\usepackage{amssymb}
\usepackage{bbm}
\usepackage{tikz-cd}
\usepackage{esint}
\usepackage{nicefrac}
\numberwithin{equation}{section}
\usepackage[colorlinks,citecolor=green,linkcolor=red]{hyperref}

\usepackage[latin1]{inputenc}

\newcommand{\N}{\mathbb{N}}
\newcommand{\R}{\mathbb{R}}

\newcommand{\sfd}{{\sf d}}
\renewcommand{\d}{{\mathrm d}}
\newcommand{\e}{{\rm e}}
\newcommand{\X}{{\rm X}}
\newcommand{\Y}{{\rm Y}}

\newcommand{\1}{\mathbbm 1}

\newcommand{\LIP}{{\rm LIP}}

\newcommand{\lip}{{\rm lip}}

\newcommand{\nablaAM}{\nabla_{\!\scriptscriptstyle\rm AM}}

\newcommand{\Eucl}{{\rm Eucl}}

\newcommand{\Ch}{{\rm Ch}}
\newcommand{\ppi}{{\mbox{\boldmath\(\pi\)}}}
\newcommand{\sppi}{{\mbox{\scriptsize\boldmath\(\pi\)}}}

\newcommand{\EAM}{{\rm E}_{\scriptscriptstyle{\rm AM}}}
\newcommand{\GAM}{G_{\scriptscriptstyle{\rm AM}}}
\newcommand{\WAM}{W_{\scriptscriptstyle{\rm AM}}}
\newcommand{\limi}{\varliminf}
\newcommand{\lims}{\varlimsup}
\renewcommand{\div}{{\rm div}}

\newcommand{\fr}{\penalty-20\null\hfill\(\blacksquare\)}

\newtheorem{theorem}{Theorem}[section]
\newtheorem{corollary}[theorem]{Corollary}
\newtheorem{lemma}[theorem]{Lemma}
\newtheorem{proposition}[theorem]{Proposition}
\newtheorem{definition}[theorem]{Definition}
\newtheorem{example}[theorem]{Example}
\newtheorem{remark}[theorem]{Remark}

\title[Upper gradients on the weighted Euclidean space]{Characterisation
of upper gradients on the weighted Euclidean space and applications}
\author{Danka Lu\v{c}i\'{c}}
\address{\scriptsize Department of Mathematics and Statistics,
P.O.\ Box 35 (MaD), FI-40014 University of Jyvaskyla}
\email{\scriptsize danka.d.lucic@jyu.fi}
\author{Enrico Pasqualetto}
\address{\scriptsize Department of Mathematics and Statistics,
P.O.\ Box 35 (MaD), FI-40014 University of Jyvaskyla}
\email{\scriptsize enrico.e.pasqualetto@jyu.fi}
\author{Tapio Rajala}
\address{\scriptsize Department of Mathematics and Statistics,
P.O.\ Box 35 (MaD), FI-40014 University of Jyvaskyla}
\email{\scriptsize tapio.m.rajala@jyu.fi}
\begin{document}
\date{\today}
\allowdisplaybreaks
\keywords{Sobolev space, weighted Euclidean space,
decomposability bundle}
\subjclass[2010]{46E35, 53C23, 26B05}
\begin{abstract}
In the context of Euclidean spaces equipped with an arbitrary
Radon measure, we prove the equivalence among several different
notions of Sobolev space present in the literature and we
characterise the minimal weak upper gradient of all
Lipschitz functions.
\end{abstract}
\maketitle
\tableofcontents
\section*{Introduction}
In this paper we study first-order Sobolev spaces on
the Euclidean space \(\R^n\) equipped with an arbitrary
Radon measure \(\mu\geq 0\). This theory has been initiated
in the late nineties, with the pioneering work \cite{BBS97}
by G.\ Bouchitt\'{e}, G.\ Buttazzo, and P.\ Seppecher. The
motivations and applications were numerous, in the fields
of calculus of variations \cite{BBS97}, shape optimisation
\cite{BBS97-2}, optimal transport problems with gradient penalisation
\cite{Louet14}, amongst many others. Compared to Allard's theory of 
varifolds \cite{Allard72} or to Federer--Fleming's theory of
currents \cite{FF60}, the usage of measures in optimisation
problems presents two main advantages:
it allows to model objects made of parts having different Hausdorff
dimension (such as multijunctions), and it rests on a solid
functional-analytic machinery. About the former feature, we
just mention that the aim of \cite{BBS97} was to represent
low-dimensional elastic structures (such as membrane and beams)
in an intrinsic way, as opposed to the more classical idea
of first `fattening' the structure under consideration and
then passing to the limit in the vanishing thickness parameter
(via \(\Gamma\)-convergence methods, for instance).
With regard to the latter feature, let us briefly explain
which is the analytic framework the theory of Sobolev spaces
on weighted \(\R^n\) relies upon.

The key idea introduced by \cite{BBS97} was to define a
suitable `tangent distribution' associated with the measure \(\mu\),
namely, a \(\mu\)-a.e.\ defined measurable subbundle
of \(T\R^n\cong\R^n\times\R^n\); see Definition \ref{def:distr}.
In the approach adopted in \cite{BBS97}, the tangent fibers are
identified by looking at vector fields whose distributional
divergence belongs to \(L^2(\mu)\) (see \eqref{eq:def_distrib_div}
for the precise definition we are referring to). A different (but
similar in spirit) notion was studied by Fragal\`{a}--Mantegazza
\cite{FM99}; we do not investigate it in this paper. For a complete
account on this technique via the distributional divergence, we
refer to the survey \cite{BF02} and the references therein.
An alternative way to select the tangent fibers was proposed
by Zhikov in \cite{Zhi00,Zhi02}, where the strategy was to perform
a relaxation at the level of gradients of smooth functions.
We introduce a useful generalisation -- called
\(G\)-structure -- of Zhikov's concept in Definition
\ref{def:G-struct}. Later on, J.\ Louet studied in his
PhD thesis \cite{LouetPhD} the relation between the above two
approaches, but their complete equivalence was not known;
we will obtain it as a byproduct of Theorem \ref{thm:alt_char_T_mu}.
Once the tangent distribution is given, the Sobolev space is defined
by first projecting the gradients of smooth functions on the
tangent fibers (obtaining the tangential gradient with respect
to \(\mu\)) and then passing to the closure. The resulting
energy functional is lower semicontinuous, or equivalently the
associated notion of weak gradient yields a closed linear
operator. It is worth to recall that other geometric and
measure-theoretic notions of tangent space to a measure are
studied in the literature -- for instance, Preiss' notion of
`tangent measure' \cite{Preiss87} or Simon's notion of `approximate
tangent space' \cite{Simon84}. However, these are not the correct
objects to look at in order to define a Sobolev space: besides
the fact that they not always exist, a noteworthy problem is
that the consequent tangential gradient may well be not closable
(since the geometric fibers are typically bigger than
the analytic ones).
\medskip

In the present paper we start our investigation of the Sobolev
space on weighted \(\R^n\) from a rather different viewpoint.
More precisely, we regard it as a special case of the more
general theory of Sobolev spaces over a metric measure space
\((\X,\sfd,\mu)\). In this respect, the first definition was
given by P.\ Haj\l asz in \cite{Haj96}, but we will not consider
it here because of its `non-local' nature. At a later time,
several other notions (which eventually turned out to be equivalent)
have been proposed by J.\ Cheeger \cite{Cheeger00},
N.\ Shanmugalingam \cite{Shanmugalingam00}, L.\ Ambrosio, N.\ Gigli,
and G.\ Savar\'{e} \cite{AmbrosioGigliSavare11}, and
S.\ Di Marino \cite{DM14}. It will be convenient for us to
work with the approach \(W^{1,2}(\X,\mu)\) based on the concept
of test plan, introduced in \cite{AmbrosioGigliSavare11}; see
Definition \ref{def:Sob_space}. The common feature of all the
above approaches is the following: in lack of an underlying Banach
structure, the weakly differentiable functions \(f\) on a metric
measure space are detected by estimating the entity of their
variation, rather than the variation itself. In other words,
one obtains the `modulus of the weak differential' \(|D_\mu f|\)
instead of the weak differential \(D_\mu f\).

Let us focus our attention on the case in which \(\mu\) is a Radon
measure on \(\R^n\). Contrarily to what was discussed in the
first part of this introduction, we now have a Sobolev space
\(W^{1,2}(\R^n,\mu)\) at our disposal, but not (a priori)
a notion of tangent fiber. Still, a tangent
distribution can be recovered by appealing to results
available in the literature, as we are going to describe:
\begin{itemize}
\item N.\ Gigli built in \cite{Gigli14} an abstract tensor
calculus for metric measure spaces \((\X,\sfd,\mu)\),
which is based upon the notion of \(L^2(\mu)\)-normed
\(L^\infty(\mu)\)-module. In particular, the Sobolev space
gives rise to a natural notion of tangent module \(L^2_\mu(T\X)\),
whose elements should be regarded as the `synthetic' vector fields
over \((\X,\sfd,\mu)\). See Definition \ref{def:tg_mod}.
\item In the framework of the weighted Euclidean space,
N.\ Gigli and the second named author proved in \cite{GP16-2}
that the tangent module \(L^2_\mu(T\R^n)\) can be isometrically
embedded into the space \(L^2(\R^n,\R^n;\mu)\) of all
\(L^2(\mu)\)-maps from \(\R^n\) to itself. See Theorem
\ref{thm:iota_as_adjoint}.
\item The first and second named authors proved in
\cite{LP18} that (locally finitely-generated) \(L^2(\mu)\)-normed
\(L^\infty(\mu)\)-modules can be always represented as the spaces
of sections of a measurable Banach bundle. In the specific case
of weighted \(\R^n\), this grants that the tangent module
\(L^2_\mu(T\R^n)\) is canonically associated with a distribution
\(T_\mu\) in \(\R^n\), that we will call the tangent
distribution. See Definition \ref{def:tg_distr}. 
\end{itemize}
One of the main achievements of the present paper is Theorem
\ref{thm:alt_char_T_mu}, where we prove that the tangent
distribution \(T_\mu\) -- and accordingly the Sobolev space
\(W^{1,2}(\R^n,\mu)\) -- is consistent both with the notion
obtained via divergence by Bouchitt\'{e}--Buttazzo--Seppecher
\cite{BBS97} and with the one via vectorial relaxation
by Zhikov \cite{Zhi00,Zhi02}. Moreover, by building on top of this
equivalence result, we will identify the minimal object
\(|D_\mu f|\) (called the minimal weak upper gradient) associated
with any compactly-supported Lipschitz function \(f\) on
\(\R^n\); see Theorem \ref{thm:mwug_Lip}. The case \(n=1\) was
previously investigated by S.\ Di Marino and G.\ Speight in
\cite{DiMarinoSpeight15}.

In order to establish the above-mentioned characterisation
of the weak gradient of Lipschitz functions, we will need
to study the interaction between the Sobolev calculus on weighted
\(\R^n\) and the Alberti--Marchese differentiation theorem
\cite{AM16}, which says -- roughly speaking -- that there
exists a maximal distribution \(V_\mu\) in \(\R^n\) along
which all Lipschitz functions are \(\mu\)-a.e.\ differentiable
(in the sense of Fr\'{e}chet); cf.\ Theorem \ref{thm:AM}.
This kind of investigation has been initiated by the first
and second named authors together with S.\ Di Marino in
\cite{DMLP20}, where it is proven that the absolute value of
the Alberti--Marchese gradient is a weak upper gradient
(see Theorem \ref{thm:DMLP}). By using the machinery discussed
so far, we show (in Corollary \ref{cor:T_mu_in_W_mu}) that
\[
T_\mu(x)\subseteq V_\mu(x),\quad
\text{ for }\mu\text{-a.e.\ }x\in\R^n.
\]
However, in general `Sobolev calculus' and `Lipschitz calculus'
are not equivalent, thus one cannot expect the equality
\(T_\mu=V_\mu\) to hold for all measures \(\mu\).
Indeed, the Alberti--Marchese distribution just depends on the
negligible sets of \(\mu\), while the Sobolev space
\(W^{1,2}(\R^n,\mu)\) -- and thus, a fortiori, the tangent
distribution \(T_\mu\) -- strongly depends on the measure
\(\mu\) itself. An example of a measure \(\mu\) on \(\R\)
for which \(T_\mu\neq V_\mu\) will be described in Remark
\ref{rmk:T_mu_neq_V_mu}.

We are now in a position to state Theorem \ref{thm:mwug_Lip}:
given any \(f\in\LIP_c(\R^n)\), it holds that
\[
|D_\mu f|=\big|{\rm pr}_{T_\mu}(\nablaAM f)\big|,
\quad\text{ in the }\mu\text{-a.e.\ sense,}
\]
where we denote by \({\rm pr}_{T_\mu}\colon V_\mu\to T_\mu\)
the natural projection operator, while \(\nablaAM f\) stands
for the Alberti--Marchese gradient of \(f\) (that is a measurable
section of the distribution \(V_\mu\)).
\medskip

In the last part of the paper -- namely, in Section
\ref{s:applications} -- we shall provide a few applications
(for the moment, only at a theoretical level) of our main
Theorems \ref{thm:alt_char_T_mu} and \ref{thm:mwug_Lip}:
\begin{itemize}
\item \textsc{Section \ref{ss:tg_sing}}:
By combining our techniques with a deep result by G.\ De Philippis
and F.\ Rindler \cite{DPR16} about Radon measures on \(\R^n\),
we prove that for \(\mu^s\)-a.e.\ point \(x\in\R^n\) the
tangent fiber \(T_\mu(x)\) cannot coincide with the whole
\(\R^n\), where \(\mu^s\) stands for the singular part of
\(\mu\) with respect to the Lebesgue measure \(\mathcal L^n\);
see Theorem \ref{thm:dim_fibers}.
\item \textsc{Section \ref{ss:geom_T_mu}}: The tangent distribution
\(T_\mu\) admits a geometric interpretation, in terms of the
initial velocities of suitably chosen test plans
on \((\R^n,\sfd_{\rm Eucl},\mu)\); see Theorem
\ref{thm:fiber_cl_dot_pi}.
\item \textsc{Section \ref{ss:tens_Ch}}: Sobolev spaces
over the weighted Euclidean space satisfy the expected
tensorisation property; see Theorem \ref{thm:tens_Sob}.
\end{itemize}
We wish to point out that in the whole paper we just stick
to the case \(p=2\), but mostly for a matter of practicality.
The main reason is that many of the tools we will use --
those concerning the theory of normed modules -- are explicitly
written in the literature only for the case \(p=2\).
However, we expect that our results have appropriate
counterparts for every \(p\in(1,\infty)\).
\medskip

Finally, we conclude this introduction by mentioning that
also second-order Sobolev spaces on weighted Euclidean
spaces (for suitable Radon measures) have been studied,
\emph{e.g.}, in \cite{Bouchitte_Fragala03}.
It would be definitely interesting to understand whether
even these second-order spaces admit an
equivalent reformulation in the language of metric
measure spaces. Yet another interesting problem would be to
study the space \({\rm BV}(\R^n,\mu)\) of functions of bounded variation.
\bigskip

{\bf Acknowledgements.}
We would like to thank Simone Di Marino for the many useful
conversations about the results of this paper. All authors
are partially supported by the Academy of Finland, project 314789.
\section*{List of symbols}
\noindent
We provide below a list of the non-standard symbols
we shall use throughout the paper.
\begin{center}
\begin{longtable}{p{2cm} p{12cm}}
\(\mathcal L_1\) & Restriction of the Lebesgue measure
to the interval \([0,1]\). See \eqref{eq:def_L_1}. \\
\(\sfd_\infty\) & Supremum distance on \(C([0,1],\X)\).
See \eqref{eq:def_d_infty}. \\
\(\e_t\), \(\e_t^\X\) & Evaluation map at time \(t\).
See \eqref{eq:def_e_t}. \\
\(|\dot\gamma|\) & Metric speed of an absolutely continuous
curve \(\gamma\). See \eqref{eq:def_ms}. \\
\({\rm KE}_t\) & Kinetic energy functional at time \(t\).
See \eqref{eq:def_KE_t}. \\
\(\lip(f)\) & Local Lipschitz constant of a Lipschitz
function \(f\). See \eqref{eq:def_lip}. \\
\({\rm Comp}(\ppi)\) & Compression constant of a test plan
\(\ppi\). See Definition \ref{def:test_plan}. \\
\({\rm Const}^\X\) & `Constant curve' map.
See \eqref{eq:def_Const}. \\
\(W^{1,2}(\X,\mu)\) & Sobolev space on a metric measure
space \((\X,\sfd,\mu)\).
See Definition \ref{def:Sob_space}. \\
\(|D_\mu f|\) & Minimal weak upper gradient of
\(f\in W^{1,2}(\X,\mu)\). See Definition \ref{def:Sob_space}. \\
\({\rm E}_\Ch\) & Cheeger energy functional.
See Definition \ref{def:Cheeger_energy}. \\
\({\rm E}_\lip\) & `Lipschitz' energy functional.
See \eqref{eq:def_E_lip}. \\
\(\Delta_\mu\) & Laplacian operator. See \eqref{eq:def_Sob_Lapl}. \\
\(\{P_t\}_{t\geq 0}\) & Heat flow semigroup. See
\eqref{eq:heat_flow}. \\
\({\sf R}_{\mathscr M}\) & Riesz isomorphism associated
with a Hilbert module \(\mathscr M\). See \eqref{eq:def_Riesz}. \\
\(\mathscr N^\perp\) & Orthogonal complement of a submodule
\(\mathscr N\subseteq\mathscr M\).
See Remark \ref{rmk:orth_compl_I}. \\
\(L^2_\mu(T^*\X)\) & Abstract cotangent module on \((\X,\sfd,\mu)\).
See Definition \ref{def:tg_mod}. \\
\(\d_\mu f\) & Abstract differential of a function
\(f\in W^{1,2}(\X,\mu)\). See Definition \ref{def:tg_mod}. \\
\(L^2_\mu(T\X)\) & Abstract tangent module on \((\X,\sfd,\mu)\).
See Definition \ref{def:tg_mod}. \\
\(\nabla_\mu f\) & Abstract gradient of a function
\(f\in W^{1,2}(\X,\mu)\). See Definition \ref{def:tg_mod}. \\
\(\ppi'_t\) & Velocity at time \(t\) of a test plan \(\ppi\).
See Proposition \ref{prop:speed_test_plan}. \\
\(\div_\mu\) & Abstract divergence operator. See
\eqref{eq:def_div_mu}. \\
\({\rm Der}_t\) & `Derivation' map. See \eqref{eq:def_Der}. \\
\(\mathbb B_\sppi\) & The space
\(L^2\big(C([0,1],\R^n),\R^n;\ppi\big)\). See
\eqref{eq:def_B_pi}. \\
\({\rm P}_\mu\) & `Projection of \(1\)-forms' map.
See Theorem \ref{thm:iota_as_adjoint}. \\
\(\iota_\mu\) & `Embedding of vector fields' map.
See Theorem \ref{thm:iota_as_adjoint}.\\
\(\underline\div_\mu\) & Concrete divergence operator.
See \eqref{eq:def_distrib_div}. \\
\({\rm Gr}(\R^n)\) & Grassmannian of \(\R^n\). See the
beginning of Section \ref{ss:distributions}. \\
\(\mathscr D_n(\mu)\) & Space of distributions on \(\R^n\)
(up to \(\mu\)-a.e.\ equality). See Definition \ref{def:distr}. \\
\(\Gamma(V)\) & Space of \(L^2(\mu)\)-sections of a
distribution \(V\in\mathscr D_n(\mu)\). See Definition
\ref{def:distr}. \\
\({\rm pr}_V\) & Orthogonal projection map onto \(\Gamma(V)\).
See Remark \ref{rmk:orth_proj}. \\
\(V^\perp\) & Orthogonal complement of a distribution
\(V\in\mathscr D_n(\mu)\). See Remark \ref{rmk:orth_compl_II}. \\
\(V_\mu\) & Alberti--Marchese distribution. See Theorem
\ref{thm:AM}. \\
\(\nablaAM f\) & Alberti--Marchese gradient of \(f\in\LIP_c(\R^n)\).
See Theorem \ref{thm:AM}. \\
\(\EAM\) & Alberti--Marchese energy functional.
See \eqref{eq:def_E_AM}. \\
\(T_\mu\) & Tangent distribution. See
Definition \ref{def:tg_distr}. \\
\((\mathcal V,\bar\nabla)\) & An arbitrary \(G\)-structure.
See Definition \ref{def:G-struct}. \\
\(G_\mu\) & The \(G_\mu\)-structure
\(\big(C^\infty_c(\R^n),\nabla\big)\). See item a)
of Example \ref{ex:G-struct}. \\
\(\GAM\) & The \(\GAM\)-structure
\(\big(\LIP_c(\R^n),\nablaAM\big)\). See item b)
of Example \ref{ex:G-struct}. \\
\(G(f)\) & The family of \(G\)-gradients of \(f\).
See Definition \ref{def:G-gradient}. \\
\(W_G\) & The unique distribution satisfying
\(\Gamma(W_G)=G(0)\). See Definition \ref{def:W_G}. \\
\(W_\mu\) & The distribution \(W_{G_\mu}\). See
item iii) of Theorem \ref{thm:alt_char_T_mu}. \\
\(\mathcal I(\underline v)\) & `Currentification' of
a vector field \(\underline v\in D(\underline\div_\mu)\).
See Example \ref{ex:vector_fields_as_currents}. \\
\({\rm D}_\sppi\) & Initial velocity of a test plan \(\ppi\).
See Theorem \ref{thm:deriv_time_0_tp}.
\end{longtable}
\end{center}
\section{Preliminaries}
\subsection{Sobolev calculus on metric measure spaces}
For the purposes of the present paper, a \textbf{metric measure space}
is any triple \((\X,\sfd,\mu)\), where \((\X,\sfd)\) is a complete and
separable metric space, while \(\mu\geq 0\) is a
boundedly finite Borel measure on \((\X,\sfd)\).
We denote by \(\mathscr P(\X)\) the family of all
Borel probability measures on \((\X,\sfd)\).
\subsubsection{Absolutely continuous curves}
First of all, let us introduce the shorthand notation
\begin{equation}\label{eq:def_L_1}
\mathcal L_1\coloneqq\mathcal L^1|_{[0,1]},\quad
\text{ where }\mathcal L^1\text{ stands for the Lebesgue measure on }\R.
\end{equation}
We denote by \(C([0,1],\X)\) the family of all continuous curves
\(\gamma\colon[0,1]\to\X\). It holds that the set \(C([0,1],\X)\) is a complete
and separable metric space when endowed with the \textbf{supremum distance}
\(\sfd_\infty\), which is defined as
\begin{equation}\label{eq:def_d_infty}
\sfd_\infty(\gamma,\sigma)\coloneqq\max_{t\in[0,1]}\sfd(\gamma_t,\sigma_t),
\quad\text{ for every }\gamma,\sigma\in C([0,1],\X).
\end{equation}
Given any \(t\in[0,1]\), we denote by \(\e_t\colon C([0,1],\X)\to\X\)
the \textbf{evaluation map at time \(t\)}, \emph{i.e.},
\begin{equation}\label{eq:def_e_t}
\e_t(\gamma)=\e^\X_t(\gamma)\coloneqq
\gamma_t,\quad\text{ for every }\gamma\in C([0,1],\X).
\end{equation}
We say that \(\gamma\in C([0,1],\X)\) is \textbf{absolutely continuous}
if there exists \(g\in L^1(0,1)\) such that
\[
\sfd(\gamma_t,\gamma_s)\leq\int_s^t g(r)\,\d r,
\quad\text{ for every }s,t\in[0,1]\text{ such that }s<t.
\]
The minimal such function \(g\) (where minimality is intended in the
\(\mathcal L_1\)-a.e.\  sense) is called the \textbf{metric speed} of
\(\gamma\) and denoted by \(|\dot\gamma|\in L^1(0,1)\).
As proven in \cite[Theorem 1.1.2]{AmbrosioGigliSavare08}, it holds
\begin{equation}\label{eq:def_ms}
|\dot\gamma_t|=\lim_{h\to 0}\frac{\sfd(\gamma_{t+h},\gamma_t)}{|h|},
\quad\text{ for }\mathcal L_1\text{-a.e.\ }t\in[0,1].
\end{equation}
The family of absolutely continuous curves on \(\X\) is denoted
by \(AC([0,1],\X)\). Also, we define
\[
AC^2([0,1],\X)\coloneqq\Big\{\gamma\in AC([0,1],\X)
\;\Big|\;|\dot\gamma|\in L^2(0,1)\Big\}.
\]
It is well-known that \(AC^2([0,1],\X)\) is a Borel subset
of the metric space \(\big(C([0,1],\X),\sfd_\infty\big)\).
Given any \(t\in(0,1]\), we define the functional
\({\rm KE}_t\colon C([0,1],\X)\to[0,+\infty]\) as
\begin{equation}\label{eq:def_KE_t}
{\rm KE}_t(\gamma)\coloneqq\left\{\begin{array}{ll}
t\big(\fint_0^t|\dot\gamma_s|^2\,\d s\big)^{1/2},\\
+\infty
\end{array}\quad\begin{array}{ll}
\text{ if }\gamma\in AC^2([0,1],\X),\\
\text{ otherwise.}
\end{array}\right.
\end{equation}
Given a reflexive, separable Banach space \(\big(\mathbb B,\|\cdot\|\big)\)
and a curve \(\gamma\in AC([0,1],\mathbb B)\), it holds that
\(\gamma\) is \(\mathcal L_1\)-a.e.\ differentiable, its
\(\mathcal L_1\)-a.e.\ derivative \(\dot\gamma\colon[0,1]\to\mathbb B\)
is Bochner integrable, and
\[
\gamma_t-\gamma_s=\int_s^t\dot\gamma_r\,\d r,
\quad\text{ for every }s,t\in[0,1]\text{ such that }s<t.
\]
Observe that the identity \(\|\dot\gamma_t\|=|\dot\gamma_t|\)
is satisfied for \(\mathcal L_1\)-a.e.\ \(t\in[0,1]\).
\subsubsection{Lipschitz functions}
The family of all real-valued Lipschitz functions defined on \((\X,\sfd)\)
is indicated with \(\LIP(\X)\). The subfamily of those
Lipschitz functions having compact support (resp.\ bounded support)
is denoted by \(\LIP_c(\X)\) (resp.\ \(\LIP_{bs}(\X)\)).
Given any \(f\in\LIP(\X)\), we define its \textbf{local Lipschitz constant} as
\begin{equation}\label{eq:def_lip}
\lip(f)(x)\coloneqq\lims_{y\to x}\frac{\big|f(x)-f(y)\big|}{\sfd(x,y)},
\quad\text{ whenever }x\in\X\text{ is an accumulation point,}
\end{equation}
and \(\lip(f)(x)\coloneqq 0\) elsewhere.
\subsubsection{Sobolev space via test plans}
We recall here the definition of Sobolev space in the metric measure
setting and its main properties. The approach we are going to describe
has been proposed in \cite{AmbrosioGigliSavare11,AmbrosioGigliSavare11-3}.
To begin with, let us recall the important notion of test plan:
\begin{definition}[Test plan
\cite{AmbrosioGigliSavare11,AmbrosioGigliSavare11-3}]
\label{def:test_plan}
Let \((\X,\sfd,\mu)\) be a metric measure space. Then we say that a Borel
probability measure \(\ppi\) on \(\big(C([0,1],\X),\sfd_\infty\big)\) is
a \textbf{test plan} on \((\X,\sfd,\mu)\) provided the following properties
are satisfied:
\begin{itemize}
\item[\(\rm i)\)] There exists a \textbf{compression constant}
\({\rm Comp}(\ppi)>0\) such that
\[
(\e_t)_*\ppi\leq{\rm Comp}(\ppi)\mu,\quad\text{ for every }t\in[0,1],
\]
where \((\e_t)_*\ppi\) stands for the pushforward measure of \(\ppi\)
under the evaluation map \(\e_t\).
\item[\(\rm ii)\)] The measure \(\ppi\) is concentrated on
\(AC^2([0,1],\X)\) and \textbf{has finite kinetic energy}, \emph{i.e.},
\[
\int{\rm KE}_1(\gamma)^2\,\d\ppi(\gamma)=
\int\!\!\!\int_0^1|\dot\gamma_t|^2\,\d t\,\d\ppi(\gamma)<+\infty.
\]
\end{itemize}
\end{definition}
\begin{example}\label{ex:Const}{\rm
Given a metric measure space \((\X,\sfd,\mu)\), we set
\({\rm Const}^\X\colon\X\to C([0,1],\X)\) as
\begin{equation}\label{eq:def_Const}
{\rm Const}^\X(x)_t\coloneqq x,\quad
\text{ for every }x\in\X\text{ and }t\in[0,1].
\end{equation}
Then the map \({\rm Const}^\X\) is an isometry and the
measure \(\ppi\coloneqq{\rm Const}^\X_*\nu\) is a test plan
on \((\X,\sfd,\mu)\) for every \(\nu\in\mathscr P(\X)\)
satisfying \(\nu\leq C\mu\) for some constant \(C>0\).
\fr}\end{example}
The notion of test plan plays an essential role in the definition
of Sobolev space:
\begin{definition}[Sobolev space via test plans \cite{AmbrosioGigliSavare11,AmbrosioGigliSavare11-3}]\label{def:Sob_space}
Let \((\X,\sfd,\mu)\) be a metric measure space. Fix \(f\in L^2(\mu)\).
Then a function \(G\in L^2(\mu)\) is said to be a
\textbf{weak upper gradient} of \(f\) provided for any test plan \(\ppi\)
on \((\X,\sfd,\mu)\) the following property is satisfied: for
\(\ppi\)-a.e.\ \(\gamma\) it holds that \(f\circ\gamma\in W^{1,1}(0,1)\) and
\[
\big|(f\circ\gamma)'_t\big|\leq G(\gamma_t)\,|\dot\gamma_t|,
\quad\text{ for }\mathcal L_1\text{-a.e.\ }t\in[0,1].
\]
We define the \textbf{Sobolev space} \(W^{1,2}(\X,\mu)\) as the family of
all those functions \(f\in L^2(\mu)\) that admit a weak upper gradient.
Given any \(f\in W^{1,2}(\X,\mu)\), we denote by \(|D_\mu f|\) the minimal
weak upper gradient of \(f\), where minimality is intended in the
\(\mu\)-a.e.\ sense.
\end{definition}
The original notion of Sobolev space \(W^{1,2}(\X,\mu)\) via
test plans has been introduced in \cite{AmbrosioGigliSavare11}, but
its equivalent reformulation we presented above has been established in
\cite[Appendix B]{Gigli12}. We chose the unusual notation
\(W^{1,2}(\X,\mu)\), where the distance \(\sfd\) does not appear
(even though it plays a role in the definition), for a matter of
practicality, since in all the cases we shall consider, the distance --
differently from the measure -- will always remain fixed.
\medskip

Given any function \(f\in\LIP_{bs}(\X)\),
it holds that \(f\in W^{1,2}(\X,\mu)\) and
\begin{equation}\label{eq:mwug_Lip}
|D_\mu f|\leq\lip(f),\quad\mu\text{-a.e.\ on }\X.
\end{equation}
The equality in \eqref{eq:mwug_Lip} is achieved only in
particular circumstances; see, \emph{e.g.},
Corollary \ref{cor:equiv_Df=lipf} and Remark
\ref{rmk:PI_case}.
\subsubsection{Energy functionals} Throughout the whole paper,
we will consider several different energy functionals
\({\rm E}\colon L^2(\mu)\to[0,+\infty]\) over a given metric measure
space \((\X,\sfd,\mu)\). Let us fix some notation.
The \textbf{finiteness domain} of \(\rm E\) is given by
\(D({\rm E})\coloneqq\big\{f\in L^2(\mu)\,:\,{\rm E}(f)<+\infty\big\}\).
We say that \(\rm E\) is \textbf{\(2\)-homogeneous} provided
\({\rm E}(\lambda f)=\lambda^2\,{\rm E}(f)\) for every
\(f\in D({\rm E})\) and \(\lambda\in\R\), while it
is \textbf{convex} provided
\({\rm E}\big(\lambda f+(1-\lambda)g\big)\leq\lambda\,{\rm E}(f)
+(1-\lambda)\,{\rm E}(g)\) for every \(f,g\in L^2(\mu)\)
and \(\lambda\in[0,1]\). The functional \(\rm E\) is said to satisfy
the \textbf{parallelogram rule} if it holds that
\[
{\rm E}(f+g)+{\rm E}(f-g)=2\,{\rm E}(f)+2\,{\rm E}(g),
\quad\text{ for every }f,g\in D({\rm E}).
\]
Moreover, we say that the functional \(\rm E\) is
\textbf{lower semicontinuous} provided
\[
{\rm E}(f)\leq\limi_{n\to\infty}{\rm E}(f_n),
\quad\text{ for every }f,f_n\in L^2(\mu)
\text{ such that }f_n\to f\text{ in }L^2(\mu).
\]
The \textbf{lower semicontinuous envelope}
\(\tilde{\rm E}\colon L^2(\mu)\to[0,+\infty]\) of \(\rm E\)
is defined as
\[
\tilde{\rm E}(f)\coloneqq\inf\limi_{n\to\infty}{\rm E}(f_n),
\quad\text{ for every }f\in L^2(\mu),
\]
where the infimum is taken among all sequences \((f_n)_n\subseteq L^2(\mu)\)
such that \(f_n\to f\) in \(L^2(\mu)\). It holds that \(\tilde{\rm E}\)
is the greatest lower semicontinuous functional which is dominated by
\(\rm E\).
\medskip

The most important energy functional we will consider is the
so-called Cheeger energy:
\begin{definition}[Cheeger energy]\label{def:Cheeger_energy}
Let \((\X,\sfd,\mu)\) be a metric measure space. Then we define
\begin{equation}\label{eq:def_E_Ch}
{\rm E}_{\rm Ch}(f)\coloneqq\left\{\begin{array}{ll}
\frac{1}{2}\int|D_\mu f|^2\,\d\mu,\\
+\infty,
\end{array}\quad\begin{array}{ll}
\text{ if }f\in W^{1,2}(\X,\mu),\\
\text{ otherwise.}
\end{array}\right.
\end{equation}
The functional \({\rm E}_{\rm Ch}\colon L^2(\mu)\to[0,+\infty]\)
is called the \textbf{Cheeger energy} associated with \((\X,\sfd,\mu)\).
\end{definition}

The map \({\rm E}_{\rm Ch}\) is convex, \(2\)-homogeneous,
and lower semicontinuous. Also, \(f\mapsto\sqrt{2\,{\rm E}_{\rm Ch}(f)}\)
is a seminorm on \(D({\rm E}_{\rm Ch})=W^{1,2}(\X,\mu)\). In particular,
\(W^{1,2}(\X,\mu)\) is a Banach space if endowed with the following norm:
\[
\|f\|_{W^{1,2}(\X,\mu)}\coloneqq
\Big(\|f\|_{L^2(\mu)}^2+2\,{\rm E}_{\rm Ch}(f)\Big)^{1/2},
\quad\text{ for every }f\in W^{1,2}(\X,\mu).
\]
Another energy functional to take into account is the following one:
\begin{equation}\label{eq:def_E_lip}
{\rm E}_\lip(f)\coloneqq\left\{\begin{array}{ll}
\frac{1}{2}\int\lip^2(f)\,\d\mu,\\
+\infty,
\end{array}\quad\begin{array}{ll}
\text{ if }f\in\LIP_{bs}(\X),\\
\text{ otherwise.}
\end{array}\right.
\end{equation}
In view of \eqref{eq:mwug_Lip}, we know that
\({\rm E}_{\rm Ch}\leq{\rm E}_\lip\). Actually, \({\rm E}_{\rm Ch}\)
is the lower semicontinuous envelope of \({\rm E}_\lip\), as granted
by the following important result.
\begin{theorem}[Density in energy of Lipschitz functions
\cite{AmbrosioGigliSavare11-3}]\label{thm:density_Lip}
Let \((\X,\sfd,\mu)\) be a metric measure space. Let \(f\in W^{1,2}(\X,\mu)\)
be given. Then there exists \((f_n)_n\subseteq\LIP_{bs}(\X)\)
such that \(f_n\to f\) and \(\lip(f_n)\to|D_\mu f|\) in \(L^2(\mu)\).
\end{theorem}
\subsubsection{Infinitesimal Hilbertianity}
The following definition has been introduced in \cite{Gigli12}:
\begin{definition}[Infinitesimal Hilbertianity]
A metric measure space \((\X,\sfd,\mu)\) is said to be
\textbf{infinitesimally Hilbertian} provided the Sobolev space
\(W^{1,2}(\X,\mu)\) is Hilbert. Equivalently, if the Cheeger
energy \({\rm E}_{\rm Ch}\) satisfies the parallelogram rule.
\end{definition}
Given an infinitesimally Hilbertian space \((\X,\sfd,\mu)\),
it holds that the mapping
\[
\langle\nabla_\mu f,\nabla_\mu g\rangle\coloneqq
\frac{\big|D_\mu(f+g)\big|^2-|D_\mu f|^2-|D_\mu g|^2}{2},
\quad\mu\text{-a.e.\ on }\X,
\]
defines a symmetric, bilinear form on \(W^{1,2}(\X,\mu)\times W^{1,2}(\X,\mu)\)
with values in \(L^1(\mu)\).
\subsubsection{Laplacian and heat flow}
Let \((\X,\sfd,\mu)\) be an infinitesimally Hilbertian space.
Given any function \(f\in W^{1,2}(\X,\mu)\), we declare that
\(f\in D(\Delta_\mu)\) if there exists \(h\in L^2(\mu)\) such that
\begin{equation}\label{eq:def_Sob_Lapl}
\int gh\,\d\mu=-\int\langle\nabla_\mu f,\nabla_\mu g\rangle\,\d\mu,
\quad\text{ for every }g\in W^{1,2}(\X,\mu).
\end{equation}
Since \(h\) is uniquely determined, we denote it by \(\Delta_\mu f\)
and call it the \textbf{Laplacian} of \(f\). It holds that \(D(\Delta_\mu)\)
is a linear subspace of \(W^{1,2}(\X,\mu)\) and
\(\Delta_\mu\colon D(\Delta_\mu)\to L^2(\mu)\) is a linear operator.
\medskip

The \textbf{heat flow} \(\{P_t\}_{t\geq 0}\) on \((\X,\sfd,\mu)\)
is defined as follows: for any given function \(f\in L^2(\mu)\),
we have that \([0,+\infty)\ni t\mapsto P_t f\in L^2(\mu)\) is the
unique continuous curve satisfying \(P_0 f=f\), which is absolutely
continuous on \((0,+\infty)\), such that \(P_t f\in D(\Delta_\mu)\)
holds for all \(t>0\) and
\begin{equation}\label{eq:heat_flow}
\frac{\d}{\d t}P_t f=\Delta_\mu P_t f,
\quad\text{ for }\mathcal L^1\text{-a.e.\ }t>0.
\end{equation}
Given any function \(f\in W^{1,2}(\X,\mu)\), it holds that
\begin{equation}\label{eq:heat_flow_contract}
\|P_t f\|_{W^{1,2}(\X,\mu)}\leq\|f\|_{W^{1,2}(\X,\mu)},
\quad\text{ for every }t>0.
\end{equation}
The above properties are ensured by the classical Komura--Brezis
theory of gradient flows.
\subsection{Differential structure of metric measure spaces}
\label{ss:diff_struct_mms}
A first-order differential calculus on metric measure spaces
has been developed in \cite{Gigli14,Gigli17}.
Let us briefly recall the key concepts.
\subsubsection{The theory of normed modules}
Let \((\X,\sfd,\mu)\) be a given metric measure space. Let \(\mathscr M\)
be an algebraic module over the commutative ring \(L^\infty(\mu)\).
Then a \textbf{pointwise norm} on \(\mathscr M\) is a mapping
\(|\cdot|\colon\mathscr M\to L^2(\mu)\) satisfying the following properties:
\[\begin{split}
|v|\geq 0,&\quad\text{ for every }v\in\mathscr M,
\text{ with equality if and only if }v=0,\\
|v+w|\leq|v|+|w|,&\quad\text{ for every }v,w\in\mathscr M,\\
|fv|=|f||v|,&\quad\text{ for every }f\in L^\infty(\mu)
\text{ and }v\in\mathscr M.
\end{split}\]
(All inequalities are intended in the \(\mu\)-a.e.\ sense.)
We say that \(\big(\mathscr M,|\cdot|\big)\), or just
\(\mathscr M\), is an \textbf{\(L^2(\mu)\)-normed \(L^\infty(\mu)\)-module}
provided the norm \(\|v\|_{\mathscr M}\coloneqq\big\||v|\big\|_{L^2(\mu)}\)
on \(\mathscr M\) is complete.
\medskip

By a \textbf{morphism} \(\varphi\colon\mathscr M\to\mathscr N\)
between two given \(L^2(\mu)\)-normed \(L^\infty(\mu)\)-modules
\(\mathscr M,\mathscr N\) we mean an \(L^\infty(\mu)\)-linear
and continuous map. The \textbf{dual module} \(\mathscr M^*\) of
\(\mathscr M\) is defined as the space of all \(L^\infty(\mu)\)-linear
and continuous maps from \(\mathscr M\) to \(L^1(\mu)\). It holds
that \(\mathscr M^*\) has a natural \(L^2(\mu)\)-normed \(L^\infty(\mu)\)-module structure, the pointwise
norm \(|L|\) of \(L\in\mathscr M^*\) being defined as the minimal
function \(G\in L^2(\mu)\), where minimality is intended in the
\(\mu\)-a.e.\ sense, such that the inequality \(\big|L(v)\big|\leq G|v|\)
is satisfied \(\mu\)-a.e.\ on \(\X\)
for every element \(v\in\mathscr M\).
\medskip

By a \textbf{Hilbert module} on \((\X,\sfd,\mu)\) we mean an \(L^2(\mu)\)-normed
\(L^\infty(\mu)\)-module \(\mathscr M\) such that
\[
|v+w|^2+|v-w|^2=2\,|v|^2+2\,|w|^2\;\;\;\mu\text{-a.e.},
\quad\text{ for every }v,w\in\mathscr M.
\]
Clearly, \(\mathscr M\) is a Hilbert module if and only if it is Hilbert
when viewed as a Banach space. Given two elements \(v,w\in\mathscr M\),
we define their \textbf{pointwise scalar product}
\(\langle v,w\rangle\in L^1(\mu)\) as
\[
\langle v,w\rangle\coloneqq\frac{|v+w|^2-|v|^2-|w|^2}{2},
\quad\text{ in the }\mu\text{-a.e.\ sense.}
\]
The resulting mapping
\(\langle\cdot,\cdot\rangle\colon\mathscr M\times\mathscr M\to L^1(\mu)\)
is \(L^\infty(\mu)\)-bilinear and symmetric. It holds that the morphism
\({\sf R}_{\mathscr M}\colon\mathscr M\to\mathscr M^*\) of \(L^2(\mu)\)-normed
\(L^\infty(\mu)\)-modules defined as
\begin{equation}\label{eq:def_Riesz}
{\sf R}_{\mathscr M}(v)(w)\coloneqq\langle v,w\rangle\in L^1(\mu),
\quad\text{ for every }v,w\in\mathscr M,
\end{equation}
is an isometric isomorphism. We call \({\sf R}_{\mathscr M}\) the
\textbf{Riesz isomorphism} associated with \(\mathscr M\).
\begin{remark}[Orthogonal complement, I]\label{rmk:orth_compl_I}{\rm
Let \(\mathscr M\) be a Hilbert module on \((\X,\sfd,\mu)\). Then we define
the \textbf{orthogonal complement} of a given submodule
\(\mathscr N\subseteq\mathscr M\) as
\[
\mathscr N^\perp\coloneqq\Big\{v\in\mathscr M\;\Big|\;\langle v,w\rangle=0
\text{ in the }\mu\text{-a.e.\ sense, for every }w\in\mathscr N\Big\}.
\]
Then \(\mathscr N^\perp\) is a submodule of \(\mathscr M\)
that satisfies \(\mathscr N\cap\mathscr N^\perp=\{0\}\) and
\(\mathscr N+\mathscr N^\perp=\mathscr M\).
\fr}\end{remark}
\begin{definition}[Dimension of a normed module
{\cite[Section 1.4]{Gigli14}}]
Let \((\X,\sfd,\mu)\) be a metric measure space. Let \(\mathscr M\) be an
\(L^2(\mu)\)-normed \(L^\infty(\mu)\)-module and let \(E\subseteq\X\) be a Borel
set such that \(\mu(E)>0\). Then:
\begin{itemize}
\item[\(\rm i)\)] We say that some elements \(v_1,\ldots,v_n\in\mathscr M\)
are \textbf{independent} on the set \(E\) provided the mapping
\(L^\infty(\mu|_E)^n\ni(f_1,\ldots,f_n)\mapsto\sum_{i=1}^n f_i\,v_i\in\mathscr M\)
is injective.
\item[\(\rm ii)\)] A family \(\mathscr F\subseteq\mathscr M\) is said to
\textbf{generate} \(\mathscr M\) on \(E\) provided the linear space \(\mathcal V\), given by
\[
\mathcal V\coloneqq\bigg\{\sum_{i=1}^n f_i\,v_i\;\bigg|\;
n\in\N,\,(f_i)_{i=1}^n\subseteq L^\infty(\mu|_E),\,
(v_i)_{i=1}^n\subseteq\mathscr F\bigg\}\subseteq\mathscr M,
\]
is dense in the restricted module
\(\mathscr M|_E\coloneqq\big\{\mathbbm 1_E\,v:v\in\mathscr M\big\}\).
\end{itemize}
We say that \(\mathscr M\) has \textbf{dimension} \(n\in\N\) on \(E\) provided
it admits a \textbf{local basis} \(v_1,\ldots,v_n\in\mathscr M\) on \(E\),
\emph{i.e.}, the elements \(v_1,\ldots,v_n\) are independent on \(E\) and
\(\{v_1,\ldots,v_n\}\) generates \(\mathscr M\) on \(E\).
\end{definition}
Let \((\X,\sfd_\X,\mu)\), \((\Y,\sfd_\Y,\nu)\) be metric measure spaces.
Let \(\varphi\colon\X\to\Y\) be a given Borel map. Then we say that \(\varphi\) is a
\textbf{map of bounded compression} provided \(\varphi_*\mu\leq C\nu\) for some \(C>0\).
\begin{theorem}[Pullback module {\cite[Section 1.4.1]{Gigli17}}]
\label{thm:pullback}
Let \((\X,\sfd_\X,\mu)\), \((\Y,\sfd_\Y,\nu)\) be two metric measure spaces.
Let \(\mathscr M\) be an \(L^2(\nu)\)-normed \(L^\infty(\nu)\)-module and
\(\varphi\colon\X\to\Y\) a map of bounded compression. Then there exists a
unique couple \((\varphi^*\mathscr M,\varphi^*)\), where
\(\varphi^*\mathscr M\) is an \(L^2(\mu)\)-normed \(L^\infty(\mu)\)-module
called the \textbf{pullback module} and \(\varphi^*\colon\mathscr M\to\varphi^*\mathscr M\)
is a linear operator called the \textbf{pullback map}, such that
\(|\varphi^*v|=|v|\circ\varphi\) holds \(\mu\)-a.e.\ for all \(v\in\mathscr M\)
and \(\{\varphi^*v:v\in\mathscr M\}\) generates \(\varphi^*\mathscr M\) on \(\X\).

Moreover, given two \(L^2(\nu)\)-normed \(L^\infty(\nu)\)-modules \(\mathscr M\),
\(\mathscr N\) and a morphism \(\Phi\colon\mathscr M\to\mathscr N\), there is
a unique morphism \(\varphi^*\Phi\colon\varphi^*\mathscr M\to\varphi^*\mathscr N\)
of \(L^2(\mu)\)-normed \(L^\infty(\mu)\)-modules such that
\[\begin{tikzcd}
\mathscr M \arrow[r,"\Phi"] \arrow[d,swap,"\varphi^*"]
& \mathscr N \arrow[d,"\varphi^*"] \\
\varphi^*\mathscr M \arrow[r,swap,"\varphi^*\Phi"]
& \varphi^*\mathscr N
\end{tikzcd}\]
is a commutative diagram.
\end{theorem}
\subsubsection{Abstract \(1\)-forms and vector fields}
The language of normed modules discussed in the previous section
can be used to provide abstract notions of \(1\)-forms and vector
fields -- tightly linked to the Sobolev calculus -- on general
metric measure spaces:
\begin{theorem}[Cotangent and tangent modules
{\cite[Sections 1.2.2 and 1.3.2]{Gigli17}}]\label{def:tg_mod}
Let \((\X,\sfd,\mu)\) be a metric measure space.
Then there exists a unique couple \(\big(L^2_\mu(T^*\X),\d_\mu\big)\),
where the \textbf{cotangent module} \(L^2_\mu(T^*\X)\) is an
\(L^2(\mu)\)-normed \(L^\infty(\mu)\)-module and the \textbf{differential}
\[
\d_\mu\colon W^{1,2}(\X,\mu)\longrightarrow L^2_\mu(T^*\X)
\]
is a linear operator, such that the following properties are satisfied:
\[\begin{split}
|\d_\mu f|=|D_\mu f|\;\;\;\mu\text{-a.e.,}&\quad\text{ for every }
f\in W^{1,2}(\X,\mu),\\
\big\{\d_\mu f\;\big|\;f\in W^{1,2}(\X,\mu)\big\}&
\quad\text{ generates }L^2_\mu(T^*\X)\text{ on }\X.
\end{split}\]
Moreover, if \((\X,\sfd,\mu)\) is infinitesimally Hilbertian, then
\(L^2_\mu(T^*\X)\) is a Hilbert module and the \textbf{tangent module}
is defined as \(L^2_\mu(T\X)\coloneqq L^2_\mu(T^*\X)^*\).
The \textbf{gradient} \(\nabla_\mu f\in L^2_\mu(T\X)\) of a function
\(f\in W^{1,2}(\X,\mu)\) is given by the image of \(\d_\mu f\)
under the Riesz isomorphism \({\sf R}_{L^2_\mu(T^*\X)}\).
\end{theorem}
It holds that a given metric measure space \((\X,\sfd,\mu)\)
is infinitesimally Hilbertian if and only if its associated
modules \(L^2_\mu(T^*\X)\) and \(L^2_\mu(T\X)\) are Hilbert.
\begin{proposition}[Closure of the differential
{\cite[Theorem 2.2.9]{Gigli14}}]\label{prop:closure_diff}
Let \((\X,\sfd,\mu)\) be a metric measure space.
Let \((f_n)_n\subseteq W^{1,2}(\X,\mu)\) satisfy
\(f_n\rightharpoonup f\) weakly in \(L^2(\mu)\) for
some \(f\in L^2(\mu)\) and \(\d_\mu f_n\rightharpoonup\omega\)
weakly in \(L^2_\mu(T^*\X)\) for some \(\omega\in L^2_\mu(T^*\X)\).
Then \(f\in W^{1,2}(\X,\mu)\) and \(\d_\mu f=\omega\).
\end{proposition}
Given a test plan \(\ppi\) on a metric measure space \((\X,\sfd,\mu)\),
it holds that for every \(t\in[0,1]\) the evaluation map \(\e_t\)
is a map of bounded compression between \(\big(C([0,1],\X),\ppi\big)\)
and \((\X,\mu)\). This allows us to consider the pullback
modules \(\e_t^*L^2_\mu(T^*\X)\) and \(\e_t^*L^2_\mu(T\X)\).
\begin{proposition}[Velocity of a test plan
{\cite[Theorem 2.3.18]{Gigli14}}]\label{prop:speed_test_plan}
Let \((\X,\sfd,\mu)\) be a metric measure space such that
the module \(L^2_\mu(T\X)\) is separable. Let \(\ppi\) be a test plan on
\((\X,\sfd,\mu)\). Then for \(\mathcal L_1\)-a.e.\ \(t\in[0,1]\)
there exists a unique element \(\ppi'_t\in\e_t^*L^2_\mu(T\X)\),
called the \textbf{velocity} of \(\ppi\) at \(t\), such that
\begin{equation}\label{eq:formula_speed_test_plan}
\lim_{h\to 0}\bigg\|\frac{f\circ\e_{t+h}-f\circ\e_t}{h}
-(\e_t^*\d_\mu f)(\ppi'_t)\bigg\|_{L^1(\sppi)}=0,
\quad\text{ for every }f\in W^{1,2}(\X,\mu).
\end{equation}
Moreover, it holds that \(|\ppi'_t|(\gamma)=|\dot\gamma_t|\)
for \((\ppi\otimes\mathcal L_1)\)-a.e.\ \((\gamma,t)\in
AC^2([0,1],\X)\times[0,1]\).
\end{proposition}
\subsubsection{Divergence of abstract vector fields}
Let \((\X,\sfd,\mu)\) be an infinitesimally Hilbertian space.
We declare that \(v\in L^2_\mu(T\X)\) belongs to \(D(\div_\mu)\)
provided there exists \(h\in L^2(\mu)\) such that
\begin{equation}\label{eq:def_div_mu}
\int\d_\mu f(v)\,\d\mu=-\int fh\,\d\mu,
\quad\text{ for every }f\in W^{1,2}(\X,\mu).
\end{equation}
The uniquely determined function \(h\) will be denoted by \(\div_\mu(v)\)
and called the \textbf{abstract divergence} of \(v\).
It can be readily checked that a function \(f\in W^{1,2}(\X,\mu)\)
belongs to \(D(\Delta_\mu)\) if and only if \(\nabla_\mu f\in D(\div_\mu)\).
In this case, it also holds that \(\div_\mu(\nabla_\mu f)=\Delta_\mu f\).
\medskip

Let \(f\in\LIP_{bs}(\X)\) and \(v\in D(\div_\mu)\) be given.
Then it holds that \(fv\in D(\div_\mu)\) and
\begin{equation}\label{eq:Leibniz_div}
\div_\mu(fv)=f\,\div_\mu(v)+\langle\nabla_\mu f,v\rangle,
\quad\mu\text{-a.e.\ on }\X.
\end{equation}
In other words, we say that the abstract divergence satisfies
the \textbf{Leibniz rule}.
\begin{lemma}[Density of vector fields with divergence]
\label{lem:density_D_Delta}
Let \((\X,\sfd,\mu)\) be an infinitesimally Hilbertian space.
Then \(D(\Delta_\mu)\) is dense in \(W^{1,2}(\X,\mu)\)
and \(D(\div_\mu)\) is dense in \(L^2_\mu(T\X)\).
\end{lemma}
\begin{proof}
First of all, fix \(f\in W^{1,2}(\X,\mu)\) and consider
\(P_t f\in D(\Delta_\mu)\) for every \(t>0\). Since the family
\(\{P_t f\}_{t>0}\subseteq W^{1,2}(\X,\mu)\) is bounded
by \eqref{eq:heat_flow_contract} and \(W^{1,2}(\X,\mu)\)
is reflexive, there exists a sequence \(t_n\searrow 0\) such
that \(P_{t_n}f\rightharpoonup f\) weakly in \(W^{1,2}(\X,\mu)\).
By Banach--Saks theorem, we have that (possibly passing to a not
relabelled subsequence) the sequence \((f_n)_n\subseteq D(\Delta_\mu)\)
given by \(f_n\coloneqq\frac{1}{n}\sum_{i=1}^n P_{t_i}f\) satisfies
\(f_n\to f\) with respect to the strong topology of \(W^{1,2}(\X,\mu)\).

In order to prove the last part of the statement, fix \(v\in L^2_\mu(T\X)\)
and \(\varepsilon>0\). We can find functions
\(f'_1,\ldots,f'_n\in W^{1,2}(\X,\mu)\)
and \(g'_1,\ldots,g'_n\in L^\infty(\mu)\) with
\(\big\|v-\sum_{i=1}^n g'_i\nabla_\mu f'_i\big\|_{L^2_\mu(T\X)}
<\frac{\varepsilon}{2}\). Thanks to the first part of the statement and
the fact that boundedly-supported Lipschitz functions are
weakly\(^*\) dense in \(L^\infty(\mu)\), there are
\(f_1,\ldots,f_n\in D(\Delta_\mu)\) and \(g_1,\ldots,g_n\in\LIP_{bs}(\X)\)
such that \(\big\|\sum_{i=1}^n(g'_i\nabla_\mu f'_i-
g_i\nabla_\mu f_i)\big\|_{L^2_\mu(T\X)}<\frac{\varepsilon}{2}\).
Consequently, we conclude that the vector field
\(w\coloneqq\sum_{i=1}^n g_i\nabla_\mu f_i\),
which belongs to \(D(\div_\mu)\) by \eqref{eq:Leibniz_div},
satisfies \(\|v-w\|_{L^2_\mu(T\X)}<\varepsilon\).
\end{proof}
\subsubsection{Concrete \(1\)-forms and vector fields on weighted \(\R^n\)}
Let \((\X,\sfd,\mu)\) be a metric measure space and
\(\big(\mathbb B,\|\cdot\|\big)\) a separable Banach space.
Then we denote by \(L^2(\X,\mathbb B;\mu)\) the family of all
Borel maps \(v\colon\X\to\mathbb B\) such that
\(\int\big\|v(x)\big\|^2\,\d\mu(x)<+\infty\),
considered up to \(\mu\)-a.e.\ equality. It holds that
\(L^2(\X,\mathbb B;\mu)\) is an \(L^2(\mu)\)-normed \(L^\infty(\mu)\)-module
when endowed with the natural pointwise operations and the following
pointwise norm: given any \(v\in L^2(\X,\mathbb B;\mu)\), we define
\[
|v|(x)\coloneqq\big\|v(x)\big\|,\quad\text{ for }\mu\text{-a.e.\ }x\in\X.
\]
Moreover, it holds (assuming \(\mu\neq 0\))
that \(L^2(\X,\mathbb B;\mu)\) is Hilbert if
and only if \(\mathbb B\) is Hilbert.
\medskip

We denote by \(\sfd_{\rm Eucl}\) the Euclidean distance
\(\sfd_{\rm Eucl}(x,y)\coloneqq|x-y|\) on \(\R^n\).
Given any \(t\in[0,1]\), we define the mapping
\({\rm Der}_t\colon C([0,1],\R^n)\to\R^n\) as
\begin{equation}\label{eq:def_Der}
{\rm Der}_t(\gamma)\coloneqq\left\{\begin{array}{ll}
\dot\gamma_t,\\
0,
\end{array}\quad\begin{array}{ll}
\text{ if }\dot\gamma_t=\lim_{h\to 0}\frac{\gamma_{t+h}-\gamma_t}{h}
\text{ exists,}\\
\text{ otherwise.}
\end{array}\right.
\end{equation}
Standard arguments show that \({\rm Der}_t\) is Borel.
Given a non-negative Radon measure \(\mu\) on \(\R^n\) and a
test plan \(\ppi\) on \((\R^n,\sfd_{\rm Eucl},\mu)\), we define
the space \(\mathbb B_\sppi\) as
\begin{equation}\label{eq:def_B_pi}
\mathbb B_\sppi\coloneqq L^2\big(C([0,1],\R^n),\R^n;\ppi\big).
\end{equation}
Observe that \(\mathbb B_\sppi\) is a separable Hilbert space.
\begin{proposition}
Let \(\mu\) be a Radon measure on \(\R^n\). Let \(\ppi\) be a test
plan on \((\R^n,\sfd_{\rm Eucl},\mu)\). Then it holds that (the equivalence
classes up to \(\mathcal L_1\)-a.e.\ equality of) the mappings
\[\begin{split}
\e-\e_0\colon[0,1]\longrightarrow\mathbb B_\sppi,&\quad
t\longmapsto\e_t-\e_0,\\
{\rm Der}\colon[0,1]\longrightarrow\mathbb B_\sppi,&\quad
t\longmapsto{\rm Der}_t,
\end{split}\]
belong to the space \(L^2([0,1],\mathbb B_\sppi;\mathcal L_1)\).
Moreover, we have that \(\e-\e_0\in AC^2([0,1],\mathbb B_\sppi)\) and
\begin{equation}\label{eq:der_e_t}
\frac{\d}{\d t}(\e_t-\e_0)={\rm Der}_t,
\quad\text{ for }\mathcal L_1\text{-a.e.\ }t\in[0,1],
\end{equation}
where the derivative is intended with respect to the strong topology
of \(\mathbb B_\sppi\).
\end{proposition}
\begin{proof} First of all, let us observe that
\[
\int_0^1\bigg(\int|{\rm Der}_t|^2\,\d\ppi\bigg)\d t
=\int\!\!\!\int_0^1|\dot\gamma_t|^2\,\d t\,\d\ppi(\gamma)<+\infty,
\]
thus \({\rm Der}_t\in\mathbb B_\sppi\) for
\(\mathcal L_1\)-a.e.\ \(t\in[0,1]\) and
\({\rm Der}\in L^2([0,1],\mathbb B_\sppi;\mathcal L_1)\); we omit
the standard proof of the fact that \(\rm Der\) is Borel.
Moreover, for every \(s,t\in[0,1]\) with \(s<t\) it holds
\[
(\e_t-\e_s)(\gamma)=\gamma_t-\gamma_s=\int_s^t\dot\gamma_r\,\d r
=\bigg(\int_s^t{\rm Der}_r\,\d r\bigg)(\gamma),\quad\text{ for }
\ppi\text{-a.e.\ }\gamma,
\]
so that \((\e_t-\e_0)-(\e_s-\e_0)=\e_t-\e_s=
\int_s^t{\rm Der}_r\,\d r\in\mathbb B_\sppi\), whence the statement follows.
\end{proof}

Given any non-negative Radon measure \(\mu\) on \(\R^n\), we will refer
to the metric measure space \((\R^n,\sfd_{\rm Eucl},\mu)\) as a
\textbf{weighted Euclidean space}. The rest of this paper is devoted
to the study of the Sobolev space and the differential structure associated
with \((\R^n,\sfd_{\rm Eucl},\mu)\). We will refer to the elements of
the Hilbert module \(L^2(\R^n,\R^n;\mu)\) as the
\textbf{concrete vector fields} on \((\R^n,\sfd_{\rm Eucl},\mu)\).
Given any \(f\in C^\infty_c(\R^n)\), we denote by
\(\nabla f\in L^2(\R^n,\R^n;\mu)\) the (equivalence class of the)
`strong' gradient of \(f\), \emph{i.e.}, for any \(x\in\R^n\)
we characterise \(\nabla f(x)\in\R^n\) as the unique vector satisfying
\[
\lim_{y\to x}\frac{f(y)-f(x)-\nabla f(x)\cdot(y-x)}{|y-x|}=0.
\]
The Hilbert module \(L^2(\R^n,(\R^n)^*;\mu)\) is the dual
module of \(L^2(\R^n,\R^n;\mu)\) and its elements are said to be
the \textbf{concrete \(1\)-forms} on \((\R^n,\sfd_{\rm Eucl},\mu)\).
The `strong' differential of a given function \(f\in C^\infty_c(\R^n)\)
will be denoted by \(\d f\in L^2(\R^n,(\R^n)^*;\mu)\).
\medskip

The relation between abstract and concrete vector fields on
the weighted Euclidean space has been investigated in \cite{GP16-2},
where the following results have been proven:
\begin{theorem}[Density in energy of smooth functions]
\label{thm:density_smooth}
Let \(\mu\) be a non-negative Radon measure on \(\R^n\).
Let \(f\in W^{1,2}(\R^n,\mu)\) be given. Then there exists
a sequence \((f_i)_i\subseteq C^\infty_c(\R^n)\) such that
\(f_i\to f\) and \(|\nabla f_i|\to|D_\mu f|\) in \(L^2(\mu)\).
\end{theorem}
The proof of the above result was obtained by combining
a standard convolution argument with (a stronger variant of)
Theorem \ref{thm:density_Lip}. As a consequence, the following
statement holds:
\begin{theorem}[The isometric embedding \(\iota_\mu\)]
\label{thm:iota_as_adjoint}
Let \(\mu\geq 0\) be a Radon measure on \(\R^n\). Then there
exists a unique morphism
\({\rm P}_\mu\colon L^2\big(\R^n,(\R^n)^*;\mu\big)\to L^2_\mu(T^*\R^n)\)
such that
\[
{\rm P}_\mu(\d f)=\d_\mu f,\quad\text{ for every }f\in C^\infty_c(\R^n).
\]
Calling \(\iota_\mu\colon L^2_\mu(T\R^n)\to L^2(\R^n,\R^n;\mu)\)
the adjoint of \({\rm P}_\mu\), \emph{i.e.}, the unique morphism
satisfying
\begin{equation}\label{eq:def_iota}
{\rm P}_\mu(\underline\omega)(v)=\underline\omega\big(\iota_\mu(v)\big),
\quad\text{ for every }v\in L^2_\mu(T\R^n)\text{ and }
\underline\omega\in L^2(\R^n,(\R^n)^*;\mu),
\end{equation}
we have that \(\big|\iota_\mu(v)\big|=|v|\) holds \(\mu\)-a.e.\ on \(\R^n\)
for any given \(v\in L^2_\mu(T\R^n)\).
\end{theorem}
\begin{remark}\label{rmk:suff_cond_iota}{\rm
Given any Radon measure \(\mu\geq 0\) on \(\R^n\) and any vector field
\(v\in L^2_\mu(T\R^n)\), it holds that \(\iota_\mu(v)\) can be
characterised as the unique element of \(L^2(\R^n,\R^n;\mu)\) such that
\begin{equation}\label{eq:def_iota_bis}
\int\d_\mu f(v)\,\d\mu=\int\d f\big(\iota_\mu(v)\big)\,\d\mu,
\quad\text{ for every }f\in C^\infty_c(\R^n).
\end{equation}
This readily follows from the fact that
\(\big\{\d f:f\in C^\infty_c(\R^n)\big\}\)
generates \(L^2(\R^n,(\R^n)^*;\mu)\) and that
\(\iota_\mu\colon L^2_\mu(T\R^n)\to L^2(\R^n,\R^n;\mu)\) is
a morphism of \(L^2(\mu)\)-normed \(L^\infty(\mu)\)-modules.
\fr}\end{remark}
As it was observed in \cite{GP16-2}, it immediately follows from
Theorem \ref{thm:iota_as_adjoint} that Euclidean
spaces are \textbf{universally infinitesimally Hilbertian},
in the following sense.
\begin{theorem}[Infinitesimal Hilbertianity of weighted \(\R^n\)]
\label{thm:Eucl_inf_Hilb}
Let \(\mu\geq 0\) be a Radon measure on \(\R^n\).
Then the metric measure space \((\R^d,\sfd_{\rm Eucl},\mu)\)
is infinitesimally Hilbertian.
\end{theorem}
We point out that other two different proofs of Theorem
\ref{thm:Eucl_inf_Hilb} are known: it directly follows from
\cite[Theorem 1.1]{DMGPS18}, and it is one of the main
achievements of \cite{DMLP20}; in Section \ref{ss:AM},
we will briefly describe the strategy of the latter approach.
Let us now recall an important consequence of Theorems
\ref{thm:density_smooth} and \ref{thm:Eucl_inf_Hilb}.
For the reader's usefulness, we also provide its proof.
\begin{corollary}[Strong density of smooth functions]
\label{cor:strong_dens_smooth}
Let \(\mu\geq 0\) be a Radon measure on \(\R^n\). Let
\(f\in W^{1,2}(\R^n,\mu)\) be given. Then there exists a
sequence \((f_i)_i\subseteq C^\infty_c(\R^n)\) such that
\begin{subequations}
\begin{align}\label{eq:strong_dens_smooth_1}
f_i\longrightarrow f,&\quad\text{ strongly in }L^2(\mu),\\
\label{eq:strong_dens_smooth_2}
|\nabla f_i|\longrightarrow|D_\mu f|,&
\quad\text{ strongly in }L^2(\mu),\\
\label{eq:strong_dens_smooth_3}
\nabla_\mu f_i\longrightarrow\nabla_\mu f,&
\quad\text{ strongly in }L^2_\mu(T\R^n).
\end{align}\end{subequations}
\end{corollary}
\begin{proof}
Thanks to Theorem \ref{thm:density_smooth}, we can find
a sequence \((f_i)_i\subseteq C^\infty_c(\R^n)\)
satisfying \eqref{eq:strong_dens_smooth_1} and
\eqref{eq:strong_dens_smooth_2}. Given that
\(|\nabla_\mu f_i|=|D_\mu f_i|\leq|\nabla f_i|\) holds
\(\mu\)-a.e.\ for every \(i\in\N\), we deduce that the
sequence \((\nabla_\mu f_i)_i\) is bounded in
\(L^2_\mu(T\R^n)\). Being \(L^2_\mu(T\R^n)\) Hilbert,
there exists \(v\in L^2_\mu(T\R^n)\) such that
(up to a not relabelled subsequence) it holds that
\(\nabla_\mu f_i\rightharpoonup v\) weakly in \(L^2_\mu(T\R^n)\).
By using Proposition \ref{prop:closure_diff} (and the Riesz
isomorphism), we obtain that \(v=\nabla_\mu f\). Moreover,
\[
\int|D_\mu f|^2\,\d\mu\leq\limi_{i\to\infty}\int|D_\mu f_i|^2
\,\d\mu\leq\lims_{i\to\infty}\int|D_\mu f_i|^2
\,\d\mu\leq\lim_{i\to\infty}\int|\nabla f_i|^2\,\d\mu
\overset{\eqref{eq:strong_dens_smooth_2}}=\int|D_\mu f|^2\,\d\mu,
\]
where the first inequality is granted by the weak convergence
\(\nabla_\mu f_i\rightharpoonup\nabla_\mu f\). Consequently,
we conclude that
\(\int|D_\mu f|^2\,\d\mu=\lim_i\int|D_\mu f_i|^2\,\d\mu\)
and thus \(\nabla_\mu f_i\to\nabla_\mu f\) strongly
in \(L^2_\mu(T\R^n)\), which shows the validity of
\eqref{eq:strong_dens_smooth_3}.
\end{proof}
\subsubsection{Divergence of concrete vector fields}
Let \(\mu\geq 0\) be a Radon measure on \(\R^n\). Then we denote by
\(D(\underline\div_\mu)\) the space of all those vector fields
\(\underline v\in L^2(\R^n,\R^n;\mu)\) whose distributional
divergence belongs to \(L^2(\mu)\). Namely, there exists a function
\(\underline\div_\mu(\underline v)\in L^2(\mu)\) such that
\begin{equation}\label{eq:def_distrib_div}
\int\nabla f\cdot\underline v\,\d\mu=
-\int f\,\underline\div_\mu(\underline v)\,\d\mu,
\quad\text{ for every }f\in C^\infty_c(\R^n).
\end{equation}
Observe that \(\underline\div_\mu\) satisfies
the Leibniz rule, \emph{i.e.}, it holds that
\(f\underline v\in D(\underline\div_\mu)\) and
\[
\underline\div_\mu(f\underline v)=
f\,\underline\div_\mu(\underline v)+\nabla f\cdot\underline v
\]
for every \(f\in C^\infty_c(\R^n)\) and
\(\underline v\in D(\underline\div_\mu)\).
\begin{lemma}[Relation between \(\div_\mu\) and
\(\underline\div_\mu\)]\label{lem:div_vs_und_div}
Let \(\mu\geq 0\) be a Radon measure on \(\R^n\). Then for any vector
field \(v\in L^2_\mu(T\R^n)\) we have that
\begin{equation}\label{eq:relation_div}
v\in D(\div_\mu)\quad\Longleftrightarrow\quad
\iota_\mu(v)\in D(\underline\div_\mu).
\end{equation}
In this case, it holds that
\(\div_\mu(v)=\underline\div_\mu\big(\iota_\mu(v)\big)\)
in the \(\mu\)-a.e.\ sense.
\end{lemma}
\begin{proof}
On the one hand, suppose \(v\in D(\div_\mu)\). Then for any
\(f\in C^\infty_c(\R^n)\) it holds that
\[
\int\nabla f\cdot\iota_\mu(v)\,\d\mu=\int\d f\big(\iota_\mu(v)\big)\,\d\mu
\overset{\eqref{eq:def_iota_bis}}=\int\d_\mu f(v)\,\d\mu
=\int\langle\nabla_\mu f,v\rangle\,\d\mu=-\int f\,\div_\mu(v)\,\d\mu,
\]
whence \(\iota_\mu(v)\in D(\underline\div_\mu)\) and
\(\underline\div_\mu\big(\iota_\mu(v)\big)=\div_\mu(v)\) holds
\(\mu\)-a.e.\ on \(\R^n\). On the other hand, suppose
\(\iota_\mu(v)\in D(\underline\div_\mu)\). Let us fix any function
\(f\in W^{1,2}(\R^n,\mu)\). Corollary \ref{cor:strong_dens_smooth}
grants the existence of \((f_i)_i\subseteq C^\infty_c(\R^n)\)
such that \(f_i\to f\) in \(L^2(\mu)\) and \(\nabla_\mu f_i\to\nabla_\mu f\)
in \(L^2_\mu(T\R^n)\). Therefore, it holds that
\[\begin{split}
\int\langle\nabla_\mu f,v\rangle\,\d\mu
&=\lim_{i\to\infty}\int\langle\nabla_\mu f_i,v\rangle\,\d\mu
\overset{\eqref{eq:def_iota_bis}}=
\lim_{i\to\infty}\int\nabla f_i\cdot\iota_\mu(v)\,\d\mu
=-\lim_{i\to\infty}\int f_i\,\underline\div_\mu\big(\iota_\mu(v)\big)\,\d\mu\\
&=-\int f\,\underline\div_\mu\big(\iota_\mu(v)\big)\,\d\mu.
\end{split}\]
This shows that \(v\in D(\div_\mu)\) and
\(\div_\mu(v)=\underline\div_\mu\big(\iota_\mu(v)\big)\)
holds \(\mu\)-a.e.\ on \(\R^n\), as required.
\end{proof}
\subsection{Distributions on the Euclidean space}
\label{ss:distributions}
We denote by \({\rm Gr}(\R^n)\) the \textbf{Grassmannian}  of \(\R^n\),
\emph{i.e.}, the family of all linear subspaces of \(\R^n\). We endow
\({\rm Gr}(\R^n)\) with the distance
\[
\sfd_{{\rm Gr}(\R^n)}(V,W)\coloneqq\max\Bigg\{
\sup_{\substack{v\in V,\\|v|\leq 1}}\inf_{\substack{w\in W,\\|w|\leq 1}}|v-w|,
\sup_{\substack{w\in W,\\|w|\leq 1}}\inf_{\substack{v\in V,\\|v|\leq 1}}|v-w|
\Bigg\},\quad\text{ for every }V,W\in{\rm Gr}(\R^n),
\]
\emph{i.e.}, \(\sfd_{{\rm Gr}(\R^n)}(V,W)\) is the Hausdorff distance
in \(\R^n\) between the closed unit balls of \(V\) and \(W\).
It holds that \(\big({\rm Gr}(\R^n),\sfd_{{\rm Gr}(\R^n)}\big)\)
is a compact metric space; see, for instance, \cite{Aliprantis_Border}.
\begin{definition}[Distribution]\label{def:distr}
A \textbf{distribution} on \(\R^n\) is a Borel map
\(V\colon\R^n\to{\rm Gr}(\R^n)\). Given any Radon measure \(\mu\) on \(\R^n\),
we denote by \(\mathscr D_n(\mu)\) the family of all distributions on \(\R^n\),
considered up to \(\mu\)-a.e.\ equality. Given any \(V\in\mathscr D_n(\mu)\),
we define \(\Gamma(V)\subseteq L^2(\R^n,\R^n;\mu)\) as
\[
\Gamma(V)\coloneqq\Big\{\underline v\in L^2(\R^n,\R^n;\mu)\;\Big|
\;\underline v(x)\in V(x),\text{ for }\mu\text{-a.e.\ }x\in\R^n\Big\}.
\]
Moreover, we define a partial order on \(\mathscr D_n(\mu)\) in the following
way: given any \(V,W\in\mathscr D_n(\mu)\), we declare that
\(V\leq W\) provided it holds that \(V(x)\subseteq W(x)\)
for \(\mu\)-a.e.\ \(x\in\R^n\).
\end{definition}
It can be readily checked that \(\Gamma(V)\) is an \(L^2(\mu)\)-normed
\(L^\infty(\mu)\)-submodule of \(L^2(\R^n,\R^n;\mu)\).
\begin{proposition}\label{prop:Gamma_biject}
Let \(\mu\geq 0\) be a Radon measure on \(\R^n\).
Then the mapping \(V\mapsto\Gamma(V)\) is a bijection between
\(\mathscr D_n(\mu)\) and the family of \(L^2(\mu)\)-normed
\(L^\infty(\mu)\)-submodules of \(L^2(\R^n,\R^n;\mu)\).
Moreover, the map \(\Gamma\) is order-preserving, \emph{i.e.},
one has \(V\leq W\) if and only if \(\Gamma(V)\subseteq\Gamma(W)\).
\end{proposition}
\begin{proof}
The only non-trivial fact to check is that the mapping \(\Gamma\)
is surjective. To this aim, let us fix an \(L^2(\mu)\)-normed
\(L^\infty(\mu)\)-submodule \(\mathscr M\) of \(L^2(\R^n,\R^n;\mu)\).
Also, take any countable, dense \(\mathbb Q\)-linear subspace
\((\underline v_i)_i\) of \(\mathscr M\). Define
\begin{equation}\label{eq:def_V_for_Gamma_surj}
V(x)\coloneqq{\rm cl}\,\big\{\underline v_i(x)\;\big|\;i\in\N\big\}
\in{\rm Gr}(\R^n),\quad\text{ for }\mu\text{-a.e.\ }x\in\R^n.
\end{equation}
The resulting map \(V\colon\R^n\to{\rm Gr}(\R^n)\)
is Borel. Indeed, for every \(W\in{\rm Gr}(\R^n)\) we have that
\[
\sfd_{{\rm Gr}(\R^n)}\big(V(x),W\big)=
\max\bigg\{\sup_{i\in\N}\,\inf_{j\in\N}\bigg|\frac{\underline v_i(x)}
{\max\big\{|\underline v_i(x)|,1\big\}}-w_j\bigg|,\,
\sup_{j\in\N}\,\inf_{i\in\N}\bigg|\frac{\underline v_i(x)}
{\max\big\{|\underline v_i(x)|,1\big\}}-w_j\bigg|\bigg\}
\]
holds for \(\mu\)-a.e.\ \(x\in\R^n\), where \((w_j)_j\) is any dense
sequence in the closed unit ball of \(W\), thus accordingly
\(x\mapsto\sfd_{{\rm Gr}(\R^n)}\big(V(x),W\big)\) is
\(\mu\)-a.e.\ equivalent to a Borel function.
Then, \(V\in\mathscr D_n(\mu)\).

Let us now prove that \(\mathscr M=\Gamma(V)\). Given that
\(\underline v_i\in\Gamma(V)\) for every \(i\in\N\) by construction
and \(\mathscr M={\rm cl}\,\{\underline v_i\,:\,i\in\N\}\), we
deduce that \(\mathscr M\subseteq\Gamma(V)\). Conversely,
fix any \(\underline v\in\Gamma(V)\). By dominated
convergence theorem we see that the sequence
\((\underline w_j)_j\subseteq\Gamma(V)\), given by
\(\underline w_j\coloneqq\1_{B_j(0)}\,\underline v\) for all
\(j\in\N\), converges to \(\underline v\) in \(\Gamma(V)\).
Fix \(j_0\in\N\) satisfying
\(\mu\big(B_{j_0}(0)\big)>0\). Given any \(j\geq j_0\),
we infer from \eqref{eq:def_V_for_Gamma_surj} that there
is a Borel partition \((E^j_i)_i\) of \(B_j(0)\)
having the property that
\begin{equation}\label{eq:aux_Gamma_surj}
\big|\underline v_i(x)-\underline v(x)\big|^2\leq
\frac{1}{j^2\mu\big(B_j(0)\big)},\quad\text{ for every }
i\in\N\text{ and }\mu\text{-a.e.\ }x\in E^j_i.
\end{equation}
Define \(\underline z_j\coloneqq\sum_{i=1}^\infty\1_{E^j_i}
\,\underline v_i\in\mathscr M\). By exploiting the inequality
in \eqref{eq:aux_Gamma_surj}, we thus obtain that
\[
\int|\underline z_j-\underline w_j|^2\,\d\mu
=\sum_{i=1}^\infty\int_{E^j_i}|\underline v_i-\underline v|^2\,\d\mu
\leq\sum_{i=1}^\infty\frac{\mu(E^j_i)}{j^2\mu\big(B_j(0)\big)}
=\frac{1}{j^2}.
\]
Therefore, we conclude that
\(\|\underline z_j-\underline v\|_{\Gamma(V)}
\leq\frac{1}{j}+\|\underline w_j-\underline v\|_{\Gamma(V)}\to 0\),
so that \(\underline v\in\mathscr M\).
\end{proof}
\begin{remark}{\rm
The statement of Proposition \ref{prop:Gamma_biject} is a particular
instance of a more general result proven in \cite{LP18}, concerning
the representation of a certain class of normed modules as spaces of
sections of a measurable Banach bundle. Nevertheless,
in the special case under consideration (\emph{i.e.}, only
submodules of \(L^2(\R^n,\R^n;\mu)\) are taken into account)
the argument is simpler than the original one in \cite{LP18},
so we opted for providing an easier proof.
\fr}\end{remark}
\begin{lemma}\label{lem:stable_spaces}
Let \(\mu\) be a Radon measure on \(\R^n\).
Let \(\mathscr V\) be a linear subspace of \(L^2(\R^n,\R^n;\mu)\)
such that \(g\underline v\in\mathscr V\) holds for every
\(g\in C^\infty_c(\R^n)\) and \(\underline v\in\mathscr V\). Given a
dense sequence \((\underline v_i)_i\subseteq\mathscr V\), we define
\(V(x)\coloneqq{\rm cl}\,\big\{\underline v_i(x)\,:\,i\in\N\big\}\)
for \(\mu\)-a.e.\ \(x\in\R^n\). Then \(\big\{V(x)\big\}_{x\in\R^n}\)
is a family of linear subspaces of \(\R^n\), which are
\(\mu\)-a.e.\ independent of \((\underline v_i)_i\). Moreover,
it holds that
\[
{\rm cl}\,\mathscr V=\Big\{\underline v\in L^2(\R^n,\R^n;\mu)\;\Big|
\;\underline v(x)\in V(x),\text{ for }\mu\text{-a.e.\ }x\in\R^n\Big\}.
\]
In particular, the map \(V\colon\R^n\to{\rm Gr}(\R^n)\) is a
distribution on \(\R^n\), the Banach space \({\rm cl}\,\mathscr V\)
is an \(L^2(\mu)\)-normed \(L^\infty(\mu)\)-submodule of
\(L^2(\R^n,\R^n;\mu)\), and \(\Gamma(V)={\rm cl}\,\mathscr V\).
\end{lemma}
\begin{proof}
The first part of the statement follows, \emph{e.g.}, from
\cite[Lemma A.1]{Bouchitte_Fragala03}. The fact that \(V\) is a
distribution on \(\R^n\) can be proved by arguing exactly
as in the proof of Proposition \ref{prop:Gamma_biject},
whence the remaining claims immediately follow.
\end{proof}
\begin{remark}[Orthogonal projection]\label{rmk:orth_proj}{\rm
Let \(\mu\geq 0\) be a Radon measure on \(\R^n\).
Let \(V\in\mathscr D_n(\mu)\) be given. We define the
\textbf{orthogonal projection} mapping
\({\rm pr}_V\colon L^2(\R^n,\R^n;\mu)\to\Gamma(V)\) as
\[
{\rm pr}_V(\underline v)(x)\coloneqq\pi_{V(x)}\big(\underline v(x)\big),
\quad\text{ for }\mu\text{-a.e.\ }x\in\R^n,
\]
where \(\pi_{V(x)}\colon\R^n\to V(x)\) is the standard orthogonal
projection. Clearly, the mapping \({\rm pr}_V\) is a surjective,
\(1\)-Lipschitz morphism of \(L^2(\mu)\)-normed
\(L^\infty(\mu)\)-modules.
\fr}\end{remark}
\begin{remark}[Orthogonal complement, II]\label{rmk:orth_compl_II}{\rm
Given any Radon measure \(\mu\) on \(\R^n\) and any distribution
\(V\in\mathscr D_n(\mu)\), we define the \textbf{orthogonal complement}
\(V^\perp\in\mathscr D_n(\mu)\) of \(V\) as
\[
V^\perp(x)\coloneqq\big(V(x)\big)^\perp\subseteq\R^n,
\quad\text{ for }\mu\text{-a.e.\ }x\in\R^n.
\]
Moreover, observe that \(\Gamma(V^\perp)=\Gamma(V)^\perp\),
where \(\Gamma(V)^\perp\) is defined as in Remark \ref{rmk:orth_compl_I}.
\fr}\end{remark}
\section{Characterisation of the Sobolev space on weighted Euclidean spaces}
\subsection{Alberti--Marchese distribution}\label{ss:AM}
In our investigation of the Sobolev space associated with a weighted
Euclidean space, a key role is played by the following result, whose
statement can be roughly summed up in this way: given a Radon measure
\(\mu\) on \(\R^n\), there is a `maximal' distribution \(V_\mu\)
on \(\R^n\) along which all Lipschitz functions are
\(\mu\)-a.e.\ (Fr\'{e}chet) differentiable.
\begin{theorem}[Alberti--Marchese distribution \cite{AM16}]\label{thm:AM}
Let \(\mu\geq 0\) be a Radon measure on \(\R^n\). Then
there exists a unique distribution \(V_\mu\in\mathscr D_n(\mu)\)
such that the following properties hold:
\begin{itemize}
\item[\(\rm i)\)] Every function \(f\in\LIP_c(\R^n)\) is
\(\mu\)-a.e.\ \textbf{differentiable with respect to \(V_\mu\)},
\emph{i.e.}, there exists a vector field
\(\nablaAM f\in\Gamma(V_\mu)\) such that
\begin{equation}\label{eq:formula_nablaAM}
\lim_{V_\mu(x)\ni v\to 0}
\frac{f(x+v)-f(x)-\nablaAM f(x)\cdot v}{|v|}=0,
\quad\text{ for }\mu\text{-a.e.\ }x\in\R^n.
\end{equation}
\item[\(\rm ii)\)] There exists a function \(f_0\in\LIP(\R^n)\)
such that for \(\mu\)-a.e.\ \(x\in\R^n\) it holds that \(f_0\)
is not differentiable at \(x\) with respect to any direction
\(v\in\R^n\setminus V_\mu(x)\).
\end{itemize}
We call \(V_\mu\) the \textbf{Alberti--Marchese distribution}
associated with \(\mu\).
\end{theorem}
In \cite{AM16} the object \(V_\mu\) is called the
`decomposability bundle' of \(\mu\). Here, we chose the term
`distribution' in order to be consistent with our Definition \ref{def:distr}.
Moreover, Theorem \ref{thm:AM} was actually proven under the additional
assumption of \(\mu\) being a finite measure, whence the case of a
possibly infinite Radon measure follows by arguing as in
\cite[Remark 1.6]{DMLP20}.
\begin{remark}\label{rmk:AM_Leb}{\rm
It follows from Rademacher theorem that
\[
V_{\mathcal L^n}(x)=\R^n,\quad\text{ for }\mathcal L^n\text{-a.e.\ }x\in\R^n.
\]
In particular, if \(\mu\ll\mathcal L^n\), then \(V_\mu(x)=\R^n\) holds
for \(\mu\)-a.e.\ \(x\in\R^n\).
\fr}\end{remark}
We shall refer to \(\nablaAM\) as the \textbf{Alberti--Marchese gradient}
operator. It readily follows from \eqref{eq:formula_nablaAM} that the
element \(\nablaAM f\) is uniquely determined (up to \(\mu\)-a.e.\ equality).
Moreover,
\begin{equation}\label{eq:nablaAM_linear}\begin{split}
&\nablaAM(f+g)(x)=\nablaAM f(x)+\nablaAM g(x),
\quad\text{ for }\mu\text{-a.e.\ }x\in\R^n,\\
&\nablaAM(f-g)(x)=\nablaAM f(x)-\nablaAM g(x),
\quad\text{ for }\mu\text{-a.e.\ }x\in\R^n,
\end{split}\end{equation}
are satisfied for every \(f,g\in\LIP_c(\R^n)\).
Let us also recall that it holds that
\begin{equation}\label{eq:nablaAM_leq_lip}
\big|\nablaAM f(x)\big|\leq\lip(f)(x),
\quad\text{ for }\mu\text{-a.e.\ }x\in\R^n,
\end{equation}
as shown in \cite[Remark 1.7]{DMLP20}.
\subsubsection{Consequences of Alberti--Marchese theorem}
Aim of this section is to illustrate the relation between the
Alberti--Marchese distribution and the Sobolev space on weighted \(\R^n\),
investigated in \cite{DMLP20}. We collect in the
following statement the main results of \cite[Section 2]{DMLP20}.
\begin{theorem}\label{thm:DMLP}
Let \(\mu\geq 0\) be a Radon measure on \(\R^n\).
Then the following properties hold:
\begin{itemize}
\item[\(\rm i)\)]
Let \(\ppi\) be a test plan on \((\R^n,\sfd_\Eucl,\mu)\).
Then for \(\ppi\)-a.e.\ curve \(\gamma\) it holds that
\[
\dot\gamma_t\in V_\mu(\gamma_t),\quad
\text{ for }\mathcal L_1\text{-a.e.\ }t\in[0,1].
\]
\item[\(\rm ii)\)] Let \(f\in\LIP_c(\R^n)\) be given. Then
\(|\nablaAM f|\in L^2(\mu)\) is a weak upper gradient of \(f\).
\item[\(\rm iii)\)] Let \(f\in W^{1,2}(\R^n,\mu)\) be given.
Then there exists a sequence \((f_i)_i\subseteq\LIP_c(\R^n)\)
such that \(f_i\to f\) and \(|\nablaAM f_i|\to|D_\mu f|\) in
the strong topology of \(L^2(\mu)\).
\end{itemize}
\end{theorem}
As we already mentioned in the paragraph below Theorem
\ref{thm:Eucl_inf_Hilb}, the universal infinitesimal
Hilbertianity of \(\R^n\) was obtained in
\cite[Theorem 2.3]{DMLP20} as a consequence of
Theorem \ref{thm:DMLP}. The argument was the following:
the Cheeger energy \({\rm E}_{\rm Ch}\) is the lower
semicontinuous envelope of the \textbf{Alberti--Marchese energy}
functional \(\EAM\colon L^2(\mu)\to[0,+\infty]\), given by
\begin{equation}\label{eq:def_E_AM}
\EAM(f)\coloneqq\left\{\begin{array}{ll}
\frac{1}{2}\int|\nablaAM f|^2\,\d\mu,\\
+\infty,\\
\end{array}\quad\begin{array}{ll}
\text{ if }f\in\LIP_c(\R^n),\\
\text{ otherwise,}
\end{array}\right.
\end{equation}
which is clearly \(2\)-homogeneous by construction, and satisfies
the parallelogram rule by \eqref{eq:nablaAM_linear}. Consequently,
the Cheeger energy associated with \((\R^n,\sfd_{\rm Eucl},\mu)\)
satisfies the parallelogram rule, thus yielding the sought conclusion.
\subsection{Identification of the tangent module}
Let \(\mu\) be a given Radon measure on \(\R^n\).
We know from Theorem \ref{thm:iota_as_adjoint} that the
tangent module \(L^2_\mu(T\R^n)\) can be canonically seen
as a submodule of \(L^2(\R^n,\R^n;\mu)\), whence (by Proposition
\ref{prop:Gamma_biject}) we have a natural notion of tangent
distribution \(T_\mu\). In this section, we provide some
alternative characterisations of \(T_\mu\), thus showing
(as described in the introduction) that our approach is
equivalent to the ones introduced in \cite{BBS97} and
\cite{Zhi00,Zhi02}. Some of the proofs that we will carry
out are inspired by \cite{Louet14}.
\subsubsection{Tangent distribution}
We introduce the notion of tangent distribution
on \((\R^n,\sfd_{\rm Eucl},\mu)\):
\begin{definition}[Tangent distribution]\label{def:tg_distr}
Let \(\mu\) be a Radon measure on \(\R^n\).
Then we define the \textbf{tangent distribution} \(T_\mu\) as the
unique element of \(\mathscr D_n(\mu)\) such that
\begin{equation}\label{eq:def_tg_distr}
\Gamma(T_\mu)=\iota_\mu\big(L^2_\mu(T\R^n)\big),
\end{equation}
where \(\iota_\mu\colon L^2_\mu(T\R^n)\to L^2(\R^n,\R^n;\mu)\)
is the isometric embedding described in Theorem
\ref{thm:iota_as_adjoint}.
\end{definition}
\begin{remark}\label{rmk:dim_fibers}{\rm
It is straightforward to check that the module \(L^2_\mu(T\R^n)\)
has dimension \(k\) on a given Borel set \(E\subseteq\R^n\) with
\(\mu(E)>0\) if and only if \(\dim T_\mu(x)=k\) for
\(\mu\)-a.e.\ \(x\in E\).
\fr}\end{remark}
The following result shows that `test plans are tangent
to the distribution \(T_\mu\)', in a sense.
\begin{lemma}\label{lem:speed_pi_tangent_MOD}
Let \(\mu\) be a Radon measure on \(\R^n\). Let \(\ppi\) be a given test plan
on \((\R^n,\sfd_{\rm Eucl},\mu)\). Then for \(\ppi\)-a.e.\ curve \(\gamma\)
it holds that
\begin{equation}\label{eq:speed_pi_tangent_claim}
\dot\gamma_t\in T_\mu(\gamma_t),
\quad\text{ for }\mathcal L_1\text{-a.e.\ }t\in[0,1].
\end{equation}
\end{lemma}
\begin{proof}
Let \(\ppi\) be a given test plan on \((\R^n,\sfd_{\rm Eucl},\mu)\).
Given any \(t\in[0,1]\), consider the pullback morphisms
\(\e_t^*{\rm P}_\mu\colon\e_t^*L^2(\R^n,(\R^n)^*;\mu)
\to\e_t^*L^2_\mu(T^*\R^n)\) and \(\e_t^*\iota_\mu\colon\e_t^*L^2_\mu(T\R^n)
\to\e_t^*L^2(\R^n,\R^n;\mu)\) as in Theorem \ref{thm:pullback}.
The spaces \(\e_t^*L^2_\mu(T\R^n)\) and \(\e_t^*L^2(\R^n,\R^n;\mu)\)
can be identified with the dual modules of \(\e_t^*L^2_\mu(T^*\R^n)\)
and \(\e_t^*L^2(\R^n,(\R^n)^*;\mu)\), respectively,
as a consequence of the separability of \(L^2(\R^n,\R^n;\mu)\)
(which can be readily checked) and of its subspace
\(\iota_\mu\big(L^2_\mu(T\R^n)\big)\); cf.\ \cite[Theorem 1.6.7]{Gigli14}.
Since \(\iota_\mu\) is the adjoint of \({\rm P}_\mu\), it holds that
\(\e_t^*\iota_\mu\) is the adjoint of \(\e_t^*{\rm P}_\mu\), thus in
particular for any element \(z\in\e_t^*L^2_\mu(T\R^n)\) we have that
\begin{equation}\label{eq:adjoint_pullback}
\big((\e_t^*{\rm P}_\mu)(\e_t^*\d f)\big)(z)
=(\e_t^*\d f)\big((\e_t^*\iota_\mu)(z)\big)\;\;\;\ppi\text{-a.e.},
\quad\text{ for every }f\in C^\infty_c(\R^n).
\end{equation}
Moreover, the morphism \(\e_t^*\iota_\mu\) preserves the pointwise norm.
In order to prove it, notice that
\[
\big|(\e_t^*\iota_\mu)(\e_t^*v)\big|=\big|\e_t^*(\iota_\mu(v))\big|
=\big|\iota_\mu(v)\big|\circ\e_t=|v|\circ\e_t=|\e_t^*v|\;\;\;\ppi\text{-a.e.,}
\quad\text{ for every }v\in L^2_\mu(T\R^n),
\]
whence \(\e_t^*\iota_\mu\) is an isometry as we know that
\(\big\{\e_t^*v\,:\,v\in L^2_\mu(T\R^n)\big\}\) generates
\(\e_t^*L^2_\mu(T\R^n)\).

One can readily check that \(\e_t^*L^2(\R^n,\R^n;\mu)\) can be identified
with the space \(\mathbb B_\sppi\), the pullback map
\(\e_t^*\colon L^2(\R^n,\R^n;\mu)\to\mathbb B_\sppi\) being
given by \(\e_t^*\underline v\coloneqq\underline v\circ\e_t\) for every
\(\underline v\in L^2(\R^n,\R^n;\mu)\). An analogous statement holds
for \(\e_t^*L^2(\R^n,(\R^n)^*;\mu)\). Observe that
\begin{equation}\label{eq:char_pullback_tg}
(\e_t^*\iota_\mu)\big(\e_t^*L^2_\mu(T\R^n)\big)=
\Big\{\underline z\in\mathbb B_\sppi\;\Big|\;
\underline z(\gamma)\in T_\mu(\gamma_t),\text{ for }
\ppi\text{-a.e.\ }\gamma\Big\}.
\end{equation}
Let us consider, for \(\mathcal L_1\)-a.e.\ \(t\in[0,1]\), the velocity
\(\ppi'_t\in\e_t^*L^2_\mu(T\R^n)\) of \(\ppi\) as in Proposition
\ref{prop:speed_test_plan}. We deduce from \eqref{eq:formula_speed_test_plan}
that for any given function \(f\in C^\infty_c(\R^n)\) it holds that
\begin{equation}\label{eq:speed_test_plan_concrete}
(f\circ\gamma)'_t=(\e_t^*\d_\mu f)(\ppi'_t)(\gamma),
\quad\text{ for }\ppi\text{-a.e.\ }\gamma.
\end{equation}
For \(\mathcal L_1\)-a.e.\ \(t\in[0,1]\), consider the mapping
\({\rm Der}_t\in\mathbb B_\sppi\) defined in
\eqref{eq:def_Der}. We claim that
\begin{equation}\label{eq:speed_test_plan_concrete_claim}
(\e_t^*\iota_\mu)(\ppi'_t)={\rm Der}_t,\quad
\text{ for }\mathcal L_1\text{-a.e.\ }t\in[0,1].
\end{equation}
Given any function \(f\in C^\infty_c(\R^n)\), we have that for
\(\ppi\)-a.e.\ curve \(\gamma\) it holds that
\[
\big((\e_t^*{\rm P}_\mu)(\e_t^*\d f)\big)(\ppi'_t)(\gamma)
=(\e_t^*\d_\mu f)(\ppi'_t)(\gamma)
\overset{\eqref{eq:speed_test_plan_concrete}}=
(f\circ\gamma)'_t=(\d f\circ\e_t)({\rm Der}_t)(\gamma)
=(\e_t^*\d f)({\rm Der}_t)(\gamma).
\]
Since the identity in \eqref{eq:adjoint_pullback} actually characterises
\(\e_t^*\iota_\mu\), we deduce that the claim
\eqref{eq:speed_test_plan_concrete_claim} holds. In particular, we have
that \({\rm Der}_t\in(\e_t^*\iota_\mu)\big(\e_t^*L^2_\mu(T\R^n)\big)\)
for \(\mathcal L_1\)-a.e.\ \(t\in[0,1]\), whence
\eqref{eq:char_pullback_tg} yields
\[
\dot\gamma_t={\rm Der}_t(\gamma)\in T_\mu(\gamma_t),
\quad\text{ for }(\ppi\otimes\mathcal L_1)\text{-a.e.\ }(\gamma,t).
\]
Thanks to Fubini theorem, we finally conclude that the sought
property \eqref{eq:speed_pi_tangent_claim} is satisfied.
\end{proof}
Clearly, in order to identify the minimal weak upper gradient
of a given Sobolev function, it is sufficient to look at the
directions that are selected by the test plans. The following
result makes this claim precise. 
\begin{lemma}\label{lem:tg_distr_gives_wug}
Let \(\mu\) be a Radon measure on \(\R^n\).
Let \(V\in\mathscr D_n(\mu)\) satisfy the following property:
given any test plan \(\ppi\) on \((\R^n,\sfd_{\rm Eucl},\mu)\),
it holds that \(\dot\gamma_t\in V(\gamma_t)\) for
\((\ppi\otimes\mathcal L_1)\)-a.e.\ \((\gamma,t)\).
Then for any function \(f\in C^\infty_c(\R^n)\) we have that
\(\big|{\rm pr}_V(\nabla f)\big|\) is a weak upper gradient of \(f\).
\end{lemma}
\begin{proof}
Fix any test plan \(\ppi\) on \((\R^n,\sfd_{\rm Eucl},\mu)\).
Then for \(\ppi\)-a.e.\ curve \(\gamma\) it holds that
\[
\big|(f\circ\gamma)'_t\big|
=\big|\nabla f(\gamma_t)\cdot\dot\gamma_t\big|
=\big|{\rm pr}_V(\nabla f)(\gamma_t)\cdot\dot\gamma_t\big|
\leq\big|{\rm pr}_V(\nabla f)\big|(\gamma_t)\,|\dot\gamma_t|,
\quad\text{ for }\mathcal L_1\text{-a.e.\ }t\in[0,1].
\]
By arbitrariness of \(\ppi\), we conclude that
\(\big|{\rm pr}_V(\nabla f)\big|\) is a weak upper gradient of \(f\).
\end{proof}
\subsubsection{An axiomatic notion of weak gradient}
Another possible way to define the tangent fibers is via
the vectorial relaxation procedure proposed by Zhikov in
\cite{Zhi00,Zhi02} and studied by Louet in \cite{Louet14}.
Below we introduce a generalisation of such approach, tailored
for our purposes.
\begin{definition}[\(G\)-structure]\label{def:G-struct}
Let \(\mu\geq 0\) be a Radon measure on \(\R^n\). Then by
\textbf{\(G\)-structure} on \((\R^n,\sfd_{\rm Eucl},\mu)\)
we mean a couple \((\mathcal V,\bar\nabla)\) satisfying
the following list of axioms:
\begin{itemize}
\item[\textbf{A1.}] \(\mathcal V\) is a linear subspace of
\(W^{1,2}(\R^n,\mu)\) containing \(C^\infty_c(\R^n)\).
\item[\textbf{A2.}] \(\bar\nabla\colon\mathcal V\to L^2(\R^n,\R^n;\mu)\)
is a linear operator.
\item[\textbf{A3.}] \(|\bar\nabla f|\) is a weak upper
gradient of \(f\) for any \(f\in\mathcal V\),
with \(|\bar\nabla f|\in L^\infty(\mu)\) if
\(f\in C^\infty_c(\R^n)\).
\item[\textbf{A4.}] \(\bar\nabla\) satisfies the
\textbf{Leibniz rule}, \emph{i.e.}, if \(f\in\mathcal V\)
and \(g\in C^\infty_c(\R^n)\), then \(fg\in\mathcal V\) and
\[
\bar\nabla(fg)=f\,\bar\nabla g+g\bar\nabla f,
\quad\text{ in the }\mu\text{-a.e.\ sense.}
\]
\item[\textbf{A5.}] Calling \({\rm E}_G\colon L^2(\mu)\to[0,+\infty]\)
the energy functional
\[
{\rm E}_G(f)\coloneqq\left\{\begin{array}{ll}
\frac{1}{2}\int|\bar\nabla f|^2\,\d\mu,\\
+\infty,
\end{array}\quad\begin{array}{ll}
\text{ if }f\in\mathcal V,\\
\text{ otherwise,}
\end{array}\right.
\]
it holds that \({\rm E}_{\rm Ch}\) is the lower
semicontinuous envelope of \({\rm E}_G\).
\end{itemize}
\end{definition}
The term `\(G\)-structure' is somehow inspired by the
notion of \(D\)-structure, which has been proposed by
V.\ Gol'dshtein and M.\ Troyanov in the paper \cite{GoldTroy01}.
Therein, they developed an axiomatic theory of Sobolev
spaces on general metric measure spaces. In our setting,
thanks to the presence of an underlying linear structure,
the axiomatisation can be formulated in terms of `gradients'
rather than `moduli of the gradients'.
\begin{remark}[Density in energy]\label{rmk:density_energy_G-struct}{\rm
Observe that axiom \textbf{A5} is equivalent to requiring that
the elements of \(\mathcal V\) are \textbf{dense in energy} in
\(W^{1,2}(\R^n,\mu)\), \emph{i.e.}, for every
\(f\in W^{1,2}(\R^n,\mu)\) there exists a sequence
\((f_i)_i\subseteq\mathcal V\) such that \(f_i\to f\) and
\(|\bar\nabla f_i|\to|D_\mu f|\) strongly in \(L^2(\mu)\).
\fr}\end{remark}
\begin{example}[Examples of \(G\)-structures]\label{ex:G-struct}{\rm
Let us describe two examples of \(G\)-structures on
\((\R^n,\sfd_{\rm Eucl},\mu)\) that will play a fundamental
role in the forthcoming discussion:
\begin{itemize}
\item[\(\rm a)\)] The \(G_\mu\)-structure
\(\big(C^\infty_c(\R^n),\nabla\big)\).
\item[\(\rm b)\)] The \(\GAM\)-structure
\(\big(\LIP_c(\R^n),\nablaAM\big)\). Observe that
\begin{equation}\label{eq:char_nabla_AM_f}
\nablaAM f={\rm pr}_{V_\mu}(\nabla f),
\quad\text{ for every }f\in C^\infty_c(\R^n).
\end{equation}
\end{itemize}
The axioms defining a \(G\)-structure are satisfied both
in a) and in b), as a consequence of the results contained in
Sections \ref{ss:diff_struct_mms} and \ref{ss:AM}, respectively.
\fr}\end{example}
Much like in the case of Sobolev spaces via test plans
and minimal weak upper gradients, any \(G\)-structure naturally
comes with a unique minimal object, called the minimal
\(G\)-gradient:
\begin{definition}[\(G\)-gradient]\label{def:G-gradient}
Let \(\mu\geq 0\) be a Radon measure on \(\R^n\) and
\((\mathcal V,\bar\nabla)\) a \(G\)-structure on
\((\R^n,\sfd_{\rm Eucl},\mu)\). Fix \(f\in L^2(\mu)\).
Then we say that \(f\) admits a \textbf{\(G\)-gradient}
\(\underline v\in L^2(\R^n,\R^n;\mu)\) provided there
exists a sequence \((f_i)_i\subseteq\mathcal V\) such that
\[\begin{split}
f_i\to f,&\quad\text{ strongly in }L^2(\mu),\\
\bar\nabla f_i\to\underline v,&\quad\text{ strongly in }L^2(\R^n,\R^n;\mu).
\end{split}\]
We denote by \(G(f)\) the closed affine subspace
of \(L^2(\R^n,\R^n;\mu)\) made of all \(G\)-gradients of \(f\).
The (unique) element of \(G(f)\) of minimal norm
is called the \textbf{minimal \(G\)-gradient} of \(f\).
\end{definition}
Observe that \(\bar\nabla f\in G(f)\) for every \(f\in\mathcal V\),
as one can see by taking \(f_i\coloneqq f\) for every \(i\in\N\).
\begin{remark}\label{rmk:fibers_G_mu}{\rm
Note that the space \(G(0)\) is closed under multiplication
by \(C^\infty_c(\R^n)\)-functions: given any \(g\in C^\infty_c(\R^n)\)
and \(\underline v\in G(0)\), it holds that \(g\underline v\in G(0)\).
Indeed, if \((f_i)_i\subseteq C^\infty_c(\R^n)\) is a sequence satisfying
\(f_i\to 0\) in \(L^2(\mu)\) and \(\bar\nabla f_i\to\underline v\)
in \(L^2(\R^n,\R^n;\mu)\), then \(g f_i\to 0\) in \(L^2(\mu)\) and
\(\bar\nabla(gf_i)=g\bar\nabla f_i+f_i\bar\nabla g\to g\underline v\)
in \(L^2(\R^n,\R^n;\mu)\). In particular, we deduce from
Lemma \ref{lem:stable_spaces} that the space \(G(0)\)
is an \(L^2(\mu)\)-normed \(L^\infty(\mu)\)-submodule of
\(L^2(\R^n,\R^n;\mu)\).
\fr}\end{remark}
\begin{definition}\label{def:W_G}
Let \(\mu\) be a Radon measure on \(\R^n\) and
\((\mathcal V,\bar\nabla)\) a \(G\)-structure on
\((\R^n,\sfd_{\rm Eucl},\mu)\). Then we define \(W_G\)
as the unique element of \(\mathscr D_n(\mu)\) such that
\begin{equation}\label{eq:def_V_G}
\Gamma(W_G)=G(0).
\end{equation}
\end{definition}
Notice that the previous definition is meaningful as a
consequence of Remark \ref{rmk:fibers_G_mu}.
\subsubsection{Alternative characterisations of the tangent distribution}
The following two results show that \(G\)-structures can be
used to provide an alternative notion of Sobolev space, which
turns out to be fully equivalent to the approach via test plans.
\begin{theorem}[Alternative characterisation of \(W^{1,2}\)]
\label{thm:alt_char_Sobolev}
Let \(\mu\geq 0\) be a Radon measure on \(\R^n\) and let
\((\mathcal V,\bar\nabla)\) be a \(G\)-structure on
\((\R^n,\sfd_{\rm Eucl},\mu)\). Then
\[
W^{1,2}(\R^n,\mu)=\big\{f\in L^2(\mu)\;\big|\;G(f)\neq\emptyset\big\}.
\]
Moreover, for every \(f\in W^{1,2}(\R^n,\mu)\) it holds
that the minimal weak upper gradient \(|D_\mu f|\) coincides
(in the \(\mu\)-a.e.\ sense) with the pointwise norm of the
minimal \(G\)-gradient of \(f\).
\end{theorem}
\begin{proof}
First of all, let us fix any function
\(f\in W^{1,2}(\R^n,\mu)\). We claim that
\(G(f)\neq\emptyset\) and that there exists an element
\(\underline v\in G(f)\) such that
\(\big\||\underline v|\big\|_{L^2(\mu)}\leq
\big\||D_\mu f|\big\|_{L^2(\mu)}\). In order to prove it, choose a
sequence \((f_i)_i\subseteq\mathcal V\) such that \(f_i\to f\)
and \(|\bar\nabla f_i|\to|D_\mu f|\) in \(L^2(\mu)\), whose existence
is observed in Remark \ref{rmk:density_energy_G-struct}. Up to a not
relabelled subsequence, it holds that
\(\bar\nabla f_i\rightharpoonup\underline v\)
weakly in \(L^2(\R^n,\R^n;\mu)\) for some vector field
\(\underline v\in L^2(\R^n,\R^n;\mu)\). By Banach--Saks theorem
we know that (up to taking a further subsequence) it holds that the functions
\(g_i\coloneqq\frac{1}{i}\sum_{j=1}^i f_j\in\mathcal V\)
satisfy \(g_i\to f\) in \(L^2(\mu)\) and
\(\bar\nabla g_i\to\underline v\) strongly in \(L^2(\R^n,\R^n;\mu)\),
which yields \(\underline v\in G(f)\). It also holds that
\(\big\||\underline v|\big\|_{L^2(\mu)}=
\lim_i\big\||\bar\nabla g_i|\big\|_{L^2(\mu)}\leq
\lim_i\frac{1}{i}\sum_{j=1}^i\big\||\bar\nabla f_j|\big\|_{L^2(\mu)}=
\big\||D_\mu f|\big\|_{L^2(\mu)}\).

Conversely, let us suppose that \(f\in L^2(\mu)\) satisfies
\(G(f)\neq\emptyset\). Fix an element
\(\underline v\in G(f)\). Pick any sequence
\((f_i)_i\subseteq\mathcal V\) such that
\(f_i\to f\) in \(L^2(\mu)\) and \(\bar\nabla f_i\to\underline v\)
in \(L^2(\R^n,\R^n;\mu)\). In particular,
\(|\bar\nabla f_i|\to|\underline v|\) in \(L^2(\mu)\).
Since \(|\bar\nabla f_i|\) is a weak upper gradient of \(f_i\)
for every \(i\in\N\), we deduce from Proposition \ref{prop:closure_diff}
that \(f\in W^{1,2}(\R^n,\mu)\) and
\(|D_\mu f|\leq|\underline v|\) holds \(\mu\)-a.e.\ in \(\R^n\).
All in all, the proof of the statement is finally achieved.
\end{proof}
\begin{proposition}\label{prop:mwug_smooth}
Let \(\mu\geq 0\) be a Radon measure on \(\R^n\) and let
\((\mathcal V,\bar\nabla)\) be a \(G\)-structure on
\((\R^n,\sfd_{\rm Eucl},\mu)\). Then for any \(f\in\mathcal V\)
it holds that \({\rm pr}_{W_G^\perp}(\bar\nabla f)\)
is the minimal \(G\)-gradient of \(f\). In particular,
\(\big|{\rm pr}_{W_G^\perp}(\bar\nabla f)\big|\)
is the minimal weak upper gradient of \(f\).
\end{proposition}
\begin{proof}
We claim that for any element \(\underline v\in G(f)\) it holds
that \({\rm pr}_{W_G^\perp}(\underline v)\) belongs to \(G(f)\)
and is independent of \(\underline v\). First, recall that
\(\bar\nabla f\in G(f)\). Since Lemma \ref{lem:stable_spaces}
yields \({\rm pr}_{W_G}(\bar\nabla f)\in G(0)\), there exists
\((g_i)_i\subseteq\mathcal V\) such that \(g_i\to 0\) in
\(L^2(\mu)\) and \(\bar\nabla g_i\to{\rm pr}_{W_G}(\bar\nabla f)\)
in \(L^2(\R^n,\R^n;\mu)\). Hence, the sequence
\((f-g_i)_i\subseteq\mathcal V\) satisfies
\(f-g_i\to f\) in \(L^2(\mu)\) and
\(\bar\nabla(f-g_i)\to{\rm pr}_{W_G^{\perp}}(\bar\nabla f)\)
in \(L^2(\R^n,\R^n;\mu)\), yielding
\({\rm pr}_{W_G^\perp}(\bar\nabla f)\in G(f)\).
Furthermore, let \(\underline v\in G(f)\) be fixed. Pick any
sequence \((f_i)_i\subseteq\mathcal V\) such that \(f_i\to f\)
in \(L^2(\mu)\) and \(\bar\nabla f_i\to\underline v\) in
\(L^2(\R^n,\R^n;\mu)\). This implies that
\((f-f_i)_i\subseteq\mathcal V\) satisfies \(f-f_i\to 0\) in
\(L^2(\mu)\) and \(\bar\nabla(f-f_i)\to\bar\nabla f-\underline v\)
in \(L^2(\R^n,\R^n;\mu)\). Consequently, we conclude that
\(\bar\nabla f-\underline v\in G(0)\), thus
\(\bar\nabla f(x)-\underline v(x)\in W_G(x)\) for
\(\mu\)-a.e.\ \(x\in\R^n\). This means that
\({\rm pr}_{W_G^\perp}(\bar\nabla f)-{\rm pr}_{W_G^\perp}
(\underline v)={\rm pr}_{W_G^\perp}(\bar\nabla f-
\underline v)=0\). All in all, the claim is proven.

Now the first part of the statement readily follows:
given any \(\underline v\in G(f)\), it holds that
\[
\big\|{\rm pr}_{W_G^\perp}(\bar\nabla f)\big\|_{L^2(\R^n,\R^n;\mu)}
=\big\|{\rm pr}_{W_G^\perp}(\underline v)\big\|_{L^2(\R^n,\R^n;\mu)}
\leq\|\underline v\|_{L^2(\R^n,\R^n;\mu)}.
\]
Therefore, we finally conclude that \({\rm pr}_{W_G^\perp}(\bar\nabla f)\)
is the minimal \(G\)-gradient of \(f\). The last part of the
statement now follows from Theorem \ref{thm:alt_char_Sobolev},
thus the proof is complete.
\end{proof}
We are now ready to state and prove the main result of
this section. It says that the tangent distribution \(T_\mu\)
can be expressed either in terms of the domain of the
distributional divergence \(\underline{\rm div}_\mu\), or
of the \(G_\mu\)-structure. We point out that, to the best
of our knowledge, the equivalence between these two approaches
(namely, items ii) and iii) of the following result) was
previously not known; one of the two implications is proved
in \cite[end of Section 1]{Louet14}. 
\begin{theorem}[Alternative characterisations of \(T_\mu\)]
\label{thm:alt_char_T_mu}
Let \(\mu\geq 0\) be a Radon measure on \(\R^n\).
Then the tangent distribution \(T_\mu\) can be equivalently
characterised in the following ways:
\begin{itemize}
\item[\(\rm i)\)] \(T_\mu\) is the unique minimal element of
\(\mathscr D_n(\mu)\) with the property that for any
test plan \(\ppi\) on \((\R^n,\sfd_{\rm Eucl},\mu)\)
it holds \(\dot\gamma_t\in T_\mu(\gamma_t)\) for
\((\ppi\otimes\mathcal L_1)\)-a.e.\ \((\gamma,t)\in
AC^2([0,1],\R^n)\times[0,1]\).
\item[\(\rm ii)\)] \(T_\mu\) is the unique minimal element of
\(\mathscr D_n(\mu)\) with the property that for any
\(\underline v\in D(\underline\div_\mu)\) it holds that
\(\underline v(x)\in T_\mu(x)\) for \(\mu\)-a.e.\ \(x\in\R^n\).
Equivalently, \(\iota_\mu\big(L^2_\mu(T\R^n)\big)={\rm cl}\,D(\underline\div_\mu)\).
\item[\(\rm iii)\)] It holds that \(T_\mu=W_\mu^\perp\), where
\(W_\mu\coloneqq W_{G_\mu}\) stands for the distribution on
\(\R^n\) associated with the \(G_\mu\)-structure
\(\big(C^\infty_c(\R^n),\nabla\big)\), which is described in
item \(\rm a)\) of Example \ref{ex:G-struct}.
\end{itemize}
In items \(\rm i)\) and \(\rm ii)\), minimality has to be intended with respect to
the partial order \(\leq\) on \(\mathscr D_n(\mu)\).
\end{theorem}
\begin{proof}
We subdivide the proof into several steps:\\
{\color{blue}\textsc{Step 1.}} First of all, we claim that
\(T_\mu\leq W_\mu^\perp\). This would follow from the inclusions
\begin{equation}\label{eq:inclusion_tg_mod}
\iota_\mu\big(L^2_\mu(T\R^n)\big)\subseteq{\rm cl}\,D(\underline\div_\mu)
\subseteq G_\mu(0)^\perp.
\end{equation}
Indeed, by using \eqref{eq:inclusion_tg_mod}, \eqref{eq:def_tg_distr},
\eqref{eq:def_V_G}, and Remark \ref{rmk:orth_compl_II}, we deduce that
\(\Gamma(T_\mu)\subseteq\Gamma(W_\mu^\perp)\), whence \(T_\mu\leq W_\mu^\perp\)
by the last part of the statement of Proposition \ref{prop:Gamma_biject}. To prove the first inclusion
in \eqref{eq:inclusion_tg_mod}, recall that
\({\rm cl}\,D(\div_\mu)=L^2_\mu(T\R^n)\) by Lemma
\ref{lem:density_D_Delta} and notice that
\[
\iota_\mu\big(L^2_\mu(T\R^n)\big)=
\iota_\mu\big({\rm cl}\,D(\div_\mu)\big)
\subseteq{\rm cl}\,\iota_\mu\big(D(\div_\mu)\big)
\overset{\eqref{eq:relation_div}}\subseteq
{\rm cl}\,D(\underline\div_\mu).
\]
To prove the second inclusion in \eqref{eq:inclusion_tg_mod},
it clearly suffices to show that \(D(\underline\div_\mu)\subseteq G_\mu(0)^\perp\).
To this aim, fix \(\underline v\in D(\underline\div_\mu)\) and
\(\underline w\in G_\mu(0)\). Choose any \((f_i)_i\subseteq C^\infty_c(\R^n)\)
such that \(f_i\to 0\) in \(L^2(\mu)\) and \(\nabla f_i\to\underline w\) in
\(L^2(\R^n,\R^n;\mu)\). Therefore, we have that
\[
\int\underline v\cdot\underline w\,\d\mu
=\lim_{i\to\infty}\int\underline v\cdot\nabla f_i\,\d\mu
=-\lim_{i\to\infty}\int f_i\,\underline\div_\mu(\underline v)\,\d\mu=0.
\]
By arbitrariness of \(\underline v\) and \(\underline w\), we conclude
that \(D(\underline\div_\mu)\subseteq G_\mu(0)^\perp\), so that
\eqref{eq:inclusion_tg_mod} is proven.\\
{\color{blue}\textsc{Step 2.}} Let \(V\in\mathscr D_n(\mu)\) be
a distribution on \(\R^n\) such that for any test plan \(\ppi\)
on \((\R^n,\sfd_{\rm Eucl},\mu)\) it holds that
\(\dot\gamma_t\in V(\gamma_t)\) for
\((\ppi\otimes\mathcal L_1)\)-a.e.\ \((\gamma,t)\).
Then we claim that \(W_\mu^\perp\leq V\).

First, from \textsc{Step 1} and Lemma
\ref{lem:speed_pi_tangent_MOD} we know that 
\(W_\mu^\perp\cap V\in\mathscr D_n(\mu)\) satisfies the same
property as \(V\), \emph{i.e.}, for any test plan \(\ppi\) one has
\(\dot\gamma_t\in W_\mu(\gamma_t)^\perp\cap V(\gamma_t)\)
for \((\ppi\otimes\mathcal L_1)\)-a.e.\ \((\gamma,t)\).
Let \(W\in\mathscr D_n(\mu)\) be defined so that \(W(x)\) is the
orthogonal complement of \(W_\mu(x)^\perp\cap V(x)\) in
\(W_\mu(x)^\perp\) for \(\mu\)-a.e.\ point \(x\in\R^n\).
Given any function \(f\in C^\infty_c(\R^n)\), it holds that
\(\big|{\rm pr}_{W_\mu^\perp\cap V}(\nabla f)\big|\) is a weak
upper gradient of \(f\) by Lemma \ref{lem:tg_distr_gives_wug},
while \(\big|{\rm pr}_{W_\mu^\perp}(\nabla f)\big|\) is the
minimal weak upper gradient of \(f\) by Proposition
\ref{prop:mwug_smooth}. This implies that
\(\big|{\rm pr}_{W_\mu^\perp\cap V}(\nabla f)\big|=
\big|{\rm pr}_{W_\mu^\perp}(\nabla f)\big|\) holds
\(\mu\)-a.e.\ in \(\R^n\) for every \(f\in C^\infty_c(\R^n)\),
thus accordingly we might conclude that
\[
\big|{\rm pr}_W(\nabla f)\big|^2
=\big|{\rm pr}_{W_\mu^\perp}(\nabla f)\big|^2
-\big|{\rm pr}_{W_\mu^\perp\cap V}(\nabla f)\big|^2=0
\;\;\;\mu\text{-a.e.},\quad\text{ for every }f\in C^\infty_c(\R^n).
\]
Given that \(\big\{\nabla f\,:\,f\in C^\infty_c(\R^n)\big\}\) generates
\(L^2(\R^n,\R^n;\mu)\) on \(\R^n\), we deduce that the image of \({\rm pr}_W\) coincides
with \(\{0\}\), thus necessarily \(W=\{0\}\). This means that
\(W_\mu(x)^\perp\cap V(x)=W_\mu(x)^\perp\) for \(\mu\)-a.e.\ point \(x\in\R^n\),
which grants that \(W_\mu^\perp\leq V\). Hence, the claim is proven.\\
{\color{blue}\textsc{Step 3.}} By Lemma \ref{lem:speed_pi_tangent_MOD}
we know that \(T_\mu\) satisfies the property in item i), whence by
\textsc{Steps} 1 and 2 we see that \(T_\mu=W_\mu^\perp\) is the (unique)
minimal distribution on \(\R^n\) having this property, proving items
iii) and i). Moreover, notice that
\(\iota_\mu\big(L^2_\mu(T\R^n)\big)={\rm cl}\,D(\underline\div_\mu)
=G_\mu(0)^\perp\) follows from \eqref{eq:inclusion_tg_mod} and the
identity \(T_\mu=W_\mu^\perp\), thus item ii) is proven as well.
\end{proof}
Note that by combining Lemma \ref{lem:div_vs_und_div} with item ii)
of the previous theorem, we obtain that
\[
\iota_\mu\big(D(\div_\mu)\big)=D(\underline\div_\mu),
\quad\text{ for every Radon measure }\mu\geq 0\text{ on }\R^n.
\]
As another immediate consequence of Theorem \ref{thm:alt_char_T_mu},
we also see that the tangent distribution is always contained in
the Alberti--Marchese distribution:
\begin{corollary}\label{cor:T_mu_in_W_mu}
Let \(\mu\geq 0\) be a Radon measure on \(\R^n\). Then it holds that
\begin{equation}\label{eq:T_mu_in_W_mu}
T_\mu\leq V_\mu.
\end{equation}
\end{corollary}
\begin{proof}
Combine item i) of Theorem \ref{thm:DMLP} with item i)
of Theorem \ref{thm:alt_char_T_mu}.
\end{proof}
\begin{remark}\label{rmk:T_mu_neq_V_mu}{\rm
It might happen that \(T_\mu\neq V_\mu\). For instance, let \(C\)
be a fat Cantor set in \(\R\) and consider \(\mu\coloneqq\mathcal L^1|_C\).
Then \(V_\mu(x)=\R\) for \(\mu\)-a.e.\ \(x\in\R\) by Remark
\ref{rmk:AM_Leb}, while \(T_\mu(x)=\{0\}\) for \(\mu\)-a.e.\ \(x\in\R\)
as a consequence of the fact that the support of \(\mu\) is totally
disconnected, thus \(W^{1,2}(\R,\mu)=L^2(\mu)\) and \(|D_\mu f|=0\)
holds \(\mu\)-a.e.\ for every \(f\in W^{1,2}(\R,\mu)\).
\fr}\end{remark}
\subsection{Identification of the minimal weak upper gradient}
Once we have the equivalent characterisations of the Sobolev
space and of the tangent distribution at our disposal, we can
identify the minimal weak upper gradient of every given Lipschitz
function. First, we deal with smooth functions (in Proposition
\ref{prop:char_minimal_Gmu_grad}), then we pass to general
Lipschitz functions (in Theorem \ref{thm:mwug_Lip}). A consequence of
Proposition \ref{prop:char_minimal_Gmu_grad} -- namely, the
fact that \(|D_\mu f|=\big|{\rm pr}_{T_\mu}(\nabla f)\big|\)
holds for every \(f\in C^\infty_c(\R^n)\) -- was already
proven by S.\ Di Marino in \cite[Theorem 7.4.8]{DiMarinoPhD}.
\begin{proposition}\label{prop:char_minimal_Gmu_grad}
Let \(\mu\geq 0\) be a Radon measure on \(\R^n\). Then for every
\(f\in W^{1,2}(\R^n,\mu)\) it holds that \(\iota_\mu(\nabla_\mu f)\)
is the minimal \(G_\mu\)-gradient of \(f\). In particular, we have that
\begin{equation}\label{eq:explicit_iota_nabla_mu}
\iota_\mu(\nabla_\mu f)={\rm pr}_{T_\mu}(\nabla f),
\quad\text{ for every }f\in C^\infty_c(\R^n).
\end{equation}
\end{proposition}
\begin{proof}
First of all, let us prove \eqref{eq:explicit_iota_nabla_mu}.
Fix any \(f\in C^\infty_c(\R^n)\). Choose any
\(v\in L^2_\mu(T\R^n)\) such that \(\iota_\mu(v)={\rm pr}_{T_\mu}(\nabla f)\).
Therefore, for every \(g\in C^\infty_c(\R^n)\) it holds that
\[\begin{split}
\d_\mu g(\nabla_\mu f)&=\d_\mu f(\nabla_\mu g)
\overset{\eqref{eq:def_iota}}=\d f\big(\iota_\mu(\nabla_\mu g)\big)
=\nabla f\cdot\iota_\mu(\nabla_\mu g)
={\rm pr}_{T_\mu}(\nabla f)\cdot\iota_\mu(\nabla_\mu g)\\
&=\iota_\mu(v)\cdot\iota_\mu(\nabla_\mu g)
=\langle v,\nabla_\mu g\rangle=\d_\mu g(v)
\overset{\eqref{eq:def_iota}}=\d g\big(\iota_\mu(v)\big)
=\d g\big({\rm pr}_{T_\mu}(\nabla f)\big).
\end{split}\]
In light of Remark \ref{rmk:suff_cond_iota}, we can conclude that
\(\iota_\mu(\nabla_\mu f)={\rm pr}_{T_\mu}(\nabla f)\), thus proving
\eqref{eq:explicit_iota_nabla_mu}.

Let us now fix \(f\in W^{1,2}(\R^n,\mu)\).
By Corollary \ref{cor:strong_dens_smooth} there is a sequence
\((f_i)_i\subseteq C^\infty_c(\R^n)\) such that \(f_i\to f\),
\(|D_\mu f_i|\to|D_\mu f|\), and \(|\nabla f_i|\to|D_\mu f|\)
in \(L^2(\mu)\). Up to a not relabelled subsequence, we can
also assume that \(\nabla_\mu f_i\rightharpoonup v\) weakly
in \(L^2_\mu(T\R^n)\), for some \(v\in L^2_\mu(T\R^n)\).
By Banach--Saks theorem we can find a sequence
\((g_i)_i\subseteq C^\infty_c(\R^n)\) such that
\(g_i\to f\) in \(L^2(\mu)\),
\(\nabla_\mu g_i\to v\) in \(L^2_\mu(T\R^n)\), and
\(\lims_i\big\||\nabla g_i|\big\|_{L^2(\mu)}
\leq\big\||D_\mu f|\big\|_{L^2(\mu)}\). It follows from
Proposition \ref{prop:closure_diff} that \(v=\nabla_\mu f\),
thus in particular \(|D_\mu g_i|\to|D_\mu f|\) in \(L^2(\mu)\).
Item iii) of Theorem \ref{thm:alt_char_T_mu} yields
\[\begin{split}
\lims_{i\to\infty}\int\big|{\rm pr}_{W_\mu}(\nabla g_i)\big|^2\,\d\mu
&\overset{\phantom{\eqref{eq:explicit_iota_nabla_mu}}}=
\lims_{i\to\infty}\bigg(\int|\nabla g_i|^2\,\d\mu-
\int\big|{\rm pr}_{T_\mu}(\nabla g_i)\big|^2\,\d\mu\bigg)\\
&\overset{\eqref{eq:explicit_iota_nabla_mu}}=
\lims_{i\to\infty}\int|\nabla g_i|^2\,\d\mu
-\lim_{i\to\infty}\int|D_\mu g_i|^2\,\d\mu\\
&\overset{\phantom{\eqref{eq:explicit_iota_nabla_mu}}}\leq
\int|D_\mu f|^2\,\d\mu-\int|D_\mu f|^2\,\d\mu=0,
\end{split}\]
whence \({\rm pr}_{W_\mu}(\nabla g_i)\to 0\) in \(L^2(\R^n,\R^n;\mu)\).
Since \(\nabla_\mu g_i\to\nabla_\mu f\) in \(L^2_\mu(T\R^n)\)
and \(\iota_\mu\) is continuous, we can finally conclude that
\[
\nabla g_i={\rm pr}_{T_\mu}(\nabla g_i)+{\rm pr}_{W_\mu}(\nabla g_i)
\overset{\eqref{eq:explicit_iota_nabla_mu}}=
\iota_\mu(\nabla_\mu g_i)+{\rm pr}_{W_\mu}(\nabla g_i)
\to\iota_\mu(\nabla_\mu f),\quad\text{ in }L^2(\R^n,\R^n;\mu).
\]
This means that \(\iota_\mu(\nabla_\mu f)\in G_\mu(f)\).
Given that \(\big|\iota_\mu(\nabla_\mu f)\big|=|D_\mu f|\) holds
\(\mu\)-a.e.\ in \(\R^n\), we infer from Theorem
\ref{thm:alt_char_Sobolev} that \(\iota_\mu(\nabla_\mu f)\) is
the minimal \(G_\mu\)-gradient of \(f\). The proof is complete.
\end{proof}
\begin{theorem}[Minimal weak upper gradient of Lipschitz functions]
\label{thm:mwug_Lip}
Let \(\mu\geq 0\) be a Radon measure on \(\R^n\).
Then it holds that
\[
|D_\mu f|=\big|{\rm pr}_{T_\mu}(\nablaAM f)\big|
\;\;\;\mu\text{-a.e.},\quad\text{ for every }f\in\LIP_c(\R^n).
\]
\end{theorem}
\begin{proof}
Consider the \(G_\mu\)-structure and the \(\GAM\)-structure, which were
defined in items a) and b) of Example \ref{ex:G-struct}, respectively.
For brevity, we call \(\WAM\coloneqq W_{\GAM}\).
First, we prove that
\begin{equation}\label{eq:equiv_T_mu}
\WAM^\perp\cap V_\mu=T_\mu.
\end{equation}
In order to show one inclusion, fix
\(\underline v\in\GAM(0)^\perp\cap\Gamma(V_\mu)\)
and \(\underline w\in G_\mu(0)\). Let us pick any sequence
\((f_i)_i\subseteq C^\infty_c(\R^n)\) satisfying
\(f_i\to 0\) in \(L^2(\mu)\) and \(\nabla f_i\to\underline w\)
in \(L^2(\R^n,\R^n;\mu)\). The latter convergence, together with
\eqref{eq:char_nabla_AM_f}, yields
\(\nablaAM f_i={\rm pr}_{V_\mu}(\nabla f_i)\to
{\rm pr}_{V_\mu}(\underline w)\) in \(L^2(\R^n,\R^n;\mu)\),
which gives \({\rm pr}_{V_\mu}(\underline w)\in\GAM(0)\).
Since \(\underline v\in\GAM(0)^\perp\cap\Gamma(V_\mu)\), we get
that \(\underline v\cdot\underline w=\underline v\cdot
{\rm pr}_{V_\mu}(\underline w)=0\) holds \(\mu\)-a.e., which
implies \(\GAM(0)^\perp\cap\Gamma(V_\mu)\subseteq G_\mu(0)^\perp\)
and thus \(\WAM^\perp\cap V_\mu\leq W_\mu^\perp=T_\mu\).

To prove the converse inclusion, let us consider the orthogonal complement
\(Z\) of \(\WAM^\perp\cap V_\mu\) in \(T_\mu\),
namely, \(Z\coloneqq(\WAM^\perp\cap V_\mu)^\perp\cap T_\mu\).
Notice that for any \(f\in C^\infty_c(\R^n)\) we have that
\begin{equation}\label{eq:mwug_Lip_aux1}
\big|{\rm pr}_{T_\mu}(\nablaAM f)\big|^2
=\big|{\rm pr}_{\WAM^\perp\cap V_\mu}(\nablaAM f)\big|^2+
\big|{\rm pr}_Z(\nablaAM f)\big|^2,\quad\text{ in the }\mu\text{-a.e.\ sense.}
\end{equation}
By applying Proposition \ref{prop:mwug_smooth} to the \(G_\mu\)-structure
and the \(\GAM\)-structure, we obtain that
\begin{equation}\label{eq:mwug_Lip_aux2}\begin{split}
\big|{\rm pr}_{T_\mu}(\nablaAM f)\big|
&\overset{\eqref{eq:T_mu_in_W_mu}}=\big|{\rm pr}_{T_\mu}(\nabla f)\big|
=\big|{\rm pr}_{W_\mu^\perp}(\nabla f)\big|=|D_\mu f|,\\
\big|{\rm pr}_{\WAM^\perp\cap V_\mu}(\nablaAM f)\big|
&\overset{\phantom{\eqref{eq:T_mu_in_W_mu}}}=
\big|{\rm pr}_{\WAM^\perp}(\nablaAM f)\big|=|D_\mu f|,
\end{split}\end{equation}
respectively. By plugging \eqref{eq:mwug_Lip_aux2} into \eqref{eq:mwug_Lip_aux1},
we deduce that
\(\big|{\rm pr}_Z(\nabla f)\big|=\big|{\rm pr}_Z(\nablaAM f)\big|=0\)
holds \(\mu\)-a.e.\ for all \(f\in C^\infty_c(\R^n)\). Since
\(\big\{\nabla f:f\in C^\infty_c(\R^n)\big\}\) generates \(L^2(\R^n,\R^n;\mu)\)
on \(\R^n\), we conclude that \(Z=\{0\}\), which means that
the identity in \eqref{eq:equiv_T_mu} is verified.

Now fix any function \(f\in\LIP_c(\R^n)\). We know that
\(\big|{\rm pr}_{\WAM^\perp}(\nablaAM f)\big|\) is the minimal
weak upper gradient of \(f\) by Proposition \ref{prop:mwug_smooth}.
Since it also holds that
\[
\big|{\rm pr}_{T_\mu}(\nablaAM f)\big|
\overset{\eqref{eq:equiv_T_mu}}=
\big|{\rm pr}_{\WAM^\perp\cap V_\mu}(\nablaAM f)\big|
=\big|{\rm pr}_{\WAM^\perp}(\nablaAM f)\big|,
\quad\text{ in the }\mu\text{-a.e.\ sense,}
\]
we finally conclude that
\(\big|{\rm pr}_{T_\mu}(\nablaAM f)\big|\)
is the minimal weak upper gradient of \(f\).
\end{proof}
It readily follows from Theorem \ref{thm:mwug_Lip} that
those measures \(\mu\) on \(\R^n\) for which minimal weak
upper gradient and local Lipschitz constant always coincide
can be explicitly characterised in terms of the tangent
distribution \(T_\mu\), as the next result shows.
\begin{corollary}\label{cor:equiv_Df=lipf}
Let \(\mu\geq 0\) be a Radon measure on \(\R^n\).
Then the following are equivalent:
\begin{itemize}
\item[\(\rm i)\)] \(|D_\mu f|=\lip(f)\) holds
\(\mu\)-a.e., for every \(f\in\LIP_c(\R^n)\),
\item[\(\rm ii)\)] \(T_\mu(x)=\R^n\), for \(\mu\)-a.e.\ \(x\in\R^n\).
\end{itemize}
\end{corollary}
\begin{proof}
Suppose i) holds. To prove ii), we argue by contradiction:
suppose there exists a Borel set \(E\subseteq\R^n\)
such that \(\mu(E)>0\) and \(T_\mu(x)\neq\R^n\) for
\(\mu\)-a.e.\ \(x\in E\). This means that for
\(\mu\)-a.e.\ point \(x\in E\) there exists a vector
\(v\in\mathbb Q^n\) such that \(v\notin T_\mu(x)\), in other words
\[
E\subseteq\bigcup_{v\in\mathbb Q^n}\big\{x\in\R^n\;\big|\;
v\notin T_\mu(x)\big\},\quad\text{ up to }\mu\text{-null sets.}
\]
Hence, there exist a vector \(v\in\mathbb Q^n\) and a Borel set
\(F\subseteq E\) such that \(\mu(F)>0\) and \(v\notin T_\mu(x)\)
for \(\mu\)-a.e.\ \(x\in F\). Choose a radius \(r>0\)
such that \(\mu\big(F\cap B_r(0)\big)>0\) and a function
\(f\in C^\infty_c(\R^n)\) satisfying \(\nabla f(x)=v\)
for every \(x\in B_r(0)\). By using Proposition
\ref{prop:char_minimal_Gmu_grad}, we thus deduce that
\[
|D_\mu f|(x)=\big|{\rm pr}_{T_\mu}(\nabla f)\big|(x)
=\big|\pi_{T_\mu(x)}(v)\big|<|v|=\lip(f)(x),
\quad\text{ for }\mu\text{-a.e.\ }x\in F\cap B_r(0).
\]
This leads to a contradiction with i), whence accordingly
ii) is proven.

Conversely, suppose ii) holds. A fortiori, we have that
\(V_\mu(x)=\R^n\) for \(\mu\)-a.e.\ \(x\in\R^n\) (recall
Corollary \ref{cor:T_mu_in_W_mu}), so that any given function
\(f\in\LIP_c(\R^n)\) is \(\mu\)-a.e.\ differentiable
and thus \(|\nablaAM f|=\lip(f)\) in the \(\mu\)-a.e.\ sense.
Finally, by using Theorem \ref{thm:mwug_Lip} we obtain that
\[
|D_\mu f|=\big|{\rm pr}_{T_\mu}(\nablaAM f)\big|
=|\nablaAM f|=\lip(f),\quad\text{ holds }\mu\text{-a.e.\ on }\R^n,
\]
proving the validity of i).
\end{proof}
\section{Some applications}\label{s:applications}
\subsection{Tangent fibers on the singular part}\label{ss:tg_sing}
In the structure theory of Radon measures on Euclidean spaces,
a breakthrough is represented by the celebrated paper \cite{DPR16}
by G.\ De Philippis and F.\ Rindler. A consequence of their main
result is reported in Theorem \ref{thm:DPR}.
\medskip

In this section, we will combine the results by De
Philippis--Rindler with our knowledge of the tangent
distribution, in order to prove that
for any Radon measure \(\mu=\rho\mathcal L^n+\mu^s\) on \(\R^n\)
(where \(\mu^s\perp\mathcal L^n\)) it holds that
\(T_\mu(x)\neq\R^n\) for \(\mu^s\)-a.e.\ point \(x\in\R^n\);
see Theorem \ref{thm:dim_fibers}. This gives a positive answer
to a variant of a question raised by I.\ Fragal\`{a} and
C.\ Mantegazza in \cite[Remark 4.4]{FM99}; the original problem
was posed in terms of a different notion of tangent fiber.
However, by adapting our arguments one can solve
also their original open problem.
We point out that neither the kind of results we will prove in
this section, nor the techniques we will use, are really new.
See, \emph{e.g.}, \cite{DPMR16,MK18,GP16-2} for similar
statements and arguments.
\subsubsection{Reminder on Euclidean \(1\)-currents}
Recall that a \textbf{\(1\)-current} \(\rm T\) on \(\R^n\) is a
linear and continuous real-valued functional defined
on the space of smooth, compactly-supported \(1\)-forms
on \(\R^n\). Its \textbf{total mass} \({\bf M}({\rm T})\) is given by the
supremum of \({\rm T}(\underline\omega)\) among all smooth,
compactly-supported \(1\)-forms \(\underline\omega\) on \(\R^n\)
that satisfy \(|\underline\omega|\leq 1\) on all \(\R^n\).
If \({\bf M}({\rm T})\) is finite, then \(\rm T\) is an
\(\R^n\)-valued Radon measure on \(\R^n\), whence by
using the Radon--Nikod\'{y}m theorem one can find a finite,
non-negative Borel measure \(\|{\rm T}\|\) on \(\R^n\) and a
vector field \(\vec{\rm T}\in L^1(\R^n,\R^n;\|{\rm T}\|)\),
with \(\big|\vec{\rm T}(x)\big|=1\) for
\(\|{\rm T}\|\)-a.e.\ point \(x\in\R^n\), such that
\({\rm T}=\vec{\rm T}\,\|{\rm T}\|\). The \textbf{boundary}
\(\partial{\rm T}\) of \(\rm T\) is the \(0\)-current (\emph{i.e.},
the generalised function) on \(\R^n\) which is defined as
\(\partial{\rm T}(f)\coloneqq{\rm T}(\d f)\) for all
\(f\in C^\infty_c(\R^n)\). A \(1\)-current \(\rm T\) on \(\R^n\)
is said to be \textbf{normal} provided
\({\bf M}({\rm T}),{\bf M}(\partial{\rm T})<+\infty\),
where \({\bf M}(\partial{\rm T})\coloneqq
\sup\big\{\partial{\rm T}(f)\,:\,f\in C^\infty_c(\R^n),
\,|f|\leq 1\text{ on }\R^n\big\}\). When the total mass
\({\bf M}(\partial{\rm T})\) is finite, the
\(0\)-current \(\partial{\rm T}\) can be canonically
identified with a (finite) signed measure on \(\R^n\).
\medskip

The following deep result, concerning the structure of normal \(1\)-currents
in the Euclidean space, has been proven by G.\ De Philippis and
F.\ Rindler in the paper \cite{DPR16}.
\begin{theorem}\label{thm:DPR}
Let \(\mu\geq 0\) be a Radon measure on \(\R^n\) and let
\({\rm T}_1,\ldots,{\rm T}_n\) be normal \(1\)-currents
in \(\R^n\) such that \(\mu\ll\|{\rm T}_i\|\) for
every \(i=1,\ldots,n\). Suppose that
\[
\vec{\rm T}_1(x),\ldots,\vec{\rm T}_n(x)\in\R^n
\text{ are linearly independent,}\quad
\text{ for }\mu\text{-a.e.\ }x\in\R^n.
\]
Then it holds that \(\mu\ll\mathcal L^n\).
\end{theorem}
As pointed out in \cite{DPR16}, Theorem \ref{thm:DPR} has
-- amongst many others -- the following consequence:
\begin{theorem}[Weak converse of Rademacher theorem]
\label{thm:weak_Rademacher}
Let \(\mu\) be a Radon measure on \(\R^n\) such that every
function \(f\in\LIP(\R^n)\) is \(\mu\)-a.e.\ differentiable.
Then it holds that \(\mu\ll\mathcal L^n\).
\end{theorem}
In turn, the weak converse of Rademacher theorem readily implies
that the Alberti--Marchese distribution has full rank if
and only if the measure under consideration is absolutely
continuous with respect to the Lebesgue measure:
\begin{corollary}\label{cor:char_AM_full}
Let \(\mu\) be a given Radon measure on \(\R^n\). Then it holds that
\[
V_\mu(x)=\R^n,\;\;\;\text{for }\mu\text{-a.e.\ }x\in\R^n
\quad\Longleftrightarrow\quad\mu\ll\mathcal L^n.
\]
\end{corollary}
\begin{proof}
If \(V_\mu(x)=\R^n\) for \(\mu\)-a.e.\ \(x\in\R^n\),
then every Lipschitz function is \(\mu\)-a.e.\ differentiable, whence
\(\mu\ll\mathcal L^n\) by Theorem \ref{thm:weak_Rademacher}.
The converse implication is observed in Remark \ref{rmk:AM_Leb}.
\end{proof}
\begin{example}[Vector fields with divergence as normal \(1\)-currents]
\label{ex:vector_fields_as_currents}
{\rm Let \(\mu\) be a finite Borel measure on \(\R^n\) and
\(\underline v\in D(\underline\div_\mu)\). Let us associate
to \(\underline v\) the \(1\)-current \(\mathcal I(\underline v)\)
on \(\R^n\), defined as
\begin{equation}\label{eq:def_currentif}
\mathcal I(\underline v)(\underline\omega)\coloneqq
\int\underline\omega(\underline v)\,\d\mu,
\quad\text{ for every smooth, compactly-supported }1\text{-form }
\underline\omega\text{ on }\R^n.
\end{equation}
Then we claim that \(\mathcal I(\underline v)\) is a normal
\(1\)-current and that it satisfies
\[
\overrightarrow{\mathcal I(\underline v)}=
\1_{\{|\underline v|>0\}}\frac{\underline v}{|\underline v|},
\quad\big\|\mathcal I(\underline v)\big\|=|\underline v|\mu,
\quad\partial\mathcal I(\underline v)=
-\underline\div_\mu(\underline v)\mu.
\]

Indeed, the fact that
\({\bf M}\big(\mathcal I(\underline v)\big)<+\infty\),
and the explicit formulae for
\(\overrightarrow{\mathcal I(\underline v)}\) and
\(\big\|\mathcal I(\underline v)\big\|\), are immediate
consequences of \eqref{eq:def_currentif}, while it readily
follows from the identity
\[
\partial\mathcal I(\underline v)(f)=\mathcal I(\underline v)(\d f)
=\int\d f(\underline v)\,\d\mu=
-\int f\,\underline\div_\mu(\underline v)\,\d\mu,
\quad\text{ for every }f\in C^\infty_c(\R^n),
\]
that the \(0\)-current \(\partial\mathcal I(\underline v)\) has
finite total mass and satisfies
\(\partial\mathcal I(\underline v)=-\underline\div_\mu(\underline v)\mu\).
\fr}\end{example}
\subsubsection{The dimension drops on the singular part}
As a first step, we show that a given Radon measure on \(\R^n\)
must be absolutely continuous with respect to the Lebesgue
measure \(\mathcal L^n\) if restricted to any Borel set
where the tangent module has maximal dimension.
\begin{proposition}\label{prop:top_dim_tg_mod}
Let \(\mu\) be a finite Borel measure on \(\R^n\).
Suppose \(L^2_\mu(T\R^n)\) has dimension equal to \(n\)
on a Borel set \(E\subseteq\R^n\). Then \(\mu|_E\ll\mathcal L^n\).
\end{proposition}
\begin{proof}
Fix a countable dense subset \(\mathcal C\) of
\(D(\underline\div_\mu)\). Given that \(\Gamma(T_\mu)=
{\rm cl}\,D(\underline\div_\mu)\) by item ii) of Theorem
\ref{thm:alt_char_T_mu} and \(\underline\div_\mu\) satisfies
the Leibniz rule, we know from Lemma \ref{lem:stable_spaces}
that \(T_\mu(x)\) coincides with \({\rm cl}\,\big\{\underline v(x)
\,:\,\underline v\in\mathcal C\big\}\) for
\(\mu\)-a.e.\ \(x\in\R^n\). In particular, Remark
\ref{rmk:dim_fibers} grants that:
\begin{equation}\label{eq:top_dim_tg_mod_aux}
\text{For }\mu\text{-a.e.\ }x\in E,\text{ there exist }
\underline v_1,\ldots,\underline v_n\in\mathcal C:\quad{\rm span}
\big\{\underline v_1(x),\ldots,\underline v_n(x)\big\}=\R^n.
\end{equation}
Consider the family \((S_k)_{k\in\N}\) of all those subsets of
\(\mathcal C\) made exactly of \(n\) elements. Given any \(k\in\N\),
we denote by \(E_k\) the set of all points \(x\in E\) such that
\(\underline v_1(x),\ldots,\underline v_n(x)\in\R^n\) are linearly
independent, where \(\{\underline v_1,\ldots,\underline v_n\}=S_k\).
Then \eqref{eq:top_dim_tg_mod_aux} grants that the Borel sets
\(E_k\) satisfy \(\mu\big(E\setminus\bigcup_k E_k\big)=0\).
Now fix any \(k\in\N\) and call
\(S_k=\{\underline v_1,\ldots,\underline v_n\}\). Thanks to
Example \ref{ex:vector_fields_as_currents}, the \(1\)-currents
\(\mathcal I(\underline v_1),\ldots,\mathcal I(\underline v_n)\)
are normal and satisfy \(\mu|_{E_k}\ll\big\|\mathcal I(\underline v_i)\big\|\)
for all \(i=1,\ldots,n\). Therefore, we conclude from Theorem \ref{thm:DPR} that
\(\mu|_{E_k}\ll\mathcal L^n\) for all \(k\in\N\),
thus \(\mu|_E\ll\mathcal L^n\).
\end{proof}
It is now easy to prove, as an immediate consequence of
Proposition \ref{prop:top_dim_tg_mod}, that the tangent
fibers cannot have dimension \(n\) on the singular part
of the measure \(\mu\) under consideration.
\begin{theorem}[Tangent fibers on the singular part]
\label{thm:dim_fibers}
Let \(\mu\) be a finite Borel measure on \(\R^n\),
with Lebesgue decomposition \(\mu=\rho\mathcal L^n+\mu^s\).
Then it holds that
\begin{equation}\label{eq:dim_fibers}
\dim T_\mu(x)<n,\quad\text{ for }\mu^s\text{-a.e.\ }x\in\R^n.
\end{equation}
\end{theorem}
\begin{proof}
Fix a Borel set \(B\subseteq\R^n\) such that \(\mathcal L^n(B)=\mu^s(\R^n\setminus B)=0\).
We argue by contradiction: suppose there is a Borel set \(E\subseteq B\)
such that \(\mu^s(E)>0\) and \(\dim T_\mu(x)=n\) for \(\mu^s\)-a.e.\ \(x\in E\).
In particular, \(\mu(E)>0\) and \(\dim T_\mu(x)=n\) for \(\mu\)-a.e.\ \(x\in E\).
As observed in Remark \ref{rmk:dim_fibers}, this means that the tangent module
\(L^2_\mu(T\R^n)\) has dimension \(n\) on \(E\). Therefore, Proposition
\ref{prop:top_dim_tg_mod} grants that \(\mu^s|_E=\mu|_E\ll\mathcal L^n\).
This leads to a contradiction, as \(\mathcal L^n(E)=0\) but \(\mu^s(E)>0\).
\end{proof}
\begin{remark}{\rm
Actually, Theorem \ref{thm:dim_fibers} holds for any non-negative
Radon measure \(\mu\) on \(\R^n\). Indeed, given any
\(\bar x\in{\rm spt}(\mu)\) and \(r>0\), it can be readily
deduced from \cite[Proposition 2.6]{Gigli12} that
\(T_{\mu_r}(x)=T_\mu(x)\) is satisfied for \(\mu\)-a.e.\ \(x\in B_r(\bar x)\),
where we set \(\mu_r\coloneqq\mu|_{B_r(\bar x)}\). Moreover,
notice that \((\mu_r)^s=\mu^s|_{B_r(\bar x)}\). Therefore, by applying
Theorem \ref{thm:dim_fibers} to the measures \((\mu_k)_{k\in\N}\)
we deduce that \(\mu\) itself satisfies \eqref{eq:dim_fibers},
thus showing that in the statement of Theorem \ref{thm:dim_fibers}
the finiteness assumption on \(\mu\) can be dropped.
\fr}\end{remark}
\begin{remark}[Weighted real line]{\rm
As already mentioned in the introduction, the Sobolev space
on weighted \(\R\) has been fully understood by S.\ Di Marino
and G.\ Speight in \cite{DiMarinoSpeight15}. More specifically,
they completely characterised the minimal weak upper gradient
of any Lipschitz function \(f\in W^{1,2}(\R,\mu)\), where
\(\mu\) is a given Radon measure on \(\R\); see
\cite[Theorem 2]{DiMarinoSpeight15}. We point out that our
results imply a part (but not the whole) of their statement:
Theorem \ref{thm:mwug_Lip} grants that
\(|D_\mu f|(x)\in\big\{0,\lip(f)(x)\big\}\) is satisfied
for \(\mu\)-a.e.\ \(x\in\R\), while Theorem \ref{thm:dim_fibers}
ensures that \(T_\mu(x)=\{0\}\) and thus \(|D_\mu f|(x)=0\)
hold for \(\mu^s\)-a.e.\ \(x\in\R\).
\fr}\end{remark}
It is worth to isolate the following statement, which might be
seen as a special case of Theorem \ref{thm:dim_fibers}
(or, alternatively, of Corollary \ref{cor:char_AM_full}).
\begin{corollary}\label{cor:necessary_measure_ac}
Let \(\mu\geq 0\) be a Radon measure on \(\R^n\) such that
\begin{equation}\label{eq:mwug_equals_lip}
|D_\mu f|=\lip(f)\;\;\;\mu\text{-a.e.},\quad
\text{ for every }f\in\LIP_c(\R^n).
\end{equation}
Then it holds that \(\mu\ll\mathcal L^n\).
\end{corollary}
\begin{proof}
By Corollary \ref{cor:equiv_Df=lipf}, we know that
\eqref{eq:mwug_equals_lip} is equivalent to
\(T_\mu(x)=\R^n\) for \(\mu\)-a.e.\ \(x\in\R^n\).
Therefore, it follows from Theorem \ref{thm:dim_fibers}
that \(\mu^s=0\), which exactly means that \(\mu\ll\mathcal L^n\).

Alternatively, one can argue as follows: since \(T_\mu(x)=\R^n\)
for \(\mu\)-a.e.\ \(x\in\R^n\), we know a fortiori that
\(V_\mu(x)=\R^n\) for \(\mu\)-a.e.\ \(x\in\R^n\), thus accordingly
\(\mu\ll\mathcal L^n\) by Corollary \ref{cor:char_AM_full}.
\end{proof}
\begin{remark}\label{rmk:PI_case}{\rm
Suppose that \(\mu\) is a Radon measure on \(\R^n\) such that
the resulting metric measure space \((\R^n,\sfd_{\rm Eucl},\mu)\)
is doubling and supports a weak \((1,2)\)-Poincar\'{e} inequality,
in the sense of \cite{HKST15}. Then the property
in \eqref{eq:mwug_equals_lip} is satisfied, as proven by J.\ Cheeger
in \cite{Cheeger00}. Therefore, it follows from Corollary
\ref{cor:necessary_measure_ac} that the measure \(\mu\) must
be absolutely continuous with respect to \(\mathcal L^n\).
This fact was already proven by A.\ Schioppa in
\cite{Schioppa15}. See also \cite{DPMR16}.
\fr}\end{remark}
\subsection{A geometric characterisation of the tangent
distribution}\label{ss:geom_T_mu}
The aim of this section is to show that the tangent distribution
\(T_\mu\) associated with a given Radon measure \(\mu\) on \(\R^n\)
admits a `geometric' characterisation in terms of the velocity
of test plans, somehow refining Theorem \ref{thm:alt_char_T_mu}.
More precisely, we will prove that there exists a sequence
\((\ppi_i)_i\) of test plans on \((\R^n,\sfd_{\rm Eucl},\mu)\)
having the following property: \(T_\mu\) is obtained as the closure
of the velocities of the plans \(\ppi_i\) at time \(0\), in a
suitable sense; see Theorem \ref{thm:fiber_cl_dot_pi} for the
correct statement. In order to achieve this goal, a key tool
is given by the notion of test plan representing a gradient,
which has been defined and proven to exist (in high generality)
by N.\ Gigli in \cite{Gigli12}.
\subsubsection{Reminder on test plans representing a gradient}
First of all, let us report the notion of test plan representing
the gradient of a Sobolev function; recall the definition
\eqref{eq:def_KE_t} of \({\rm KE}_t\).
\begin{definition}[Test plan representing a gradient \cite{Gigli12}]
Let \((\X,\sfd,\mu)\) be a metric measure space. Let \(f\in W^{1,2}(\X,\mu)\)
be given. Then a test plan \(\ppi\) on \((\X,\sfd,\mu)\) is said to
\textbf{represent the gradient} of the function \(f\) provided it satisfies
the following property:
\begin{equation}\label{eq:plan_repr_grad}
\lim_{t\searrow 0}\frac{f\circ\e_t-f\circ\e_0}{{\rm KE}_t}
=\lim_{t\searrow 0}\frac{{\rm KE}_t}{t}=|D_\mu f|\circ\e_0,
\quad\text{ strongly in }L^2(\ppi).
\end{equation}
\end{definition}
Test plans representing a gradient exist under mild assumptions,
as the next result shows.
\begin{theorem}[Existence of test plans representing a gradient \cite{Gigli12}]
\label{thm:existence_plan_repr_grad}
Let \((\X,\sfd,\mu)\) be a metric measure space.
Let \(\nu\) be a Borel probability measure on \((\X,\sfd)\) such that
\(\int\sfd^2(\cdot,\bar x)\,\d\nu<+\infty\) for every \(\bar x\in\X\),
and \(\nu\leq C\mu\) for some constant \(C>0\). Let \(f\in W^{1,2}(\X,\mu)\)
be given. Then there exists a test plan \(\ppi\) on \((\X,\sfd,\mu)\)
that represents the gradient of \(f\) and satisfies \((\e_0)_*\ppi=\nu\).
\end{theorem}
In lack of an appropriate reference, we provide a quick
proof of the following elementary continuity result. To
do so, we use the well-known density of \(\LIP_c(\R^n,\R^n)\)
in \(L^2(\R^n,\R^n;\mu)\).
\begin{lemma}\label{lem:cont_comp_e_t}
Let \(\mu\geq 0\) be a Radon measure on \(\R^n\). Let \(\ppi\) be
a test plan on \((\R^n,\sfd_{\rm Eucl},\mu)\). Then for every
\(\underline v\in L^2(\R^n,\R^n;\mu)\) it holds that
\[
[0,1]\ni t\longmapsto\underline v\circ\e_t\in
\mathbb B_\sppi\;\text{ is a continuous curve.}
\]
\end{lemma}
\begin{proof}
Fix any \(\underline v\in L^2(\R^n,\R^n;\mu)\).
Choose compactly-supported Lipschitz maps \(\underline v_i\colon\R^n\to\R^n\)
such that \(\underline v_i\to\underline v\) in \(L^2(\R^n,\R^n;\mu)\).
Given any \(t\in[0,1]\), we have \(\lim_{s\to t}
\int|\underline v_i\circ\e_s-\underline v_i\circ\e_t|^2\,\d\ppi=0\)
by dominated convergence theorem, so
\([0,1]\ni t\mapsto
\underline v_i\circ\e_t\in\mathbb B_\sppi\)
is continuous. Moreover, the curves
\(t\mapsto\underline v_i\circ\e_t\) uniformly converge to
\(t\mapsto\underline v\circ\e_t\) as \(i\to\infty\). Indeed, it holds that
\[
\sup_{t\in[0,1]}
\int\big|\underline v_i\circ\e_t-\underline v\circ\e_t\big|^2\,\d\ppi
=\sup_{t\in[0,1]}
\int|\underline v_i-\underline v|^2\circ\e_t\,\d\ppi
\leq{\rm Comp}(\ppi)\int|\underline v_i-\underline v|^2\,\d\mu.
\]
Therefore, the curve \([0,1]\ni t\mapsto\underline v\circ\e_t\in\mathbb B_\sppi\)
is continuous as well, as required.
\end{proof}
As one might expect, if a test plan \(\ppi\) represents the
gradient of a Sobolev function \(f\), then for any other Sobolev
function \(g\) we have, roughly speaking, that the derivative
at \(t=0\) of the map \(t\mapsto g\circ\e_t\in L^1(\ppi)\)
coincides with the scalar product
\(\langle\nabla_\mu g,\nabla_\mu f\rangle\circ\e_0\).
This claim is made precise by the ensuing result, which has been
proven in \cite[Corollary 2.4]{P20}.
\begin{proposition}\label{prop:plan_repr_grad_prod}
Let \((\X,\sfd,\mu)\) be an infinitesimally Hilbertian space.
Let \(f\in W^{1,2}(\X,\mu)\) be given. Let \(\ppi\) be a test
plan on \((\X,\sfd,\mu)\) that represents the gradient of \(f\).
Then for every function \(g\in W^{1,2}(\X,\mu)\) it holds that
\[
\frac{g\circ\e_t-g\circ\e_0}{t}\rightharpoonup
\langle\nabla_\mu g,\nabla_\mu f\rangle\circ\e_0,
\quad\text{ weakly in }L^1(\ppi)\text{ as }t\searrow 0.
\]
\end{proposition}
\subsubsection{Geometric characterisation of the tangent fibers}
In the setting of weighted Euclidean spaces, we have that
test plans representing a gradient admit a `concrete' derivative
at \(t=0\):
\begin{theorem}[Initial velocity of test plans representing a gradient]
\label{thm:deriv_time_0_tp}
Let \(\mu\geq 0\) be a Radon measure on \(\R^n\). Let
\(f\in W^{1,2}(\R^n,\mu)\) be given. Let \(\ppi\) be a test plan
on \((\R^n,\sfd_{\rm Eucl},\mu)\) that represents the gradient
of \(f\). Then it holds that
\begin{equation}\label{eq:def_D_pi}
\exists\,{\rm D}_\sppi\coloneqq\lim_{t\searrow 0}\frac{\e_t-\e_0}{t}=
\iota_\mu(\nabla_\mu f)\circ\e_0,\quad\text{ strongly in }
\mathbb B_\sppi.
\end{equation}
\end{theorem}
\begin{proof}
Fix any sequence \(t_i\searrow 0\). Observe that for every \(t\in(0,1)\)
we have that
\begin{equation}\label{eq:bound_Der}
\bigg|\frac{\e_t-\e_0}{t}\bigg|(\gamma)\leq
\fint_0^t|\dot\gamma_s|\,\d s\leq
\bigg(\fint_0^t|\dot\gamma_s|^2\,\d s\bigg)^{1/2}=
\frac{{\rm KE}_t(\gamma)}{t},\quad\text{ for }\ppi\text{-a.e.\ }\gamma.
\end{equation}
Since \(\big({\rm KE}_{t_i}/t_i\big)_i\) is convergent in
\(L^2(\ppi)\), we deduce that \(\big((\e_{t_i}-\e_0)/t_i\big)_i\)
is bounded in \(\mathbb B_\sppi\), thus accordingly (up to a not relabelled
subsequence) it holds that
\((\e_{t_i}-\e_0)/t_i\rightharpoonup\ell\) weakly
in \(\mathbb B_\sppi\) for some \(\ell\in\mathbb B_\sppi\). Given
\(\underline v\in L^2(\R^n,\R^n;\mu)\) and \(E\subseteq C([0,1],\R^n)\)
Borel, we claim that
\begin{equation}\label{eq:Der_0_aux}
\int_E(\underline v\circ\e_0)\cdot\ell\,\d\ppi=
\lim_{i\to\infty}\int_E\fint_0^{t_i}(\underline v\circ\e_t)
\cdot{\rm Der}_t\,\d t\,\d\ppi.
\end{equation}
In order to prove it, observe that
\[\begin{split}
\int_E(\underline v\circ\e_0)\cdot\ell\,\d\ppi
=\lim_{i\to\infty}\int_E(\underline v\circ\e_0)\cdot\frac{\e_{t_i}-\e_0}{t_i}
\,\d\ppi\overset{\eqref{eq:der_e_t}}=\lim_{i\to\infty}\int_E
\fint_0^{t_i}(\underline v\circ\e_0)\cdot{\rm Der}_t\,\d t\,\d\ppi
\end{split}\]
and that by exploiting Lemma \ref{lem:cont_comp_e_t} we obtain that
\[\begin{split}
&\lims_{i\to\infty}\bigg|\int_E\fint_0^{t_i}(\underline v\circ\e_t)
\cdot{\rm Der}_t\,\d t\,\d\ppi-\int_E\fint_0^{t_i}(\underline v\circ\e_0)
\cdot{\rm Der}_t\,\d t\,\d\ppi\bigg|\\
\leq\,&\lims_{i\to\infty}\int_E\fint_0^{t_i}\big|\underline v\circ\e_t-
\underline v\circ\e_0\big||{\rm Der}_t|\,\d t\,\d\ppi\\
\leq\,&\lim_{i\to\infty}\bigg(\fint_0^{t_i}\big\|\underline v\circ\e_t-
\underline v\circ\e_0\big\|_{\mathbb B_\sppi}^2\,\d t\bigg)^{1/2}
\bigg(\int\frac{{\rm KE}_{t_i}^2}{t_i^2}\,\d\ppi\bigg)^{1/2}=0.
\end{split}\]
Since one has \({\rm Der}_t(\gamma)\in T_\mu(\gamma_t)\) for
\((\ppi\otimes\mathcal L_1)\)-a.e.\ \((\gamma,t)\) by Lemma
\ref{lem:speed_pi_tangent_MOD}, we deduce from \eqref{eq:Der_0_aux} that
\(\int_E(\underline v\circ\e_0)\cdot\ell\,\d\ppi=0\) for every
\(\underline v\in\Gamma(T_\mu^\perp)\) and \(E\subseteq C([0,1],\R^n)\)
Borel, thus accordingly
\begin{equation}\label{eq:Der_0_aux2}
\ell(\gamma)\in T_\mu(\gamma_0),\quad\text{ for }\ppi\text{-a.e.\ }\gamma.
\end{equation}
Also, given \(g\in C^\infty_c(\R^n)\) and \(E\subseteq C([0,1],\R^n)\)
Borel, we know from Proposition \ref{prop:plan_repr_grad_prod} that
\[\begin{split}
&\int_E\big(\iota_\mu(\nabla_\mu g)\circ\e_0\big)\cdot
\big(\iota_\mu(\nabla_\mu f)\circ\e_0\big)\,\d\ppi\\
\overset{\phantom{\eqref{eq:explicit_iota_nabla_mu}}}=\,&\int_E
\langle\nabla_\mu g,\nabla_\mu f\rangle\circ\e_0\,\d\ppi=\lim_{i\to\infty}
\int_E\frac{g\circ\e_{t_i}-g\circ\e_0}{t_i}\,\d\ppi\\ 
\overset{\phantom{\eqref{eq:explicit_iota_nabla_mu}}}=\,&\lim_{i\to\infty}
\int_E\fint_0^{t_i}\frac{\d}{\d t}\,g(\gamma_t)\,\d t\,\d\ppi(\gamma)=
\lim_{i\to\infty}\int_E\fint_0^{t_i}(\nabla g\circ\e_t)
\cdot{\rm Der}_t\,\d t\,\d\ppi\\
\overset{\phantom{\eqref{eq:explicit_iota_nabla_mu}}}=\,&
\lim_{i\to\infty}\int_E\fint_0^{t_i}\big({\rm pr}_{T_\mu}(\nabla g)\circ\e_t\big)
\cdot{\rm Der}_t\,\d t\,\d\ppi\\
\overset{\eqref{eq:explicit_iota_nabla_mu}}=\,&
\lim_{i\to\infty}\int_E\fint_0^{t_i}\big(\iota_\mu(\nabla_\mu g)\circ\e_t\big)
\cdot{\rm Der}_t\,\d t\,\d\ppi\overset{\eqref{eq:Der_0_aux}}=
\int_E\big(\iota_\mu(\nabla_\mu g)\circ\e_0\big)\cdot\ell\,\d\ppi,
\end{split}\]
whence it follows that \(\big(\iota_\mu(\nabla_\mu g)\circ\e_0\big)
\cdot\big(\ell-\iota_\mu(\nabla_\mu f)\circ\e_0\big)=0\) holds
\(\ppi\)-a.e.. By using \eqref{eq:Der_0_aux2} and the arbitrariness of
\(g\in C^\infty_c(\R^n)\), we get \(\ell=\iota_\mu(\nabla_\mu f)\circ\e_0\).
Being the limit \(\ell\) independent of the sequence \((t_i)_i\),
we deduce that
\begin{equation}\label{eq:Der_0_aux3}
\frac{\e_t-\e_0}{t}\rightharpoonup\iota_\mu(\nabla_\mu f)\circ\e_0,
\quad\text{ weakly in }\mathbb B_\sppi\text{ as }t\searrow 0.
\end{equation}
Finally, let us observe that
\[\begin{split}
\int|D_\mu f|^2\circ\e_0\,\d\ppi
&\overset{\phantom{\eqref{eq:bound_Der}}}=
\int\big|\iota_\mu(\nabla_\mu f)\big|^2\circ\e_0\,\d\ppi
\overset{\eqref{eq:Der_0_aux3}}\leq
\limi_{t\searrow 0}\int\bigg|\frac{\e_t-\e_0}{t}\bigg|^2\,\d\ppi
\leq\lims_{t\searrow 0}\int\bigg|\frac{\e_t-\e_0}{t}\bigg|^2\,\d\ppi\\
&\overset{\eqref{eq:bound_Der}}\leq\lim_{t\searrow 0}\int\frac{{\rm KE}_t^2}
{t^2}\,\d\ppi=\int|D_\mu f|^2\circ\e_0\,\d\ppi.
\end{split}\]
This shows that \(\int\big|\iota_\mu(\nabla_\mu f)\big|^2\circ\e_0\,\d\ppi
=\lim_{t\searrow 0}\int\big|\frac{\e_t-\e_0}{t}\big|^2\,\d\ppi\), which
together with \eqref{eq:Der_0_aux3} grant that
\(\frac{\e_t-\e_0}{t}\to\iota_\mu(\nabla_\mu f)\circ\e_0\) strongly
in \(\mathbb B_\sppi\) as \(t\searrow 0\), thus proving the statement.
\end{proof}
By building upon Theorem \ref{thm:deriv_time_0_tp},
we can eventually prove the main result of this section. 
\begin{theorem}[Geometric characterisation of the tangent fibers]
\label{thm:fiber_cl_dot_pi}
Let \(\mu\geq 0\) be a Radon measure on \(\R^n\). Then there exists
a sequence \((\ppi_i)_i\) of test plans on \((\R^n,\sfd_{\rm Eucl},\mu)\)
such that the limits \({\rm D}_{\sppi_i}\) exist as in \eqref{eq:def_D_pi},
the property \(\mu\ll(\e_0)_*\ppi_i\ll\mu\) holds for every \(i\in\N\), and
\begin{equation}\label{eq:fiber_cl_dot_pi_claim}
T_\mu(x)={\rm cl}\,\Big\{{\rm Im}_{\e_0,\sppi_i}({\rm D}_{\sppi_i})(x)
\;\Big|\;i\in\N\Big\},\quad\text{ for }\mu\text{-a.e.\ }x\in\R^n,
\end{equation}
where the \textbf{essential image}
\({\rm Im}_{\e_0,\sppi_i}({\rm D}_{\sppi_i})\colon C([0,1],\R^n)\to\R^n\)
of \({\rm D}_{\sppi_i}\) under \(\e_0\) is defined as
\[
{\rm Im}_{\e_0,\sppi_i}({\rm D}_{\sppi_i})\coloneqq
\frac{\d(\e_0)_*({\rm D}_{\sppi_i}\ppi_i)}{\d(\e_0)_*\ppi_i},
\quad\text{ for every }i\in\N.
\]
\end{theorem}
\begin{proof}
Given that \(C^\infty_c(\R^n)\) is strongly dense in \(W^{1,2}(\R^n,\mu)\)
by Corollary \ref{cor:strong_dens_smooth}, we can find a countable
\(\mathbb Q\)-linear subspace \((f_i)_i\) of \(C^\infty_c(\R^n)\) that
is dense in \(W^{1,2}(\R^n,\mu)\). In particular, the family
\(\mathscr V\coloneqq\big\{\sum_{j=1}^k g_j\nabla_\mu f_{i_j}\,:\,
k\in\N,\,(g_j)_{j=1}^k\subseteq L^\infty(\mu),\,(i_j)_{j=1}^k\subseteq\N\big\}\)
is dense in \(L^2_\mu(T\R^n)\), thus the linear space
\(\iota_\mu(\mathscr V)\) is dense in \(\Gamma(T_\mu)\).
By using Lemma \ref{lem:stable_spaces}, we can deduce that
\begin{equation}\label{eq:fiber_cl_dot_pi_aux}
T_\mu(x)={\rm cl}\,\big\{\iota_\mu(\nabla_\mu f_i)(x)\;\big|\;
i\in\N\big\},\quad\text{ for }\mu\text{-a.e.\ }x\in\R^n.
\end{equation}
It is straightforward to check that one can find a Borel probability
measure \(\nu\) on \(\R^n\) such that \(\int|x|^2\,\d\nu(x)<+\infty\)
and \(\mu\ll\nu\leq C\mu\) for some \(C>0\). Given any \(i\in\N\),
we know from Theorem \ref{thm:existence_plan_repr_grad} that there
exists a test plan \(\ppi_i\) on \((\R^n,\sfd_{\rm Eucl},\mu)\)
representing the gradient of \(f_i\) and satisfying \((\e_0)_*\ppi_i=\nu\). 
Theorem \ref{thm:deriv_time_0_tp} grants that \({\rm D}_{\sppi_i}\) exists
as in \eqref{eq:def_D_pi}. Also, it holds
\[\begin{split}
{\rm Im}_{\e_0,\sppi_i}({\rm D}_{\sppi_i})&=
{\rm Im}_{\e_0,\sppi_i}\big(\iota_\mu(\nabla_\mu f_i)\circ\e_0\big)=
\frac{\d(\e_0)_*\big(\iota_\mu(\nabla_\mu f_i)\circ\e_0\,\ppi_i\big)}
{\d\nu}=\frac{\d\big(\iota_\mu(\nabla_\mu f_i)\,\nu\big)}{\d\nu}\\
&=\iota_\mu(\nabla_\mu f_i).
\end{split}\]
By taking \eqref{eq:fiber_cl_dot_pi_aux} into account, we eventually
obtain \eqref{eq:fiber_cl_dot_pi_claim}, as desired.
\end{proof}
\subsection{Tensorisation of the Cheeger energy on weighted
Euclidean spaces}\label{ss:tens_Ch}
In the framework of Sobolev calculus on metric measure spaces,
a surprisingly difficult problem is the following: given two
metric measure spaces \((\X,\sfd_\X,\mu)\) and \((\Y,\sfd_\Y,\nu)\),
is the Sobolev space on the product space
\((\X\times\Y,\sfd_{\X\times\Y},\mu\otimes\nu)\) the
\emph{tensorisation} of \(W^{1,2}(\X,\mu)\) and \(W^{1,2}(\Y,\nu)\)?
\medskip

The precise statement would read as follows: given any function
\(f\in W^{1,2}(\X\times\Y,\mu\otimes\nu)\), it holds for
\((\mu\otimes\nu)\)-a.e.\ \((x,y)\in\X\times\Y\) that
\(f^{(y)}\in W^{1,2}(\X,\mu)\), \(f_{(x)}\in W^{1,2}(\Y,\nu)\), and
\[
|D_{\mu\otimes\nu}f|^2(x,y)=|D_\mu f^{(y)}|^2(x)
+|D_\nu f_{(x)}|^2(y),
\]
where we set \(f^{(y)}(x)=f_{(x)}(y)\coloneqq f(x,y)\).
(Here, Fubini theorem plays a role.)
\medskip

A positive answer to the above question is known only in some
particular circumstances. About the spaces having such
tensorisation property, this is the current state of the art:
\begin{itemize}
\item[\(\rm a)\)] L.\ Ambrosio, N.\ Gigli, and G.\ Savar\'{e}
proved in \cite{AmbrosioGigliSavare11-2} that
\({\sf RCD}(K,\infty)\) spaces, for any given \(K\in\R\),
have the tensorisation property.
\item[\(\rm b)\)] L.\ Ambrosio, A.\ Pinamonti, and G.\ Speight
proved in \cite{APS14} the tensorisation property on doubling metric
measure spaces supporting a weak \((1,2)\)-Poincar\'{e} inequality.
\item[\(\rm c)\)] N.\ Gigli and B.-X.\ Han showed in \cite{GH15} that
the Sobolev space tensorises as soon as one of the two factors
is a closed real interval \(I\subseteq\R\).
\end{itemize}
To the best of our knowledge, these are all the cases that
have been studied so far. The aim of this section is to prove
that weighted Euclidean spaces have the tensorisation property
(cf.\ Theorem \ref{thm:tens_Sob}), and we do so by first
showing that the fibers of the tangent distribution
`tensorise' as well (cf.\ Proposition \ref{prop:tens_t}).
Notice that the family of all weighted Euclidean spaces
is not contained in any of the classes of spaces described
in items a), b), and c) above.
\subsubsection{Test plans on product spaces}
Let \((\X,\sfd_\X,\mu)\), \((\Y,\sfd_\Y,\nu)\) be two given metric
measure spaces. The cartesian product \(\X\times\Y\) will be
implicitly endowed with the product distance
\[
\sfd_{\X\times\Y}\big((x,y),(x',y')\big)\coloneqq
\sqrt{\sfd_\X(x,x')^2+\sfd_\Y(y,y')^2},\quad
\text{ for every }(x,y),(x',y')\in\X\times\Y,
\]
and the product measure \(\mu\otimes\nu\). We denote by
\(p^\X\colon\X\times\Y\to\X\) and \(p^\Y\colon\X\times\Y\to\Y\)
the canonical projection maps \(p^\X(x,y)\coloneqq x\) and
\(p^\Y(x,y)\coloneqq y\). They induce the \(1\)-Lipschitz maps
\[\begin{split}
\boldsymbol p^\X\colon C([0,1],\X\times\Y)\longrightarrow C([0,1],\X),&
\quad\boldsymbol p^\X(\gamma)\coloneqq p^\X\circ\gamma,\\
\boldsymbol p^\Y\colon C([0,1],\X\times\Y)\longrightarrow C([0,1],\Y),&
\quad\boldsymbol p^\Y(\gamma)\coloneqq p^\Y\circ\gamma.
\end{split}\]
It can be readily checked that
\[\begin{split}
&\boldsymbol p^\X\big(AC^2([0,1],\X\times\Y)\big)\subseteq AC^2([0,1],\X),\\
&\boldsymbol p^\Y\big(AC^2([0,1],\X\times\Y)\big)\subseteq AC^2([0,1],\Y).
\end{split}\]
Moreover, let us consider the joint mapping
\begin{equation}\label{eq:def_j}\begin{split}
(\boldsymbol p^\X,\boldsymbol p^\Y)\colon C([0,1],\X\times\Y)&\longrightarrow
C([0,1],\X)\times C([0,1],\Y),\\
\gamma&\longmapsto\big(\boldsymbol p^\X(\gamma),\boldsymbol p^\Y(\gamma)\big).
\end{split}\end{equation}
It turns out that \((\boldsymbol p^\X,\boldsymbol p^\Y)\) is a
\(\sqrt 2\)-Lipschitz bijection whose inverse is \(1\)-Lipschitz. Also,
\begin{equation}\label{eq:beh_AC_under_j}
(\boldsymbol p^\X,\boldsymbol p^\Y)\big(AC^2([0,1],\X\times\Y)\big)=
AC^2([0,1],\X)\times AC^2([0,1],\Y).
\end{equation}
More precisely, given any curve
\(\gamma=(\gamma^\X,\gamma^\Y)\in AC^2([0,1],\X\times\Y)\), it holds that
\begin{equation}\label{eq:tens_ms}
|\dot\gamma_t|^2=|\dot\gamma^\X_t|^2+|\dot\gamma^\Y_t|^2,
\quad\text{ for }\mathcal L_1\text{-a.e.\ }t\in[0,1].
\end{equation}
For completeness, we report below the elementary proofs
of the following two technical results:
\begin{lemma}\label{lem:proj_test_plan}
Let \((\X,\sfd_\X,\mu)\), \((\Y,\sfd_\Y,\nu)\) be metric measure
spaces such that \(\mu\), \(\nu\) are finite Borel measures.
Let \(\ppi\) be a given test plan on
\((\X\times\Y,\sfd_{\X\times\Y},\mu\otimes\nu)\).
Then \(\ppi_\X\coloneqq\boldsymbol p^\X_*\ppi\) is a test plan on
\((\X,\sfd_\X,\mu)\) and \(\ppi_\Y\coloneqq\boldsymbol p^\Y_*\ppi\)
is a test plan on \((\Y,\sfd_\Y,\nu)\). Moreover, it holds that
\[
{\rm Comp}(\ppi_\X)\leq{\rm Comp}(\ppi)\,\nu(\Y),
\quad{\rm Comp}(\ppi_\Y)\leq{\rm Comp}(\ppi)\,\mu(\X).
\]
\end{lemma}
\begin{proof}
By symmetry, it suffices to prove the statement just for \(\ppi_\X\).
Since \(\ppi\) is concentrated on
\(AC^2([0,1],\X\times\Y)\), we have that \(\ppi_\X\)
is concentrated on \(AC^2([0,1],\X)\).
Moreover, for any curve \(\gamma=(\gamma^\X,\gamma^\Y)\in
AC^2([0,1],\X\times\Y)\) it holds
\(|\dot\gamma^\X_t|\leq|\dot\gamma_t|\) for
\(\mathcal L_1\)-a.e.\ \(t\in[0,1]\), whence
\[
\int\!\!\!\int_0^1|\dot\gamma^\X_t|^2\,\d t\,\d\ppi_\X(\gamma^\X)
=\int\!\!\!\int_0^1|\dot\gamma^\X_t|^2\,\d t\,\d\ppi(\gamma^\X,\gamma^\Y)
\leq\int\!\!\!\int_0^1|\dot\gamma_t|^2\,\d t\,\d\ppi(\gamma)<+\infty.
\]
Finally, for any Borel set \(A\subseteq\X\) we have that
\[\begin{split}
(\e^\X_t)_*\ppi_\X(A)=\ppi_\X\big((\e^\X_t)^{-1}(A)\big)&=
\ppi\big((\e^{\X\times\Y}_t)^{-1}(A\times\Y)\big)=
(\e^{\X\times\Y}_t)_*\ppi(A\times\Y)\\
&\leq{\rm Comp}(\ppi)\,(\mu\otimes\nu)(A\times\Y)
={\rm Comp}(\ppi)\,\nu(\Y)\,\mu(A),
\end{split}\]
for all \(t\in[0,1]\). Hence, \(\ppi_\X\) is a test plan on
\((\X,\sfd_\X,\mu)\) and \({\rm Comp}(\ppi_\X)\leq{\rm Comp}(\ppi)\,\nu(\Y)\).
\end{proof}
\begin{lemma}\label{lem:prod_test_plan}
Let \((\X,\sfd_\X,\mu)\), \((\Y,\sfd_\Y,\nu)\) be metric measure
spaces. Let \(\ppi_\X\) and \(\ppi_\Y\) be test plans on
\((\X,\sfd_\X,\mu)\) and \((\Y,\sfd_\Y,\nu)\), respectively.
Then \(\ppi\coloneqq(\boldsymbol p^\X,\boldsymbol p^\Y)^{-1}_*
(\ppi_\X\otimes\ppi_\Y)\) is a test plan on
\((\X\times\Y,\sfd_{\X\times\Y},\mu\otimes\nu)\). Moreover, it holds
that \({\rm Comp}(\ppi)\leq{\rm Comp}(\ppi_\X)\,{\rm Comp}(\ppi_\Y)\).
\end{lemma}
\begin{proof}
We know from \eqref{eq:beh_AC_under_j} that \(\ppi\) is concentrated
on \(AC^2([0,1],\X\times\Y)\), while \eqref{eq:tens_ms} yields
\[
\int\!\!\!\int_0^1|\dot\gamma_t|^2\,\d t\,\d\ppi(\gamma)
=\int\!\!\!\int_0^1|\dot\gamma^\X_t|^2\,\d t\,\d\ppi_\X(\gamma^\X)
+\int\!\!\!\int_0^1|\dot\gamma^\Y_t|^2\,\d t\,\d\ppi_\Y(\gamma^\Y)<+\infty.
\]
Moreover, given any non-negative Borel function \(f\) on \(\X\times\Y\),
for every \(t\in[0,1]\) it holds that
\[\begin{split}
\int f\,\d(\e^{\X\times\Y}_t)_*\ppi&=
\int f(\gamma_t)\,\d\ppi(\gamma)=
\int\!\!\!\int f(\gamma^\X_t,\gamma^\Y_t)\,\d\ppi_\X(\gamma^\X)
\,\d\ppi_\Y(\gamma^\Y)\\
&=\int\!\!\!\int f(x,y)\,\d(\e^\X_t)_*\ppi_\X(x)\,\d(\e^\Y_t)_*\ppi_\Y(y)\\
&\leq{\rm Comp}(\ppi_\X)\,{\rm Comp}(\ppi_\Y)\int\!\!\!\int f(x,y)\,\d\mu(x)\,\d\nu(y)\\
&={\rm Comp}(\ppi_\X)\,{\rm Comp}(\ppi_\Y)\int f\,\d(\mu\otimes\nu),
\end{split}\]
whence \((\e^{\X\times\Y}_t)_*\ppi\leq
{\rm Comp}(\ppi_\X)\,{\rm Comp}(\ppi_\Y)\,\mu\otimes\nu\).
This proves the statement.
\end{proof}
\subsubsection{Tensorisation of the tangent distribution}
Let us denote by \(p^n\) and \(p^m\) the canonical projections of
the product \(\R^{n+m}\cong\R^n\times\R^m\) onto \(\R^n\) and
\(\R^m\), respectively, instead of \(p^{\R^n}\) and \(p^{\R^m}\).
Also, we define the embedding maps
\(\iota^n\colon\R^n\to\R^{n+m}\) and \(\iota^m\colon\R^m\to\R^{n+m}\) as
\[\begin{split}
\iota^n(v)&\coloneqq(v,0)\in\R^n\times\R^m,\quad\text{ for every }v\in\R^n,\\
\iota^m(w)&\coloneqq(0,w)\in\R^n\times\R^m,\quad\text{ for every }w\in\R^m.
\end{split}\]
\begin{proposition}[Tangent distribution on the product space]
\label{prop:tens_t}
Let \(\mu\) and \(\nu\) be finite Borel measures on \(\R^n\)
and \(\R^m\), respectively. Then it holds that
\[
T_{\mu\otimes\nu}(x,y)=\iota^n\big(T_\mu(x)\big)\oplus
\iota^m\big(T_\nu(y)\big),\quad\text{ for }(\mu\otimes\nu)\text{-a.e.\ }
(x,y)\in\R^{n+m}.
\]
\end{proposition}
\begin{proof}
Let us define \(S(x,y)\coloneqq
\iota^n\big(T_\mu(x)\big)\oplus\iota^m\big(T_\nu(y)\big)\)
for \((\mu\otimes\nu)\)-a.e.\ \((x,y)\in\R^{n+m}\). It is
straightforward to check that \(S\in \mathscr D_{n+m}(\mu\otimes\nu)\).
To prove the statement amounts to showing that \(T_{\mu\otimes\nu}=S\).
First, let us prove that \(T_{\mu\otimes\nu}\leq S\). In light of
item i) of Theorem \ref{thm:alt_char_T_mu}, this is equivalent to saying that
for any test plan \(\ppi\) on \((\R^{n+m},\sfd_{\rm Eucl},\mu\otimes\nu)\)
it holds that
\begin{equation}\label{eq:tens_T_aux}
\dot\gamma_t\in S(\gamma_t),\quad\text{ for }(\ppi\otimes\mathcal L_1)
\text{-a.e.\ }(\gamma,t)\in AC^2([0,1],\R^{n+m})\times[0,1].
\end{equation}
Call \(\ppi_n\coloneqq\boldsymbol p^n_*\ppi\) and
\(\ppi_m\coloneqq\boldsymbol p^m_*\ppi\). We know from Lemma 
\ref{lem:proj_test_plan} that \(\ppi_n\) and \(\ppi_m\) are test
plans on \((\R^n,\sfd_{\rm Eucl},\mu)\) and
\((\R^m,\sfd_{\rm Eucl},\nu)\), respectively. Hence,
item i) of Theorem \ref{thm:alt_char_T_mu} gives
\[\begin{split}
\dot\gamma^n_t\in T_\mu(\gamma^n_t),&\quad\text{ for }
(\ppi_n\otimes\mathcal L_1)\text{-a.e.\ }(\gamma^n,t)\in
AC^2([0,1],\R^n)\times[0,1],\\
\dot\gamma^m_t\in T_\nu(\gamma^m_t),&\quad\text{ for }
(\ppi_m\otimes\mathcal L_1)\text{-a.e.\ }(\gamma^m,t)\in
AC^2([0,1],\R^m)\times[0,1],
\end{split}\]
which can be equivalently restated as follows: for
\(\ppi\)-a.e.\ \(\gamma=(\gamma^n,\gamma^m)\in AC^2([0,1],\R^{n+m})\)
it holds \((\dot\gamma^n_t,\dot\gamma^m_t)\in T_\mu(\gamma^n_t)
\times T_\nu(\gamma^m_t)\) for \(\mathcal L_1\)-a.e.\ \(t\in[0,1]\).
This proves \eqref{eq:tens_T_aux}, whence \(T_{\mu\otimes\nu}\leq S\).

In order to prove that \(S\leq T_{\mu\otimes\nu}\),
it is clearly sufficient to show that
\(T_\mu(x)\subseteq p^n\big(T_{\mu\otimes\nu}(x,y)\big)\)
and \(T_\nu(y)\subseteq p^m\big(T_{\mu\otimes\nu}(x,y)\big)\)
hold for \((\mu\otimes\nu)\)-a.e.\ \((x,y)\in\R^{n+m}\).
Let us just prove the former inclusion, since the latter one
can be obtained by an analogous argument. Trivially, we have that
\(p^n\big(T_{\mu\otimes\nu}(\cdot,y)\big)\in\mathscr D_n(\mu)\)
for \(\nu\)-a.e.\ \(y\in\R^m\).
Now fix a test plan \(\ppi_n\) on \((\R^n,\sfd_{\rm Eucl},\mu)\).
We then define the measure \(\ppi\) on \(C([0,1],\R^{n+m})\) as
\[
\ppi\coloneqq({\boldsymbol p}^n,{\boldsymbol p}^m)^{-1}
_*(\ppi_n\otimes{\rm Const}^m_*\nu),
\]
where the map \({\rm Const}^m\coloneqq{\rm Const}^{\R^m}\)
is defined as in Example \ref{ex:Const}.
Lemma \ref{lem:prod_test_plan} grants that
\(\ppi\) is a test plan on \((\R^{n+m},\sfd_{\rm Eucl},\mu\otimes\nu)\),
thus item i) of Theorem \ref{thm:alt_char_T_mu} ensures that
\(\dot\gamma_t\in T_{\mu\otimes\nu}(\gamma_t)\) is satisfied for
\((\ppi\otimes\mathcal L_1)\)-a.e.\ \((\gamma,t)\in
AC^2([0,1],\R^{n+m})\times[0,1]\). This can be rewritten as
\[
(\dot\gamma^n_t,0)\in T_{\mu\otimes\nu}(\gamma^n_t,y),
\quad\text{ for }(\ppi_n\otimes\nu\otimes\mathcal L_1)
\text{-a.e.\ }(\gamma^n,y,t)\in AC^2([0,1],\R^n)\times\R^m\times[0,1].
\]
Therefore, by arbitrariness of \(\ppi_n\) we can finally conclude
that \(T_\mu(x)\subseteq p^n\big(T_{\mu\otimes\nu}(x,y)\big)\)
holds for \((\mu\otimes\nu)\)-a.e.\ \((x,y)\in\R^{n+m}\), whence
the proof of the statement is complete.
\end{proof}
\begin{remark}{\rm
Proposition \ref{prop:tens_t} is claimed in
\cite[Remark 2.2(iv)]{BBS97}. Therein, the tangent distribution
is defined in terms of the distributional divergence,
an approach that is equivalent to ours in view of item ii) of
Theorem \ref{thm:alt_char_T_mu}.
\fr}\end{remark}
\subsubsection{Tensorisation of the Sobolev space}
We are in a position -- by exploiting Propositions
\ref{prop:tens_t} and \ref{prop:char_minimal_Gmu_grad} -- to prove
that weighted Euclidean spaces have the tensorisation property.
\medskip

Given a Borel function \(f\colon\R^{n+m}\to\R\), we define
\(f^{(y)}\colon\R^n\to\R\) and \(f_{(x)}\colon\R^m\to\R\) as
\[
f^{(y)}(x)=f_{(x)}(y)\coloneqq f(x,y),
\quad\text{ for every }(x,y)\in\R^{n+m}.
\]
Observe that \(f^{(y)}\) and \(f_{(x)}\) are Borel functions as well.
Also, thanks to Fubini theorem, for every \(f\in L^2(\mu\otimes\nu)\)
we have \(f^{(y)}\in L^2(\mu)\) for \(\nu\)-a.e.\ \(y\in\R^m\)
and \(f_{(x)}\in L^2(\nu)\) for \(\mu\)-a.e.\ \(x\in\R^n\).
\begin{theorem}[Tensorisation of the Sobolev space on weighted \(\R^n\)]
\label{thm:tens_Sob}
Let \(\mu\) and \(\nu\) be finite Borel measures on \(\R^n\)
and \(\R^m\), respectively. Let
\(f\in W^{1,2}(\R^{n+m},\mu\otimes\nu)\) be given. Then
\begin{equation}\label{eq:tens_Sob_claim}\begin{split}
f^{(y)}\in W^{1,2}(\R^n,\mu),&\quad\text{ for }\nu\text{-a.e.\ }y\in\R^m,\\
f_{(x)}\in W^{1,2}(\R^m,\nu),&\quad\text{ for }\mu\text{-a.e.\ }x\in\R^n.
\end{split}\end{equation}
Moreover, it holds that
\begin{equation}\label{eq:tens_Sob_claim1}
|D_{\mu\otimes\nu}f|^2(x,y)=|D_\mu f^{(y)}|^2(x)+|D_\nu f_{(x)}|^2(y),
\quad\text{ for }(\mu\otimes\nu)\text{-a.e.\ }(x,y)\in\R^{n+m}.
\end{equation}
\end{theorem}
\begin{proof}
First of all, we know from Proposition \ref{prop:char_minimal_Gmu_grad}
that \(G_{\mu\otimes\nu}(f)\neq\emptyset\) and
\(\iota_{\mu\otimes\nu}(\nabla_{\mu\otimes\nu}f)\) is the minimal
\(G_{\mu\otimes\nu}\)-gradient of \(f\). We can choose a sequence
\((f_i)_i\subseteq C^\infty_c(\R^{n+m})\) such that \(f_i\to f\)
in \(L^2(\mu\otimes\nu)\) and
\(\nabla f_i\to\iota_{\mu\otimes\nu}(\nabla_{\mu\otimes\nu}f)\)
in \(L^2(\R^{n+m},\R^{n+m};\mu\otimes\nu)\). Notice that
\((f_i)^{(y)}\in C^\infty_c(\R^n)\) and \((f_i)_{(x)}\in C^\infty_c(\R^m)\)
for every \(i\in\N\) and \((x,y)\in\R^{n+m}\). Proposition
\ref{prop:tens_t} grants that
\begin{equation}\label{eq:tens_Sob_aux}
{\rm pr}_{T_{\mu\otimes\nu}}(\nabla f_i)(x,y)
=\Big({\rm pr}_{T_\mu}\big(\nabla(f_i)^{(y)}\big)(x),
{\rm pr}_{T_\nu}\big(\nabla(f_i)_{(x)}\big)(y)\Big),
\end{equation}
for \((\mu\otimes\nu)\)-a.e.\ \((x,y)\in\R^{n+m}\).
Recalling Proposition \ref{prop:char_minimal_Gmu_grad},
we deduce from \eqref{eq:tens_Sob_aux} that
\begin{equation}\label{eq:tens_Sob_aux1}
|D_{\mu\otimes\nu}f_i|^2(x,y)=\big|D_\mu(f_i)^{(y)}\big|^2(x)
+\big|D_\nu(f_i)_{(x)}\big|^2(y),\quad\text{ for }
(\mu\otimes\nu)\text{-a.e.\ }(x,y)\in\R^{n+m}.
\end{equation}
Thanks to Fubini theorem, we have (up to a not relabelled subsequence) that
\begin{equation}\label{eq:tens_Sob_aux2}\begin{split}
(f_i)^{(y)}\longrightarrow f^{(y)},&\quad\text{ strongly in }
L^2(\mu)\text{ for }\nu\text{-a.e.\ }y\in\R^m,\\
(f_i)_{(x)}\longrightarrow f_{(x)},&\quad\text{ strongly in }
L^2(\nu)\text{ for }\mu\text{-a.e.\ }x\in\R^n.
\end{split}\end{equation}
Call \(\underline v\coloneqq\iota_{\mu\otimes\nu}
(\nabla_{\mu\otimes\nu}f)\in\Gamma(T_{\mu\otimes\nu})\).
We have that
\({\rm pr}_{T_{\mu\otimes\nu}}(\nabla f_i)\to\underline v\)
in \(L^2(\R^{n+m},\R^{n+m};\mu\otimes\nu)\), so that (up to
passing to a further subsequence) it holds that
\begin{equation}\label{eq:tens_Sob_aux2_bis}
{\rm pr}_{T_{\mu\otimes\nu}}(\nabla f_i)(x,y)\longrightarrow
\underline v(x,y),\quad\text{ for }(\mu\otimes\nu)
\text{-a.e.\ }(x,y)\in\R^{n+m},
\end{equation}
thus in particular
\begin{equation}\label{eq:tens_Sob_aux3}
|D_{\mu\otimes\nu}f_i|\overset{\eqref{eq:explicit_iota_nabla_mu}}=
\big|{\rm pr}_{T_{\mu\otimes\nu}}(\nabla f_i)
\big|\longrightarrow|\underline v|=|D_{\mu\otimes\nu}f|,
\quad\text{ in the }(\mu\otimes\nu)\text{-a.e.\ sense.}
\end{equation}
Set
\(\underline v^{(y)}(x)\coloneqq p^n\big(\underline v(x,y)\big)
\in T_\mu(x)\) and \(\underline v_{(x)}(y)\coloneqq p^m
\big(\underline v(x,y)\big)\in T_\nu(y)\) for
\((\mu\otimes\nu)\)-a.e.\ \((x,y)\). Therefore,
for \((\mu\otimes\nu)\)-a.e.\ \((x,y)\in\R^{n+m}\) it holds that
\[\begin{split}
\iota_\mu\big(\nabla_\mu(f_i)^{(y)}\big)
\overset{\eqref{eq:explicit_iota_nabla_mu}}=
{\rm pr}_{T_\mu}\big(\nabla(f_i)^{(y)}\big)\longrightarrow
\underline v^{(y)},&\quad\text{ strongly in }L^2(\R^n,\R^n;\mu),\\
\iota_\nu\big(\nabla_\nu(f_i)_{(x)}\big)
\overset{\eqref{eq:explicit_iota_nabla_mu}}=
{\rm pr}_{T_\nu}\big(\nabla(f_i)_{(x)}\big)\longrightarrow
\underline v_{(x)},&\quad\text{ strongly in }L^2(\R^m,\R^m;\nu).
\end{split}\]
This implies \(\underline v^{(y)}\in\iota_\mu\big(L^2_\mu(T\R^n)\big)\)
and \(\underline v_{(x)}\in\iota_\nu\big(L^2_\nu(T\R^m)\big)\) for
\((\mu\otimes\nu)\)-a.e.\ \((x,y)\in\R^{n+m}\), and
\begin{equation}\label{eq:tens_Sob_aux4}\begin{split}
\nabla_\mu(f_i)^{(y)}\longrightarrow
\iota_\mu^{-1}(\underline v^{(y)})\eqqcolon v^{(y)},&
\quad\text{ strongly in }L^2_\mu(T\R^n),\\
\nabla_\nu(f_i)_{(x)}\longrightarrow
\iota_\nu^{-1}(\underline v_{(x)})\eqqcolon v_{(x)},&
\quad\text{ strongly in }L^2_\nu(T\R^m).
\end{split}\end{equation}
Moreover, it follows from \eqref{eq:tens_Sob_aux}
and \eqref{eq:tens_Sob_aux2_bis} that for
\((\mu\otimes\nu)\)-a.e.\ \((x,y)\in\R^{n+m}\) it holds that
\begin{equation}\label{eq:tens_Sob_aux5}\begin{split}
\big|D_\mu(f_i)^{(y)}\big|
\overset{\eqref{eq:explicit_iota_nabla_mu}}=
\big|{\rm pr}_{T_\mu}\big(\nabla(f_i)^{(y)}\big)\big|
\longrightarrow|\underline v^{(y)}|=|v^{(y)}|,&
\quad\text{ in the }\mu\text{-a.e.\ sense,}\\
\big|D_\nu(f_i)_{(x)}\big|
\overset{\eqref{eq:explicit_iota_nabla_mu}}=
\big|{\rm pr}_{T_\nu}\big(\nabla(f_i)_{(x)}\big)\big|
\longrightarrow|\underline v_{(x)}|=|v_{(x)}|,&
\quad\text{ in the }\nu\text{-a.e.\ sense.}
\end{split}\end{equation}
By applying Proposition \ref{prop:closure_diff}, we deduce from
\eqref{eq:tens_Sob_aux2} and \eqref{eq:tens_Sob_aux4} that
\eqref{eq:tens_Sob_claim} is satisfied, that
\(\nabla_\mu f^{(y)}=v^{(y)}\) for \(\nu\)-a.e.\ \(y\in\R^m\),
and that \(\nabla_\nu f_{(x)}=v_{(x)}\) for \(\mu\)-a.e.\ \(x\in\R^n\).
Consequently, by letting \(i\to\infty\) in \eqref{eq:tens_Sob_aux1}
and using \eqref{eq:tens_Sob_aux3} and \eqref{eq:tens_Sob_aux5}, we
finally conclude that \eqref{eq:tens_Sob_claim1} holds.
\end{proof}
\begin{remark}{\rm
Proposition \ref{prop:tens_t} and Theorem \ref{thm:tens_Sob} are
verified even when \(\mu\) and \(\nu\) are (not necessarily finite)
Radon measures, by taking into account \cite[Proposition 2.6]{Gigli12},
which says that the Sobolev space can be `localised' in a suitable sense.
We omit the details.
\fr}\end{remark}
\def\cprime{$'$} \def\cprime{$'$}

\end{document}